\documentclass{amsart}
\usepackage[english]{babel}
\usepackage[utf8]{inputenc}
\usepackage[T1]{fontenc}
\usepackage{amsmath,amssymb,amsthm}
\usepackage[abbrev]{amsrefs}
\usepackage{multirow}
\usepackage{alltt}
\usepackage{graphicx}
\usepackage{bbm}
\usepackage{mathrsfs}
\usepackage{textcomp}
\usepackage{calc}
\usepackage{fullpage}
\usepackage{mathtools}
\usepackage{enumitem}
\usepackage{comment}
\usepackage{accents}

\numberwithin{equation}{section}

\usepackage[usenames,dvipsnames]{color}
\usepackage[hyperindex,linktocpage]{hyperref}
\hypersetup{colorlinks=true, linkcolor=blue, citecolor=BrickRed}

\newcommand{\R}{\mathbb{R}}
\newcommand{\Q}{\mathbb{Q}}
\newcommand{\N}{\mathbb{N}}

\newcommand{\CC}{\mathscr{C}}
\newcommand{\DD}{\mathscr{D}}

\newcommand{\eps}{\varepsilon}
\newcommand{\dist}{\mathrm{dist}}

\newcommand{\diam}{\text{diam}}
\newcommand{\loc}{\text{loc}}
\newcommand{\pv}{\mathrm{p.v.}\!}
\newcommand{\Ds}{{\left(-\lapl\right)}^s}
\newcommand{\lapl}{\Delta}

\newcommand{\cH}{{\mathcal H}}

\newtheorem{theorem}{Theorem}[section]
\newtheorem{lemma}[theorem]{Lemma}
\newtheorem{proposition}[theorem]{Proposition}

\theoremstyle{remark}

\theoremstyle{definition}

\allowdisplaybreaks

\title[Optimal boundary regularity for mixed local and nonlocal equations]{Optimal boundary regularity\\ for mixed local and nonlocal equations}
\date{\today}

\author{Nicola Abatangelo}
\address{(N. Abatangelo) Dipartimento di Matematica, Alma Mater Studiorum Università di Bologna, P.zza di Porta S. Donato 5, 40126 Bologna, Italy.}
\email{nicola.abatangelo@unibo.it}

\author{Elisa Affili}
\address{(E. Affili)
	Univ Rouen Normandie, CNRS, Normandie Univ, LMRS UMR 6085, F-76000, Rouen,
	France.}
\email{elisa.affili@univ-rouen.fr}

\author{Matteo Cozzi}
\address{(M. Cozzi) Dipartimento di Matematica ``Federigo Enriques'', Università degli Studi di Milano, via Saldini 50, 20133 Milan, Italy.}
\email{matteo.cozzi@unimi.it}



\begin{document}

\begin{abstract}
We provide sharp boundary regularity estimates for solutions to elliptic equations
driven by an integro-differential operator obtained as 
the sum of a Laplacian with a nonlocal operator
generalizing a fractional Laplacian.

Our approach makes use of weighted H\"older spaces 
as well as regularity estimates for the Laplacian in this context
and a fixed-point argument.

We show the optimality of the obtained estimates by means of a counterexample
that we have striven to keep as explicit as possible.
\end{abstract}

\maketitle

\tableofcontents

\section{Introduction}

Let~$n \in \N$ and~$\Omega \subset \R^n$ be a bounded open set with Lipschitz boundary.
We are interested in the boundary regularity of equations of the form
\begin{align}\label{eq}
p \, (- \Delta) u 
+ q \, L_k u 
+ g\cdot Du = f
\qquad\text{in }\Omega,
\end{align}
where~$p,q,f:\Omega\to\R$, $g:\Omega\to\R^n$ are measurable functions 
and~$L_k$ stands for the integro-differential nonlocal operator
\begin{align}\label{Ls}
L_k u(x) \coloneqq \pv\int_{\R^n} \big({u(x)-u(x+z)}\big) k(z) \, dz = \lim_{\varepsilon \rightarrow 0^+}\int_{\R^n \setminus B_\varepsilon} \big({u(x)-u(x+z)}\big) k(z) \, dz
\end{align}
determined by a measurable kernel~$k: \R^n \to \R$ satisfying 
\begin{equation}\label{k}
k(z) = k(-z) \quad \mbox{and} \quad
\frac{\kappa_1}{|z|^{n + 2 s}} \leq k(z) \leq \frac{\kappa_2}{|z|^{n + 2 s}}
\quad \mbox{for a.e.~} z \in \R^n,
\end{equation}
for some exponent~$s \in (0, 1)$ and some constants~$\kappa_2 \ge \kappa_1 > 0$. The presence of the nonlocal term makes a (homogeneous) Dirichlet problem 
associated to~\eqref{eq} look like
\begin{align}\label{dir}
\left\lbrace\begin{aligned}
p \, (- \Delta) u + q \, L_k u + g\cdot Du &= f
&& \text{in }\Omega, \\
u &= 0 && \text{on }\partial\Omega, \\
u &= 0 && \text{in }\R^n\setminus\overline{\Omega}.
\end{aligned}\right.
\end{align}
In fact, the prescription of the values attained by the solution~$u$ in~$\R^n\setminus\overline\Omega$ is of the utmost importance to make sense of the definition of~$L_k u$
in~\eqref{Ls}, whilst the prescribed values on~$\partial\Omega$ are immaterial for~$L_k$
as it does not see negligible sets, being an integral operator. 
Conversely, the boundary conditions on~$\partial\Omega$ 
are somewhat required for the uniqueness of a solution by the term~$-\Delta u$,
which in turn is not affected by the conditions on~$\R^n\setminus\overline\Omega$.
A more detailed analysis of these interactions has been carried out by two of the authors in~\cite{AC21},
although here we are only interested in the homogeneous boundary values as in~\eqref{dir}. 
It is however important to underline that the boundary behavior of a solution to~\eqref{dir}
is strongly affected by the double nature of the left-hand side of~\eqref{eq},
\textit{i.e.}, by the presence of both a Laplacian and a nonlocal operator.

Indeed, the equation in~\eqref{eq} falls into the category of mixed local-nonlocal equations,
which are equations driven by an operator obtained by superposing a local differential operator
with a nonlocal integro-differential one. 
The most canonical example of an operator of this sort is obtained
by considering the kernel~$k$ in~\eqref{Ls} and~\eqref{k} to be given by~$k(z) = |z|^{- n - 2 s}$,
in which case~$L_k$ reduces to (a multiple of) the fractional Laplacian~$\Ds$,
for which we refer to~\cites{MR3967804,libbro,MR3469920,MR2944369,MR3916700}.
In this spirit, one of the simplest examples of mixed local-nonlocal equation, 
which is also covered by the structure of~\eqref{dir}, is given by
\begin{align}\label{eq-simple}
\left\lbrace\begin{aligned}
-\Delta u+\Ds u &= f && \text{in }\Omega, \\
u &= 0 && \text{on }\partial\Omega, \\
u &= 0 && \text{in }\R^n\setminus\overline{\Omega}.,
\end{aligned}\right.
\end{align}
which has lately received a great deal of attention.
Without claiming to be exhaustive, we give below a brief account of the known regularity results
on~\eqref{eq-simple} which are more closely related to our scopes. 

From the broad results by Barles, Chasseigne, and Imbert \cite{MR2735074},
which are actually operating in a fully nonlinear setting,
it is possible to deduce global~$C^{0,\alpha}$ regularity of solutions.
Biagi, Dipierro, Vecchi, and Valdinoci \cite{BDVV22} laid out the foundations of the elliptic theory
for~\eqref{eq-simple}, providing existence and uniqueness of weak and classical solutions, maximum principles, interior Sobolev regularity, and Lipschitz regularity up to the boundary: in some more detail, they showed (see~\cite{BDVV22}*{Theorem 1.6}) that the (unique) solution to~\eqref{eq-simple} satisfies the estimate
\[
\|u\|_{C^{0,1}(\overline\Omega)}\leq C\|f\|_{L^\infty(\Omega)}
\qquad\text{if $f$ is smooth enough and $\Omega$ is strictly convex.}
\]
The same authors in~\cite{BDVV23}*{Theorem 2.7} improved the boundary regularity in~$C^{1,1}$ bounded
domains, showing that
\begin{align}\label{c1beta}
\text{if $f\in L^\infty(\Omega)$, then there exists $\beta\in(0,1)$ such that }u\in C^{1,\beta}(\overline\Omega).
\end{align}
Additionally, \cite{BDVV23}*{Theorem B.1} gives that
\begin{align}\label{c2beta}
\text{if $s\in\Big(0,\frac12\Big)$, $\beta\in(0,1-2s)$, and $f\in C^\beta(\overline\Omega)$,
then $u\in C^{2,\beta}(\overline\Omega)$}.
\end{align}
Biswas, Modasiya, and Sen~\cite{BMS23}*{Theorem 1.3} generalized~\eqref{c1beta}
to a larger context, replacing the fractional Laplacian~${(-\Delta)}^s$
with a nonlocal term~$L_k$
(handling a class of operators which strictly contains
the one we are considering with~\eqref{Ls} and~\eqref{k}) 
and allowing for gradient terms in the equation.
In a series of works, Su, Valdinoci, Wei, and Zhang \cites{SVWZ22,SVWZ23,SVWZ25}
looked at some semilinear counterpart of~\eqref{eq-simple},
reaching the following results in the linear case:
by~\cite{SVWZ25}*{Theorem 1.3} one has
\begin{align}\label{c1beta-exp}
\text{if $f\in L^\infty(\Omega)$, then $u\in C^{1,\beta}(\overline\Omega)$ for every $\beta\in(0,\min\{1,2-2s\})$},
\end{align}
while, by~\cite{SVWZ25}*{Theorem 1.6},
\begin{align}\label{c2beta-crit}
\text{if $s\in\Big(0,\frac12\Big)$, $\beta\in(0,1-2s]$, and $f\in C^\beta(\overline\Omega)$,
then $u\in C^{2,\beta}(\overline\Omega)$}.
\end{align}
In this way,~\eqref{c1beta-exp} improves~\eqref{c1beta} by quantifying the regularity exponent~$\beta$
and~\eqref{c2beta-crit} improves~\eqref{c2beta} by including the case~$\beta=1-2s$.

At a first glance, the above results might sound disappointing, especially in the case~$s>1/2$
when they are far from the full Schauder regularity, not even reaching~$C^2$ up to the boundary:
this might look inconsistent with the fact that equations~\eqref{eq} and~\eqref{eq-simple}
are led by the Laplacian to leading order. 
Although this intuition is heuristically correct far from~$\partial\Omega$, 
things start to get more involved at the boundary as an effect of the nonlocal part,
which somehow detects any lack of smoothness of the solutions \textit{across} the boundary
and consequently reproduces this into the equations.
Indeed, solutions cannot be expected to be better than Lipschitz across the boundary,
owing to a Hopf-type behavior.

So, the boundary regularity of solutions to~\eqref{eq} and~\eqref{eq-simple}
is where it is truly possible to witness a deep interaction
between the local and the nonlocal components of the equation,
and even some sort of competition between the two:
on the one hand, the Laplacian is pushing for a regularization of the solutions,
while, on the other hand,~$L_s$ is preventing them from becoming too regular.

\subsection{Main results}

We show in this paper that the above available boundary regularity results 
are not far from being optimal. Our first result is stated as follows.

\begin{theorem} \label{mainlinearexthmm}
Let~$\Omega \subset \R^n$ be a bounded open set with boundary of class~$C^{2, \alpha}$, for some~$\alpha \in (0, 1)$. Let~$k$ be a kernel satisfying~\eqref{k}, for some~$s \in (0, 1)$ and~$\kappa_2 \ge \kappa_1 > 0$. Let~$p, q \in C^\alpha(\overline\Omega)$ be two non-negative functions, with~$p$ satisfying~$\inf_\Omega p > 0$. Let~$f,g \in C^\alpha(\overline\Omega)$.

Then, problem~\eqref{dir} has a unique solution~$u \in C^2(\Omega) \cap C^0(\R^n)$. Moreover,~$u$ has the following global regularity properties:
\begin{itemize}[leftmargin=*]
\item If~$s \in (0,\frac12)$, then~$u \in C^{2,\beta}(\overline\Omega)$ and it satisfies~$\| u \|_{C^{2,\beta}(\overline\Omega)} \le C \| f \|_{C^\alpha(\overline\Omega)}$, with~$\beta=\min\{\alpha,1-2s\}$.
\item If~$s=\frac12$, then~$u \in C^{1,1 - \varepsilon}(\overline\Omega)$ for every~$\varepsilon \in (0, 1)$ and it satisfies~$\| u \|_{C^{1,1 - \varepsilon}(\overline\Omega)} \le C_\varepsilon \| f \|_{C^\alpha(\overline\Omega)}$.
\item If~$s \in (\frac12,1)$, then~$u \in C^{1,2-2s}(\overline\Omega)$ and it satisfies~$\| u \|_{C^{1,2-2s}(\overline\Omega)} \le C \| f \|_{C^\alpha(\overline\Omega)}$.
\end{itemize}
The constant~$C$ depends only on~$n$,~$s$,~$\alpha$,~$\kappa_1$,~$\kappa_2$,~$\Omega$,~$\| p \|_{C^\alpha(\overline\Omega)}$,~$\inf_\Omega p$,~$\| q \|_{C^\alpha(\overline\Omega)}$, and~$\| g \|_{C^\alpha(\overline\Omega)}$, while~$C_\varepsilon$ also depends on~$\varepsilon$.
\end{theorem}

The statement of Theorem~\ref{mainlinearexthmm} improves the existing literature in that we reach the optimal regularity for~$s\in(1/2,1)$, besides the presence of the coefficients~$p$ and~$q$, of the gradient term, and of the kernel~$k$.
Indeed its claim is sharp in the sense that, regardless of how smooth~$p$,~$q$,~$g$, and~$f$ are, there exist solutions which are not smoother than~$C^{2,1 - 2 s}$ (if~$s\in(0,1/2)$) or~$C^{1,2-2s}$ (if~$s\in(1/2,1)$) up to the boundary. We show this through the following statement valid for dimension~$n = 1$
and constant coefficients~$p,q,g$.

\begin{theorem} \label{counterthm}
Let~$n = 1$ and~$s \in (0, 1)$. For every~$k \in \N$, there exist a function~$f_k \in C^k \! \left( \left[ 0, \frac{1}{2} \right] \right)$ and a solution~$u_k \in L^\infty(\R) \cap C^0(\R) \cap C^\infty \! \left( \left( 0, 1 \right) \right)$ of
\[
\left\lbrace\begin{aligned}
- u_k'' + (-\Delta)^s u_k &= f_k && \mbox{in } \Big( 0, \frac{1}{2} \Big), \\
u_k &= 0 && \mbox{in } (-\infty, 0],
\end{aligned}\right.
\]
such that
$$
u_k \in \begin{dcases}
C^{2,1 - 2 s} \! \left( \left[ 0, \frac{1}{2} \right] \right) \setminus \bigcup_{\varepsilon > 0} C^{2,1 - 2 s + \varepsilon} \! \left( \left[ 0, \frac{1}{2} \right] \right) & \quad \mbox{if } s \in\Big( 0, \frac{1}{2}\Big), \\
\bigcap_{\varepsilon > 0} C^{1,1 - \varepsilon} \! \left( \left[ 0, \frac{1}{2} \right] \right) \setminus C^{1,1} \! \left( \left[ 0, \frac{1}{2} \right] \right) & \quad \mbox{if } s = \frac{1}{2}, \\
C^{1,2 - 2 s} \! \left( \left[ 0, \frac{1}{2} \right] \right) \setminus \bigcup_{\varepsilon > 0} C^{1,2 - 2 s + \varepsilon} \! \left( \left[ 0, \frac{1}{2} \right] \right) & \quad \mbox{if } s \in\Big( \frac{1}{2} , 1 \Big). \\
\end{dcases}
$$
\end{theorem}

Our analysis relies, especially when~$s \in \left( \frac{1}{2}, 1 \right)$, on regularity estimates for the classical Poisson equation in a variant of H\"older spaces, to wit, with weights which are powers of the distance to the boundary.
These spaces are introduced in Section~\ref{sec:wei-holder} and the associated regularity theory for the Poisson equation is presented in Section~\ref{sec:reg-poisson}. 
Section~\ref{sec:maxprinc} derives some maximum principles for~\eqref{eq} which are pivotal for the construction of suitable barriers in the subsequent analysis. 
Section~\ref{sec:main} contains the proof of Theorem~\ref{mainlinearexthmm}
split into intermediate statements and claims: among these, we prove also the following result, valid for~$s \in \left( \frac{1}{2}, 1 \right)$.

\begin{proposition} 
Let~$\Omega \subset \R^n$ be a bounded open set with boundary of class~$C^{2, \alpha}$, for some~$\alpha \in (0, 1)$. Let~$k$ be a kernel satisfying~\eqref{k}, for some~$s \in \left( \frac{1}{2}, 1 \right)$ and~$\kappa_2 \ge \kappa_1 > 0$. Let~$p, q \in C^\alpha(\overline\Omega)$ be two non-negative functions, with~$p$ satisfying~$\inf_\Omega p > 0$. Let~$f,g \in C^\alpha(\overline\Omega)$.

Then, problem~\eqref{dir} has a unique solution~$u \in C^2(\Omega) \cap C^0(\R^n)$. Moreover,~$u$ satisfies
\begin{gather*}
\sup_{x \in \Omega} \Big( {d_x^{2s-1} \, \big|D^2 u(x)\big|} \Big)
+\sup_{\substack{x, y \in \Omega\\ x \ne y}} \left( d_{x, y}^{\beta + 2s-1} \, \frac{|D^2 u(x) - D^2 u(y)|}{|x - y|^\beta} \right)
\le C \| f \|_{C^\alpha(\overline\Omega)} \\
d_x \coloneqq \dist(x,\partial\Omega),\qquad d_{x,y} \coloneqq \min\{d_x,d_y\},
\end{gather*}
with~$\beta \coloneqq \min \big\{ {2 - 2 s, \alpha} \big\}$ and for some constant~$C > 0$ depending only on~$n$,~$s$,~$\alpha$,~$\kappa_1$,~$\kappa_2$,~$\Omega$,~$\| p \|_{C^\alpha(\overline\Omega)}$, $\inf_\Omega p$, $\| q \|_{C^\alpha(\overline\Omega)}$, and~$\| g \|_{C^\alpha(\overline\Omega)}$.
\end{proposition}

The above result is saying that the solution~$u$ actually belongs to a specific weighted~$C^2$ space, showing that the second derivatives of~$u$ might blow-up like the distance to the boundary raised to~$1 - 2 s$. 
We can interpret this as a quantitative and more precise information than that given by Theorem~\ref{mainlinearexthmm}.
A similar phenomenon holds also when~$s = \frac{1}{2}$, in which case~$|D^2 u|$ blows up slower than any negative power. 

The paper is concluded with Section~\ref{sec:counter}. It contains the proof of Theorem~\ref{counterthm}, which is based on somewhat long and delicate computations.

\section{Functional framework: weighted H\"older spaces}\label{sec:wei-holder}

We collect in this section several tools that will be needed for the proofs of the main results of the paper. We begin by introducing a few functional spaces that will be needed to handle our solutions and the corresponding right-hand sides in the case~$s \in \left[ \frac{1}{2}, 1 \right)$.

Let~$\Omega \subset \R^n$ be a bounded open set with Lipschitz boundary. Denote by~$d_x = d(x)$ the distance of a point~$x \in \Omega$ from the boundary of~$\Omega$ and write~$d_{x, y} \coloneqq \min \{ d_x, d_y \}$ for every~$x, y \in \Omega$.

For~$\beta \in (0, 1)$, we consider the weighted spaces
\begin{align*}
C^0_{-1}(\Omega) & \coloneqq \Big\{ {u \in C^0(\Omega) : \| u \|_{C^0_{-1}(\Omega)} < +\infty} \Big\}, \\
C^1_{-1}(\Omega) & \coloneqq \Big\{ {u \in C^0(\Omega) : \| u \|_{C^1_{-1}(\Omega)} < +\infty} \Big\}, \\
C^2_{-1}(\Omega) & \coloneqq \Big\{ {u \in C^2(\Omega) : \| u \|_{C^2_{-1}(\Omega)} < +\infty} \Big\}, \\
C^{2, \beta}_{-1}(\Omega) & \coloneqq \Big\{ {u \in C^{2, \beta}_\loc(\Omega) : \| u \|_{C^{2, \beta}_{-1}(\Omega)} < +\infty} \Big\},
\end{align*}
respectively endowed with the norms
\begin{align*}
\| u \|_{C^0_{-1}(\Omega)} & \coloneqq \sup_{x \in \Omega} \Big( {d_x^{-1} |u(x)|} \Big), \\
\| u \|_{C^1_{-1}(\Omega)} & \coloneqq \| u \|_{C^0_{-1}(\Omega)} + \| Du \|_{L^\infty(\Omega)}, \\
\| u \|_{C^2_{-1}(\Omega)} & \coloneqq \| u \|_{C^1_{-1}(\Omega)}+ \sup_{x \in \Omega} \Big( {d_x |D^2 u(x)|} \Big), \\
\| u \|_{C^{2, \beta}_{-1}(\Omega)} & \coloneqq \| u \|_{C^2_{-1}(\Omega)} + \sup_{\substack{x, y \in \Omega\\ x \ne y}} \left( d_{x, y}^{1 + \beta} \frac{|D^2 u(x) - D^2 u(y)|}{|x - y|^\beta} \right).
\end{align*}

We will also need the following spaces, depending on an additional parameter~$\gamma \in (0, 1)$:
\begin{align*}
C_\gamma^\beta(\Omega) & \coloneqq \left\{ f \in C^\beta_\loc(\Omega) : \| f \|_{C^\beta_\gamma(\Omega)} < +\infty \right\}, \\
C_{-1,\gamma}^2(\Omega) & \coloneqq \left\{ u \in C^{2}(\Omega) : \| u \|_{C_{-1,\gamma}^2(\Omega)} < +\infty \right\}, \\
C_{-1,\gamma}^{2,\beta}(\Omega) & \coloneqq \left\{ u \in C^{2, \beta}_\loc(\Omega) : \| u \|_{C_{-1,\gamma}^{2,\beta}(\Omega)} < +\infty \right\},
\end{align*}
with
\begin{align*}
\| f \|_{C^\beta_\gamma(\Omega)} & \coloneqq \sup_{x \in \Omega} \Big( {d_x^\gamma \, |f(x)|} \Big) + \sup_{\substack{x, y \in \Omega\\ x \ne y}} \left( d_{x, y}^{\beta + \gamma} \, \frac{|f(x) - f(y)|}{|x - y|^\beta} \right), \\
\| u \|_{C_{-1,\gamma}^2(\Omega)} & \coloneqq \| u \|_{C^0_{-1}(\Omega)} + \| Du \|_{L^\infty(\Omega)} + \sup_{x \in \Omega} \Big( {d_x^\gamma \, |D^2 u(x)|} \Big), \\
\| u \|_{C_{-1,\gamma}^{2,\beta}(\Omega)} & \coloneqq \| u \|_{C_{-1,\gamma}^2(\Omega)} + \sup_{\substack{x, y \in \Omega\\ x \ne y}} \left( d_{x, y}^{\beta + \gamma} \, \frac{|D^2 u(x) - D^2 u(y)|}{|x - y|^\beta} \right).
\end{align*}

We briefly explain the above notational choices. Spaces denoted with just one subscript contain functions which behave like the distance to the boundary raised to minus the power represented by said subscript and whose higher order derivatives (gradients, Hessians, and/or the corresponding H\"older seminorms) scale in a homogeneous fashion. Conversely, if two subscripts are used, the first one prescribes the behavior of the function (with its gradient following the homogeneous scaling), while the second one is attached to the blow-up rate of the Hessian (and, possibly, of its H\"older seminorm).
Moreover, the space~$C_{\gamma}^\beta$ will be typically used for the source term~$f$ of an elliptic equation, while the other spaces for the solution~$u$.

We begin with the following interpolation inequality for functions in the class~$C^{2, \beta}_{-1}(\Omega)$.

\begin{lemma} \label{interpollem}
Let~$\beta \in (0, 1)$ and~$\Omega \subset \R^n$ be a bounded open set. Then,
\begin{equation} \label{interpolineq}
\| u \|_{C^2_{-1}(\Omega)} \le C \| u \|_{C^0_{-1}(\Omega)}^{\frac{\beta}{2 (1 + \beta)}} \| u \|_{C^{2, \beta}_{-1}(\Omega)}^{\frac{2 + \beta}{2 (1 + \beta)}} \quad \mbox{for every } u \in C^{2, \beta}_{-1}(\Omega),
\end{equation}
for some constant~$C > 0$ depending only on~$n$ and~$\beta$.
\end{lemma}
\begin{proof}
We adapt the argument of~\cite{L95}*{Proposition~1.1.2}. First, we claim that
\begin{equation} \label{gradinterpol}
|Du(x)| \le C \| u \|_{C^0_{-1}(\Omega)}^{\frac{1}{2}} \| u \|_{C^2_{-1}(\Omega)}^{\frac{1}{2}} \quad \mbox{for every } x \in \Omega \mbox{ and } u \in C^2_{-1}(\Omega),
\end{equation}
for some dimensional constant~$C > 0$. Let~$x \in \Omega$ and assume that~$u \neq 0$ in~$\Omega$ and~$Du(x) \ne 0$, as otherwise there is nothing to prove. For~$z \in B_{\frac{d_x}{2}}$, we have
\begin{align*}
|D u(x) \cdot z| & \le \left| u(x + z) - u(x) - Du(x) \cdot z \right| + |u(x + z)| + |u(x)| \\
& \le d_x^{-1} |z|^2 \sup_{y \in \Omega} \Big( {d_y |D^2 u(y)|} \Big) + 3 d_x \sup_{y \in \Omega} \Big( d_y^{-1} |u(y)| \Big).
\end{align*}
Taking~$z \coloneqq \ell \, d_x \frac{Du(x)}{|Du(x)|}$, we get that
$$
|Du(x)| \le 3 \left\{ \ell \, \| u \|_{C^2_{-1}(\Omega)} + \ell^{-1} \| u \|_{C^0_{-1}(\Omega)} \right\} \quad \mbox{for every } \ell \in \left( 0, \frac{1}{2} \right).
$$
Claim~\eqref{gradinterpol} immediately follows after an optimization in~$\ell$, that is, by choosing, \textit{e.g.},
$$
\ell \coloneqq \frac14 \sqrt{ \frac{\| u \|_{C^0_{-1}(\Omega)}}{\| u \|_{C^2_{-1}(\Omega)}}}.
$$

Next, we claim that
\begin{equation} \label{hessinterpol}
d_x |D^2 u(x)| \le C \| u \|_{C^0_{-1}(\Omega)}^{\frac{\beta}{2 (1 + \beta)}} \| u \|_{C^{2, \beta}_{-1}(\Omega)}^{\frac{2 + \beta}{2 (1 + \beta)}} \quad \mbox{for every } x \in \Omega \mbox{ and } u \in C^{2, \beta}_{-1}(\Omega),
\end{equation}
for some constant~$C > 0$ depending only on~$n$ and~$\beta$. The verification of this fact is similar to the previous one. Indeed, for~$j \in \{1, \ldots, n\}$ and~$z = \ell \, d_x \frac{D D_j u(x)}{|D D_j u(x)|}$ with~$\ell \in \left( 0, \frac{1}{2} \right)$, we compute
\begin{align*}
d_x |D D_j u(x)| & = \ell^{-1} |D D_j u(x) \cdot z| \\
& \le \ell^{-1} \Big( {\big| {D_j u(x + z) - D_j u(x) - D D_j u(x) \cdot z} \big| + |D_j u(x + z)| + |D_j u(x)|} \Big) \\
& \le 2 \, \ell^{-1} \Bigg\{ {d_{x,x+z}^{- 1 - \beta} |z|^{1 + \beta} \sup_{\substack{y, w \in \Omega\\y \ne w}} \left( {d_{y, w}^{1 + \beta} \frac{|D^2 u(y) - D^2 u(w)|}{|y - w|^\beta}} \right) + \| Du \|_{L^\infty(\Omega)}} \Bigg\} \\
& \le 2^{2+\beta} \Bigg\{ {\ell^{\beta} \| u \|_{C^{2, \beta}_{-1}(\Omega)} + \ell^{-1} \| Du \|_{L^\infty(\Omega)}} \Bigg\},
\end{align*}
where we have used that $d_{x,x+z}\geq d_x-|z|=(1-\ell)d_x>\frac12d_x$.
Estimate~\eqref{hessinterpol} follows by taking
$$
\ell = \frac{1}{4} \left\{ \frac{\| Du \|_{L^\infty(\Omega)}}{\| u \|_{C^{2, \beta}_{-1}(\Omega)}} \right\}^{\! \frac{1}{1 + \beta}}
$$
and applying~\eqref{gradinterpol}.

By combining~\eqref{gradinterpol} and~\eqref{hessinterpol}, we easily obtain inequality~\eqref{interpolineq}. The proof is thus complete.
\end{proof}

Next, we show that the gradients of the functions in~$C_{-1,\gamma}^{2}(\Omega)$ are actually H\"older continuous up to the boundary of~$\Omega$.

\begin{lemma} \label{C2alphaisunweightedlem}
Let~$\gamma \in (0, 1)$,~$\Omega \subset \R^n$ be a bounded open set with Lipschitz boundary, and~$u \in C_{-1,\gamma}^{2}(\Omega)$. Then,~$u \in C^{1, 1 - \gamma}(\overline\Omega)$ with
$$
\| u \|_{C^{1, 1 - \gamma}(\overline\Omega)} \le C \| u \|_{C^{2}_{-1,\gamma}(\Omega)},
$$
for some constant~$C >0$ depending only on~$n$,~$\gamma$, and~$\Omega$.
\end{lemma}
\begin{proof}
As~$\| u \|_{C^1(\Omega)} \le \big( {1 + \diam(\Omega)} \big) \| u \|_{C_{-1,\gamma}^{2}(\Omega)}$, we only need to verify that
\begin{equation} \label{Holderestclaim}
|Du(x) - Du(y)| \le C \| u \|_{C^{2}_{-1,\gamma}(\Omega)} |x - y|^{1 - \gamma} \quad \mbox{for all } x, y \in \Omega,
\end{equation}
for some~$C >0$ depending only on~$n$,~$\gamma$, and~$\Omega$.

To this aim, observe that, if~$|x - y| < \frac{d_x}{2}$, then
\begin{align*}
|Du(x) - Du(y)| & \le |x - y| \int_0^1 \big| {D^2 u(t x + (1 - t) y)} \big| \, dt \\
& \le 2 \, d_x^{-\gamma} \, |x - y| \sup_{z \in B_{d_x / 2}(x)} \Big( {d_z^{\gamma} |D^2 u(z)|} \Big) \le 2 \| u \|_{C_{-1,\gamma}^{2}(\Omega)} |x - y|^{1 - \gamma}.
\end{align*}
Thus, we proved that
\begin{equation} \label{Holderestclosepoints}
|Du(x) - Du(y)| \le 2 \| u \|_{C^{2}_{-1,\gamma}(\Omega)} |x - y|^{1 - \gamma} \quad \mbox{for all } x, y \in \Omega \mbox{ such that } |x - y| < \frac{d_x}{2}.
\end{equation}
Thanks to this, the interior differentiability of~$Du$, and a simple covering argument, it is not hard to see that~\eqref{Holderestclaim} will be established if we show that, for every~$\bar{x} \in \partial \Omega$ there exist a radius~$R_{\bar{x}} > 0$ and a constant~$C_{\bar{x}} > 0$, both depending at most on~$n$,~$\gamma$,~$\Omega$, and~$\bar{x}$, such that
\begin{equation} \label{Holderestclaiminball}
|Du(x) - Du(y)| \le C_{\bar{x}} \| u \|_{C^{2}_{-1,\gamma}(\Omega)} |x - y|^{1 - \gamma} \quad \mbox{for all } x, y \in \Omega \cap B_{R_{\bar{x}}}(\bar{x}) \mbox{ such that } |x - y| \ge \frac{d_x}{2}.
\end{equation}

Claim~\eqref{Holderestclaiminball} can be established by arguing as in the proof of~\cite{RS14}*{Proposition~1.1}. We reproduce here the details for the convenience of the reader.

Let~$\bar{x} \in \partial \Omega$. As~$\partial \Omega$ is Lipschitz, there exist a radius~$R > 0$ and a Lipschitz diffeomorphism~$\Psi: \R^n \to \R^n$, with inverse~$\Phi \coloneqq \Psi^{-1}$, such that~$\Psi(0) = \bar{x}$,~$d_{\Psi(z)} \le z_n \le K d_{\Psi(z)}$ for all~$z \in B_{K^2 R} \cap \R^n_+$, where~$\R^n_+ \coloneqq \R^{n - 1} \times (0, +\infty)$, and
\begin{gather*}
B_{K^4 R}(\bar{x}) \cap \Omega \subset \Psi \Big( {B_{K^5 R} \cap \R^n_+} \Big) \subset B_{K^6 R}(\bar{x}) \cap \Omega, \\
B_{K^4 R}(\bar{x}) \cap \partial \Omega \subset \Psi \Big( {B_{K^5 R} \cap \partial \R^n_+} \Big) \subset B_{K^6 R}(\bar{x}) \cap \partial \Omega,
\end{gather*} 
with~$K \coloneqq 1 + \| D\Psi \|_{L^\infty(\R^n)} + \| D\Phi \|_{L^\infty(\R^n)} \ge 3$---see the proof of the forthcoming Theorem~\ref{aprioriC2alphaestlem} for the construction of~$\Psi$. Set~$\mathbf{v}(z) \coloneqq \left( Du \circ \Psi \right)(z)$ for~$z \in B_{K^5 R} \cap \R^n_+$. Then, estimate~\eqref{Holderestclosepoints} yields that
\begin{equation} \label{Holderforvbold}
\left| \mathbf{v}(z) - \mathbf{v}(w) \right| \le 2 K \| u \|_{C_{-1,\gamma}^{2}(\Omega)} |z - w|^{1 - \gamma} \quad \mbox{for all } z, w \in B_{K^2 R} \cap \R^n_+ \mbox{ such that } |z - w| < K^{-3} z_n.
\end{equation}

Let now instead~$z, w \in B_{K R} \cap \R^n_+$ be such that~$|z - w| \ge K^{-3} z_n$. Consider the auxiliary points~$\bar{z} \coloneqq z + |z - w| e_n$ and~$\bar{w} \coloneqq w + |z - w| e_n$ in~$B_{K^2 R} \cap \R^n_+$, as well as the sequences~$\{ z^{(j)} \}, \{ w^{(j)} \} \subset B_{K^2 R} \cap \R^n_+$ defined by~$z^{(j)} \coloneqq z + \frac{|z - w|}{(1 + K^{-3})^{j - 1}} \, e_n$ and~$w^{(j)} \coloneqq w + \frac{|z - w|}{(1 + K^{-3})^{j - 1}} \, e_n$ for~$j \in \N$. Clearly~$z^{(1)} = \bar{z}$,~$w^{(1)} = \bar{w}$ and~$z^{(j)} \rightarrow z$,~$w^{(j)} \rightarrow w$ as~$j \rightarrow +\infty$. Moreover,~$|z^{(j + 1)} - z^{(j)}| \le K^{-3} z_n^{(j + 1)}$ and~$|w^{(j + 1)} - w^{(j)}| \le K^{-3} w_n^{(j + 1)}$ for every~$j \in \N$. As a result, we may apply~\eqref{Holderforvbold} to deduce that
\begin{align*}
|\mathbf{v}(z) - \mathbf{v}(\bar{z})| & \le \sum_{j \in \N} \big| {\mathbf{v}(z^{(j + 1)}) - \mathbf{v}(z^{(j)})} \big| \le 2 K \| u \|_{C_{-1,\gamma}^{2}(\Omega)} \sum_{j \in \N} \big| {z^{(j + 1)} - z^{(j)}} \big|^{1 - \gamma} \\
& \le 2 K^{1 - 3 (1 - \gamma)} \| u \|_{C_{-1,\gamma}^{2}(\Omega)} |z - w|^{1 - \gamma} \sum_{j \in \N} \big( {1 + K^{-3}} \big)^{\! -(1 - \gamma)j} \le C_{K, \gamma} \| u \|_{C_{-1,\gamma}^{2}(\Omega)} |z - w|^{1 - \gamma},
\end{align*}
and, similarly, that~$|\mathbf{v}(w) - \mathbf{v}(\bar{w})| \le C_{K, \gamma} \| u \|_{C_{-1,\gamma}^{2}(\Omega)} |z - w|^{1 - \gamma}$, for some constant~$C_{K, \gamma} > 0$ depending only on~$K$ and~$\gamma$. 

Since~$\bar{z}_n, \bar{w}_n \ge |z - w|$ and~$|\bar{z} - \bar{w}| = |z - w|$, we can also obtain that~$|\mathbf{v}(\bar{z}) - \mathbf{v}(\bar{w})| \le C_{K, \gamma} \| u \|_{C_{-1,\gamma}^{2}(\Omega)} |z - w|^{1 - \gamma}$, for some possibly larger~$C_{K, \gamma}$, by dividing the segment joining~$\bar{z}$ and~$\bar{w}$ into at most~$\lceil K^3   \rceil + 1$ equally spaced subsegments, applying~\eqref{Holderforvbold} between each two consecutive endpoints, and adding up the result.

All in all, we proved that~$|\mathbf{v}(z) - \mathbf{v}(w)| \le C_{K, \gamma} \| u \|_{C_{-1,\gamma}^{2}(\Omega)} |z - w|^{1 - \gamma}$ for every~$z, w \in B_{K R} \cap \R^n_+$ such that~$|z - w| \ge K^{-3 } z_n$. By going back to~$u$ and the variables~$x, y \in \Omega$, we easily arrive at~\eqref{Holderestclaiminball} with~$R_{\bar{x}} = R$. The proof of the lemma is thus finished.
\end{proof}

Note that the norm~$\| \cdot \|_{C^{2}_{-1,\gamma}(\Omega)}$ is (strictly) stronger than~$\| \cdot \|_{C^2_{-1}(\Omega)}$, since the latter allows the Hessian to blow-up at the boundary at a faster rate. Consequently, the space~$C^{2}_{-1,\gamma}(\Omega)$ is continuously embedded in~$C^2_{-1}(\Omega)$. The next lemma shows that, by requiring some H\"older continuity on the second derivatives---\textit{i.e.}, by considering~$C^{2, \beta}_{-1,\gamma}(\Omega)$ in place of~$C^{2}_{-1,\gamma}(\Omega)$---, the embedding is also compact.

\begin{lemma} \label{compactembedlem}
Let~$\beta, \gamma \in (0, 1)$ and~$\Omega \subset \R^n$ be a bounded open set with Lipschitz boundary. Then,~$C^{2, \beta}_{-1,\gamma}(\Omega)$ is compactly embedded in~$C^2_{-1}(\Omega)$.
\end{lemma}
\begin{proof}
Let~$\{ u_k \}_{k\in\N} \subset C^{2, \beta}_{-1,\gamma}(\Omega)$ be a bounded sequence, \textit{i.e.}, without loss of generality, such that
\begin{equation} \label{ukbounded}
\| u_k \|_{C^{2, \beta}_{-1,\gamma}(\Omega)} \le 1 \quad \mbox{for every } k \in \N.
\end{equation}
Since, by Lemma~\ref{C2alphaisunweightedlem}, the space~$C_{-1,\gamma}^{2}(\Omega)$ (and thus the smaller~$C^{2, \beta}_{-1,\gamma}(\Omega)$) is continuously embedded in~$C^{1, 1- \gamma}(\overline\Omega)$, standard compact embedding theorems for (unweighted) H\"older spaces and a diagonal argument yield that, up to a subsequence,~$u_k$ converges in~$C^{1}(\overline{\Omega}) \cap C^{2}_\loc(\Omega)$ to a function~$u \in C^{1, 1- \gamma}(\overline\Omega) \cap C^{2, \beta}_\loc(\Omega)$ satisfying
\begin{equation} \label{ubounded}
\| u \|_{C^{2, \beta}_{-1,\gamma}(\Omega)} \le 1.
\end{equation}

As each~$u_k$ can be continuously extended to a function defined on~$\overline{\Omega}$  vanishing on its boundary, so does~$u$, and we easily deduce that~$\sup_{x \in \Omega} \Big( {d_x^{-1} |u_k(x) - u(x)|} \Big) \le \| D u_k - D u \|_{L^\infty(\Omega)} \rightarrow 0$ as~$k \rightarrow +\infty$. We are left to prove that
\begin{equation} \label{weightedC2convofuktou}
\lim_{k \rightarrow +\infty} \sup_{x \in \Omega} \Big( d_x |D^2 u_k(x) - D^2 u(x)| \Big) = 0.
\end{equation}
To see it, let~$\varepsilon > 0$ be fixed and observe that, by the local~$C^2$ convergence, there exists~$N_\varepsilon \in \N$ such that~$\| D^2 u_k - D^2 u \|_{L^\infty(\Omega_{\delta_\varepsilon})} \le \diam(\Omega)^{-1} \varepsilon$ for every~$k \in \N$ with~$k \ge N_\varepsilon$, where~$\delta_\varepsilon \coloneqq \varepsilon^{\frac{1}{1 - \gamma}} / 2$. Here, we adopted the notation~$\Omega_\delta \coloneqq \big\{ {x \in \Omega : d_x > \delta} \big\}$, for~$\delta > 0$. Consequently,
$$
\sup_{x \in \Omega_{\delta_\varepsilon}} \Big( d_x |D^2 u_k(x) - D^2 u(x)| \Big) \le \diam(\Omega) \| D^2 u_k - D^2 u \|_{L^\infty(\Omega_{\delta_\varepsilon})} \le \varepsilon \quad \mbox{for every } k \in \N \mbox{ with } k \ge N_\varepsilon.
$$
On the other hand, recalling~\eqref{ukbounded} and~\eqref{ubounded}, we obtain
$$
\sup_{x \in \Omega \setminus \Omega_{\delta_\varepsilon}} \Big( d_x |D^2 u_k(x) - D^2 u(x)| \Big) \le \delta_\varepsilon^{1 - \gamma} \left\{ \sup_{x \in \Omega} \Big( d_x^\gamma |D^2 u_k(x)| \Big) + \sup_{x \in \Omega} \Big( d_x^\gamma |D^2 u(x)| \Big) \right\} \le 2 \delta_\varepsilon^{1 - \gamma} = \varepsilon,
$$
for every~$k \in \N$. These two facts immediately lead us to~\eqref{weightedC2convofuktou} and, with it, to the conclusion of the proof.
\end{proof}

\section{Regularity for the Poisson equation in weighted H\"older spaces}\label{sec:reg-poisson}

The aim of this section is to study the solvability of the Dirichlet problem for the Laplacian with zero boundary datum and with a right-hand side that blows up at a strictly slower rate than the inverse of the distance to the boundary. In particular, we are interested in describing the boundary behavior of the solution and of its derivatives through the spaces~$C^{2, \beta}_{-1,\gamma}$. Our statement is as follows.

\begin{theorem} \label{mainC2alphaexistprop}
Let~$\beta, \gamma \in (0, 1)$,~$\Omega \subset \R^n$ be a bounded open set with boundary of class~$C^{2, \beta}$, and~$f \in C^{\beta}_\gamma(\Omega)$. Then, there exists a unique solution~$u \in C^2(\Omega) \cap C^0(\overline{\Omega})$ of
\begin{equation} \label{localDirprobforu}
\left\lbrace\begin{aligned}
- \Delta u &= f && \quad \mbox{in } \Omega, \\
u &= 0 && \quad \mbox{on } \partial \Omega.
\end{aligned}\right.
\end{equation}
Moreover,~$u \in C^{2, \beta}_{-1,\gamma}(\Omega)$ and it satisfies
\begin{equation} \label{mainC2alphaest}
\| u \|_{C^{2, \beta}_{-1,\gamma}(\Omega)} \le C \| f \|_{C^{\beta}_{\gamma}(\Omega)},
\end{equation}
for some constant~$C > 0$ depending only on~$n$,~$\beta$,~$\gamma$, and~$\Omega$.
\end{theorem}

Due to its rather classical flavor, Theorem~\ref{mainC2alphaexistprop} is presumably stated somewhere in the literature and probably known to the expert reader. However, since we could not find an exact reference, we provide all the details of its proof---which takes great inspiration from~\cite{GT01}*{Sections~4 and~6}.

To begin with, observe that, when~$n = 1$, the equation~$- \Delta u = f$ is an ODE and the claims of Theorem~\ref{mainC2alphaexistprop} can be readily established by integration. In the following, we thus restrict to the case~$n \ge 2$. In addition, we assume without loss of generality~$\Omega$ to actually be a domain, \textit{i.e.}, a connected open set.

Let~$\Gamma$ be the fundamental solution for the (minus) Laplacian, \textit{i.e.},
$$
\Gamma(z) \coloneqq \begin{dcases}
- \frac{1}{2\pi} \log |z| & \quad \mbox{if } n = 2, \\
\frac{1}{n (n - 2) |B_1|} \frac{1}{|z|^{n - 2}} & \quad \mbox{if } n \ge 3.
\end{dcases}
$$
For~$r > 0$, we write~$\CC_r \coloneqq B_r' \times (-r, r)$ and~$\CC_r^+ \coloneqq B_r' \times (0, r)$,
see Figure~\ref{fig:cilindretto}.
We also set~$\CC_r(x) \coloneqq x + \CC_r$ and~$\DD_r \coloneqq B_r' \times \{ 0 \}$. 
Given a function~$f: \CC_{r}^+ \to \R$, we denote with~$f^\circ$ its odd reflection with respect to the horizontal hyperplane, \textit{i.e.}, the function defined at a.e.~$x = (x', x_n) \in \CC_{r}$ by
$$
f^\circ(x) \coloneqq \begin{cases}
f(x', x_n) & \quad \mbox{if } x_n > 0, \\
- f(x', - x_n) & \quad \mbox{if } x_n < 0.
\end{cases}
$$

\begin{figure}
\centering
\includegraphics[scale=1]{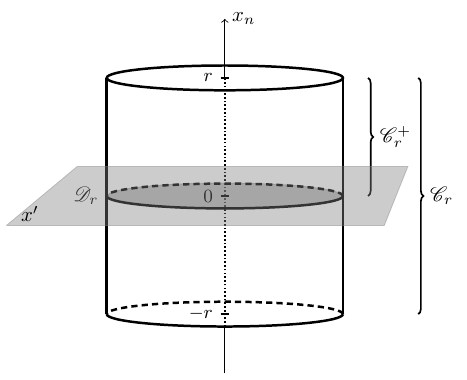}
\caption{A depiction of the notations used in Lemma~\ref{Newtonpotlem} and
Proposition~\ref{C2alphacyllem}.}\label{fig:cilindretto}
\end{figure}

In the following lemma, we explore the regularity properties of the Newtonian potential of~$f^\circ$ near the part of the boundary of~$\CC_{2 r}^+$ constituted by the disk~$\DD_r$. In order to do this, we need some more notation. 

For~$\beta, \gamma \in (0, 1)$, we define the spaces
\begin{align*}
\widetilde{C}_\gamma^\beta(\CC_r^+) & \coloneqq \left\{ f \in C^\beta_\loc(\CC_r^+) : \| f \|_{\widetilde{C}^\beta_\gamma(\CC_r^+)} < +\infty \right\}, \\
\widetilde{C}_{-1,\gamma}^{2, \beta}(\CC_r^+) & \coloneqq \left\{ w \in C^{2, \beta}_\loc(\CC_r^+) : \| w \|_{\widetilde{C}_{-1,\gamma}^{2, \beta}(\CC_r^+)} < +\infty \right\},
\end{align*}
respectively endowed with the norms
\begin{align*}
\| f \|_{\widetilde{C}^\beta_\gamma(\CC_r^+)} & \coloneqq \sup_{x \in \CC_{r}^+} \Big( {x_n^\gamma \, |f(x)|} \Big) + \sup_{\substack{x, y \in \CC_{r}^+\\ x \ne y}} \left( \min \{ x_n, y_n \}^{\beta + \gamma} \, \frac{|f(x) - f(y)|}{|x - y|^\beta} \right), \\
\| w \|_{\widetilde{C}_{-1,\gamma}^{2, \beta}(\CC_r^+)} & \coloneqq \sup_{x \in \CC_r^+} \Big( {x_n^{-1} |w(x)|} \Big) + \| Dw \|_{L^\infty(\CC_r^+)} + \| D^2 w \|_{\widetilde{C}^\beta_\gamma(\CC_r^+)}.
\end{align*}
Note that each function in~$\widetilde{C}_{-1,\gamma}^{2, \beta}(\CC_r^+)$ can be uniquely extended to a function in~$C^{0, 1}(\CC_r^+ \cup \DD_r)$, vanishing on~$\DD_r$. We then have the following result.

\begin{lemma} \label{Newtonpotlem}
Let~$\beta, \gamma \in (0, 1)$ and~$f \in \widetilde{C}^\beta_\gamma(\CC_{2}^+)$. Then, the function
\[
w(x) \coloneqq \int_{\CC_{2}} \Gamma(x - y) f^\circ(y) \, dy \qquad \mbox{for } x \in \CC_2^+ \cup \DD_2,
\]
satisfies~$- \Delta w = f$ in~$\CC_2^+$,~$w = 0$ on~$\DD_2$, and it belongs to~$\widetilde{C}_{-1,\gamma}^{2, \beta}(\CC_1^+)$, with
\[
\| w \|_{\widetilde{C}^{2, \beta}_{-1,\gamma}(\CC_1^+)} \le C \| f \|_{\widetilde{C}^\beta_\gamma(\CC_{2}^+)},
\]
for some constant~$C > 0$ depending only on~$n$,~$\beta$, and~$\gamma$.
\end{lemma}
\begin{proof}
First of all, it is classical that~$w \in C^{2, \beta}_\loc(\CC_2^+)$ and that it satisfies~$- \Delta w = f$ in~$\CC_2^+$.

The uniform~$C^1$ bounds are rather straightforward. Indeed, for~$n \ge 3$ and~$x \in \CC_1^+ \cup \DD_1$, by changing variables appropriately we have
\begin{equation} \label{wLinftybound}
\begin{aligned}
|w(x)| & \le C \sup_{z \in \CC_{2}^+} \Big( {z_n^\gamma \, |f(z)|} \Big) \int_{- 2}^{2} \Bigg( {\int_{B_{2}'} \frac{dy'}{\big( {|x' - y'|^2 + (x_n - y_n)^2} \big)^{\frac{n - 2}{2}}}} \Bigg) \frac{dy_n}{|y_n|^\gamma} \\
& \le C \| f \|_{\widetilde{C}^\beta_\gamma(\CC_{2}^+)} \int_{- 2}^{2} \left( \int_0^{\frac{3}{|x_n - y_n|}} \frac{t^{n - 2}}{(1 + t)^{n - 2}} \, dt \right) \frac{|x_n - y_n|}{|y_n|^\gamma} \, dy_n \le C \| f \|_{\widetilde{C}^\beta_\gamma(\CC_{2}^+)},
\end{aligned}
\end{equation}
for some constant~$C > 0$ depending only on~$n$ and~$\gamma$. After some tedious computations, one similarly handles the case~$n = 2$ and the estimate for the gradient of~$w$. From this it follows in particular that~$w$ is continuous in the whole~$\CC_1^+ \cup \DD_1$.

The fact that~$w = 0$ on~$\DD_1$ stems from symmetry considerations. From this and the bound for the gradient, we can also improve~\eqref{wLinftybound} to the desired weighted~$L^\infty$ bound.

To estimate the Hessian of~$w$, we take advantage of~\cite{GT01}*{Lemma~4.2} and write
\begin{equation} \label{hessassum}
D_{i j} w(x) = I_1(x) + I_2(x) + I_3(x),
\end{equation}
for~$x \in \CC_1^+$ and~$i, j \in \{ 1, \ldots, n \}$, where
\begin{equation} \label{Ijsdefs}
\begin{aligned}
I_1(x) & \coloneqq \int_{\CC_{2} \setminus K} D_{i j} \Gamma(x - y) f^\circ(y) \, dy, \\
I_2(x) & \coloneqq \int_{K} D_{i j} \Gamma(x - y) \big( {f(y) - f(x)} \big) \, dy, \\
I_3(x) & \coloneqq - f(x) \int_{\partial K} D_i \Gamma(x - y) \nu_j(y) \, d\mathcal{H}^{n - 1}(y),
\end{aligned}
\end{equation}
and~$K \subset \subset \CC_{2}^+$ is any open set with Lipschitz boundary containing~$x$. Choosing~$K = \CC_{\frac{x_n}{2}}(x)$, we have
\begin{equation} \label{I1est}
\begin{aligned}
|I_1(x)| & \le C \| f \|_{\widetilde{C}^\beta_\gamma(\CC_{2}^+)} \left\{ \vphantom{ \int_{\frac{x_n}{2}}^{\frac{3 x_n}{2}} \Bigg( {\int_{B_{2}' \setminus B_{\frac{x_n}{2}}'(x')} \frac{dy'}{\big( {|x' - y'|^2 + (x_n - y_n)^2} \big)^{\frac{n}{2}}}} \Bigg) \frac{dy_n}{|y_n|^\gamma}} \int_{(- 2, 2) \setminus \left[ \frac{x_n}{2}, \frac{3 x_n}{2} \right]} \left( \int_{B_{2}'} \frac{dy'}{\big( {|x' - y'|^2 + (x_n - y_n)^2} \big)^{\! \frac{n}{2}}} \right) \frac{dy_n}{|y_n|^\gamma} \right. \\
& \quad \left. + \int_{\frac{x_n}{2}}^{\frac{3 x_n}{2}} \Bigg( {\int_{B_{2}' \setminus B_{\frac{x_n}{2}}'(x')} \frac{dy'}{\big( {|x' - y'|^2 + (x_n - y_n)^2} \big)^{\! \frac{n}{2}}}} \Bigg) \frac{dy_n}{y_n^\gamma} \right\} \\
& \le C \| f \|_{\widetilde{C}^\beta_\gamma(\CC_{2}^+)} \left\{ \int_{\frac{x_n}{2}}^{3} \left( \int_0^{\frac{3}{\ell}} \frac{t^{n - 2}}{(1 + t)^n} \, dt \right) \frac{d\ell}{\ell |\ell - x_n|^\gamma} + \int_0^{\frac{x_n}{2}} \left( \int_{\frac{x_n}{2 \ell}}^{\frac{3}{\ell}} \frac{t^{n - 2}}{(1 + t)^n} \, dt \right) \frac{d\ell}{\ell^{1 + \gamma}} \right\} \\
& \le C \| f \|_{\widetilde{C}^\beta_\gamma(\CC_{2}^+)} \left\{ x_n^{- \gamma} \int_{\frac{1}{2}}^{+\infty} \frac{dm}{m|m - 1|^\gamma} + x_n^{-1} \int_0^{\frac{x_n}{2}} \frac{d\ell}{\ell^{\gamma}} \right\} \le C \| f \|_{\widetilde{C}^\beta_\gamma(\CC_{2}^+)} x_n^{-\gamma}
\end{aligned}
\end{equation}
and
\begin{align*}
|I_2(x)| & \le C \| f \|_{\widetilde{C}^\beta_\gamma(\CC_{2}^+)} x_n^{- \beta - \gamma} \int_{\frac{x_n}{2}}^{\frac{3 x_n}{2}} \left( \int_{B'_{\frac{x_n}{2}}(x')} \frac{dy'}{\big( {|x' - y'|^2 + (x_n - y_n)^2} \big)^{\! \frac{n - \beta}{2}}} \right) dy_n \\
& \le C \| f \|_{\widetilde{C}^\beta_\gamma(\CC_{2}^+)} x_n^{- \beta - \gamma} \int_{0}^{\frac{x_n}{2}} \left( \int_0^{\frac{x_n}{2 \ell}} \frac{t^{n - 2}}{(1 + t)^{n - \beta}} \, dt \right) \frac{d\ell}{\ell^{1 - \beta}} \le C \| f \|_{\widetilde{C}^\beta_\gamma(\CC_{2}^+)} x_n^{- \gamma}.
\end{align*}
It is immediate to verify that the same bound also holds for~$|I_3(x)|$. Hence, we conclude that
\begin{equation} \label{hessLinftyest}
|D^2 w(x)| \le C \| f \|_{\widetilde{C}^\beta_\gamma(\CC_{2}^+)} x_n^{- \gamma} \quad \mbox{for every } x \in \CC_1^+.
\end{equation}

We now claim that
\begin{equation} \label{hessCalphaest}
\min \{ x_n, \tilde{x}_n \}^{\beta + \gamma} \frac{\big| {D^2 w(x) - D^2 w(\tilde{x})} \big|}{|x - \tilde{x}|^\beta} \le C \| f \|_{\widetilde{C}^\beta_\gamma(\CC_{2}^+)} \quad \mbox{for every } x, \tilde{x} \in \CC_1^+ \mbox{ such that } x \ne \tilde{x}.
\end{equation}
Clearly, it is enough to establish~\eqref{hessCalphaest} for~$x, \tilde{x} \in \CC_1^+$ with~$\tilde{x}_n \ge x_n$. Also, in view of~\eqref{hessLinftyest}, we may assume that~$|x - \tilde{x}| \le \frac{x_n}{4}$. By expressing~$D_{i j} w(x)$ and~$D_{i j} w(\tilde{x})$ via~\eqref{hessassum} and~\eqref{Ijsdefs} with~$K = \CC_{\frac{x_n}{2}}(x)$ in both cases, we get that
\begin{equation} \label{D2incremdeco}
\big| {D_{i j} w(x) - D_{i j} w(\tilde{x})} \big| \le \sum_{i = 1}^7  J_i(x, \tilde{x}),
\end{equation}
where
\begin{align*}
J_1(x, \tilde{x}) & \coloneqq \int_{\CC_{2} \setminus \CC_{\frac{x_n}{2}}(x)} \big| {D_{i j} \Gamma(x - y) - D_{i j} \Gamma(\tilde{x} - y)} \big| |f^\circ(y)| \, dy, \\
J_2(x, \tilde{x}) & \coloneqq \int_{\CC_{\frac{x_n}{2}}(x) \setminus B_{|x - \tilde{x}|} \left( \frac{x + \tilde{x}}{2} \right)} \big| {D_{i j} \Gamma(x - y) - D_{i j} \Gamma(\tilde{x} - y)} \big| \big| {f(y) - f(x)} \big| \, dy, \\
J_3(x, \tilde{x}) & \coloneqq \big| {f(x) - f(\tilde{x})} \big| \left| \int_{\CC_{\frac{x_n}{2}}(x) \setminus B_{|x - \tilde{x}|} \left( \frac{x + \tilde{x}}{2} \right)} {D_{i j} \Gamma(\tilde{x} - y)} \, dy \right|, \\
J_4(x, \tilde{x}) & \coloneqq \int_{B_{|x - \tilde{x}|} \left( \frac{x + \tilde{x}}{2} \right)} \big| {D_{i j} \Gamma(x - y)} \big| \big| {f(y) - f(x)} \big| \, dy, \\
J_5(x, \tilde{x}) & \coloneqq \int_{B_{|x - \tilde{x}|} \left( \frac{x + \tilde{x}}{2} \right)} \big| {D_{i j} \Gamma(\tilde{x} - y)} \big| \big| {f(y) - f(\tilde{x})} \big| \, dy, \\
J_6(x, \tilde{x}) & \coloneqq |f(x) - f(\tilde{x})| \int_{\partial \CC_{\frac{x_n}{2}}(x)} \big| {D_i \Gamma(x - y)} \big| |\nu_j(y)| \, d\mathcal{H}^{n - 1}(y), \\
J_7(x, \tilde{x}) & \coloneqq |f(\tilde{x})| \int_{\partial \CC_{\frac{x_n}{2}}(x)} \big| {D_i \Gamma(x - y) - D_i \Gamma(\tilde{x} - y)} \big| |\nu_j(y)| \, d\mathcal{H}^{n - 1}(y).
\end{align*}
We now estimate each of these terms. Observe that, as~$|x - \tilde{x}| \le \frac{x_n}{4}$, for every~$y \in \CC_{2} \setminus \CC_{\frac{x_n}{2}}(x)$ we have
\begin{equation} \label{D2Gamma-D2Gamma}
\big| {D^2 \Gamma(x - y) - D^2 \Gamma(\tilde{x} - y)} \big| \le C |x - \tilde{x}| \int_0^1 \frac{dt}{|t x + (1 - t) \tilde{x} - y|^{n + 1}} \le C \frac{|x - \tilde{x}|}{|x - y|^{n + 1}}.
\end{equation}
Thus, computing as in~\eqref{I1est}, we easily obtain that
\begin{equation} \label{J1est}
J_1(x, \tilde{x}) \le C \| f \|_{\widetilde{C}^\beta_\gamma(\CC_{2}^+)} |x - \tilde{x}| \int_{\CC_{2} \setminus \CC_{\frac{x_n}{2}}(x)} \frac{dy}{|x - y|^{n + 1} |y_n|^\gamma} \le C \| f \|_{\widetilde{C}^\beta_\gamma(\CC_{2}^+)} x_n^{- 1 - \gamma} |x - \tilde{x}|.
\end{equation}
A calculation similar to~\eqref{D2Gamma-D2Gamma} yields
$$
\big| {D^2 \Gamma(x - y) - D^2 \Gamma(\tilde{x} - y)} \big| \le C \frac{|x - \tilde{x}|}{\left| y - \frac{x + \tilde{x}}{2} \right|^{n + 1}} \quad \mbox{for every } y \in \CC_{\frac{x_n}{2}}(x) \setminus B_{|x - \tilde{x}|} \left( \frac{x + \tilde{x}}{2} \right).
$$
From this and the fact that
$$
\big| {f(y) - f(x)} \big| \le C \| f \|_{\widetilde{C}^\beta_\gamma(\CC_{2}^+)} x_n^{- \beta - \gamma} \left| y - \frac{x + \tilde{x}}{2} \right|^\beta  \quad \mbox{for every } y \in \CC_{\frac{x_n}{2}}(x) \setminus B_{|x - \tilde{x}|} \left( \frac{x + \tilde{x}}{2} \right),
$$
changing variables we get
\begin{equation} \label{J2est}
J_2(x, \tilde{x}) \le C \| f \|_{\widetilde{C}^\beta_\gamma(\CC_{2}^+)} x_n^{- \beta - \gamma} |x - \tilde{x}| \int_{\R^n \setminus B_{|x - \tilde{x}|}} \frac{dz}{|z|^{n + 1 - \beta}} \le C \| f \|_{\widetilde{C}^\beta_\gamma(\CC_{2}^+)} x_n^{- \beta - \gamma} |x - \tilde{x}|^\beta.
\end{equation}
To estimate~$J_3$, we employ the divergence theorem to write
\begin{align*}
\int_{\CC_{\frac{x_n}{2}}(x) \setminus B_{|x - \tilde{x}|} \left( \frac{x + \tilde{x}}{2} \right)} {D_{i j} \Gamma(\tilde{x} - y)} \, dy = & - \int_{\partial \CC_{\frac{x_n}{2}}(x)} D_i \Gamma(\tilde{x} - y) \nu_j(y) \, d\mathcal{H}^{n - 1}(y) \\
& + \int_{\partial B_{|x - \tilde{x}|} \left( \frac{x + \tilde{x}}{2} \right)} D_i \Gamma(\tilde{x} - y) \nu_j(y) \, d\mathcal{H}^{n - 1}(y).
\end{align*}
Since both surface integrals are bounded uniformly with respect to~$x$ and~$\tilde{x}$, we infer that
\begin{equation} \label{J3est}
J_3(x, \tilde{x}) \le C \| f \|_{\widetilde{C}^\beta_\gamma(\CC_{2}^+)} x_n^{- \beta - \gamma} |x - \tilde{x}|^\beta.
\end{equation}
The terms~$J_4$ and~$J_5$ can be dealt with at once. Using that
\[
B_{|x - \tilde{x}|} \Big( \frac{x + \tilde{x}}{2} \Big) \ \subset\ 
B_{2 |x - \tilde{x}|}(x) \cap B_{2 |x - \tilde{x}|}(\tilde{x}),
\]
we estimate
\begin{equation} \label{J4+J5est}
J_4(x, \tilde{x}) + J_5(x, \tilde{x}) \le C \| f \|_{\widetilde{C}^\beta_\gamma(\CC_{2}^+)} x_n^{- \beta - \gamma} \int_{B_{2 |x - \tilde{x}|}} \frac{dz}{|z|^{n - \beta}} \le C \| f \|_{\widetilde{C}^\beta_\gamma(\CC_{2}^+)} x_n^{- \beta - \gamma} |x - \tilde{x}|^\beta.
\end{equation}
The analogous estimate for~$J_6$ follows by noticing that the surface integral which defines it is bounded, while that for~$J_7$ is a simple consequence of the bound
\[
\big| {D \Gamma(x - y) - D \Gamma(\tilde{x} - y)} \big| \le C \, \frac{|x - \tilde{x}|}{x_n^n} \quad \mbox{for every } y \in \partial \CC_{\frac{x_n}{2}}(x).
\]
From these observations,~\eqref{D2incremdeco},~\eqref{J1est},~\eqref{J2est},~\eqref{J3est}, and~\eqref{J4+J5est}, we conclude that claim~\eqref{hessCalphaest} holds true. The proof of the lemma is thus complete.
\end{proof}

From this, we may easily infer the regularity of the solution of~\eqref{localDirprobforu} near flat portions of the boundary.

\begin{proposition} \label{C2alphacyllem}
Let~$\beta, \gamma \in (0, 1)$,~$r \in (0, 1]$, and~$f \in \widetilde{C}^{\beta}_\gamma(\CC_{2 r}^+)$. Let~$u \in C^2(\CC_{2 r}^+) \cap C^0(\CC_{2 r}^+ \cup \DD_{2 r})$ be a bounded solution of~$- \Delta u = f$ in~$\CC_{2 r}^+$ satisfying~$u = 0$ on~$\DD_{2 r}$. Then,~$u \in \widetilde{C}^{2, \beta}_{-1,\gamma}(\CC_r^+)$ and
\begin{equation} \label{C2alphaestincyl}
\| u \|_{\widetilde{C}^{2, \beta}_{-1,\gamma}(\CC_r^+)} \le C \left( r^{\gamma - 2} \| u \|_{L^\infty(\CC_{2 r}^+)} + \| f \|_{\widetilde{C}^{\beta}_\gamma(\CC_{2 r}^+)} \right),
\end{equation}
for some constant~$C > 0$ depending only on~$n$,~$\beta$, and~$\gamma$.
\end{proposition}

\begin{proof}
First of all, by scaling we can reduce ourselves to the case~$r = 1$. Indeed, assuming the validity of the proposition for~$r = 1$, we define~$u_r(x) \coloneqq u(r x)$,~$f_r(x) \coloneqq r^2 f(r x)$ for~$x \in \CC_{2}^+ \cup \DD_2$, notice that they satisfy~$- \Delta u_r = f_r$ in~$\CC_2^+$,~$u_r = 0$ on~$\DD_2$, and infer therefore that
$$
\| u_r \|_{\widetilde{C}^{2, \beta}_{-1,\gamma}(\CC_1^+)} \le C \left( \| u_r \|_{L^\infty(\CC_{2}^+)} + \| f_r \|_{\widetilde{C}^{\beta}_\gamma(\CC_{2}^+)} \right).
$$
Rephrasing this inequality back to~$u$ and~$f$, we obtain
$$
r \left\{ \sup_{x \in \CC_r^+} \Big( {x_n^{-1} |u(x)|} \Big) + \| Du \|_{L^\infty(\CC_r^+)} \right\} + r^{2 - \gamma} \| D^2 u \|_{\widetilde{C}^\beta_\gamma(\CC_{r}^+)} \le C \left( \| u \|_{L^\infty(\CC_{2 r}^+)} + r^{2 - \gamma} \| f \|_{\widetilde{C}^\beta_\gamma(\CC_{2 r}^+)} \right),
$$
which gives~\eqref{C2alphaestincyl} since~$r \le 1$.

Let~$w$ be the function introduced in Lemma~\ref{Newtonpotlem} and write~$v \coloneqq u - w$. Clearly,~$v$ is of class~$C^2$ in~$\CC_{2}^+$, continuous up to~$\DD_{2}$, and harmonic in~$\CC_{2}^+$. As~$v$ vanishes on~$\DD_{2}$, we can consider its odd reflection across~$\DD_{2}$ which is harmonic in the whole~$\CC_{2}$. By standard interior estimates for harmonic functions, we deduce that~$v$ is smooth and satisfies~$\| D v \|_{C^2(\CC_1^+)} \le C \| v \|_{L^\infty(\CC_{2}^+)}$ for some dimensional constant~$C > 0$. The estimate in~\eqref{C2alphaestincyl} immediately follows from this and Lemma~\ref{Newtonpotlem}.
\end{proof}

By straightening the boundary, we may deduce from the previous result an \textit{a priori} estimate in~$C^{2, \beta}_{-1,\gamma}(\Omega)$ for solutions of~\eqref{localDirprobforu}. The precise statement is as follows.

\begin{theorem} \label{aprioriC2alphaestlem}
Let~$\beta, \gamma \in (0, 1)$,~$\Omega \subset \R^n$ be a bounded domain with boundary of class~$C^{2, \beta}$, and~$f \in C^{\beta}_\gamma(\Omega)$. Let~$u \in C^{2, \beta}_{-1,\gamma}(\Omega)$ be a solution of~$-\Delta u = f$ in~$\Omega$. Then,
\begin{equation} \label{C2alphamainest}
\| u \|_{C^{2, \beta}_{-1,\gamma}(\Omega)} \le C \left( \| u \|_{L^\infty(\Omega)} + \| f \|_{C^{\beta}_{\gamma}(\Omega)} \right),
\end{equation}
for some constant~$C > 0$ depending only on~$n$,~$\beta$,~$\gamma$, and~$\Omega$.
\end{theorem}
\begin{proof}
Given~$\delta > 0$, consider the set
\begin{equation} \label{Omegaetadef}
\Omega_\delta \coloneqq \left\{ x \in \Omega : d_x > \delta \right\}.
\end{equation}
By standard interior estimates, it is clear that~$u \in C^{2, \beta}(\overline{\Omega_\delta})$ for every~$\delta > 0$ with
\begin{equation} \label{C2alphaestinOmegaeta}
\| u \|_{C^{2, \beta}(\overline{\Omega_\delta})} \le C_\delta \left( \| u \|_{L^\infty(\Omega)} + \| f \|_{C^{\beta}_{\gamma}(\Omega)} \right),
\end{equation}
for some constant~$C_\delta > 0$ depending only on~$n$,~$\beta$,~$\gamma$, and~$\delta$.

We now address the boundary regularity of~$u$. We claim that, for every~$\bar{x} \in \partial \Omega$, there exist a radius~$\varrho_{\bar{x}} > 0$ and a constant~$C_{\bar{x}} > 0$, depending only on~$n$,~$\beta$,~$\gamma$,~$\partial \Omega$, and~$\bar{x}$, such that
\begin{equation} \label{uC2alphaestle0.25claim}
\| u \|_{\widetilde{C}^{2, \beta}_{-1,\gamma} \left( \Omega; B_{2 \varrho_{\bar{x}}}(\bar{x}) \right)} \le \frac{1}{16} \, \| u \|_{C^{2, \beta}_{-1,\gamma}(\Omega)} + C_{\bar{x}} \left( \| u \|_{L^\infty(\Omega)} + \| f \|_{C^\beta_\gamma(\Omega)} \right),
\end{equation}
where, for~$\bar{x} \in \partial \Omega$ and~$r > 0$,
\begin{align*}
\| u \|_{\widetilde{C}_{-1,\gamma}^{2, \beta}(\Omega; B_r(\bar{x}))} & \coloneqq \sup_{x \in \Omega \cap B_r(\bar{x})} \Big( {d_x^{-1} |u(x)|} \Big) + \| Du \|_{L^\infty(\Omega \cap B_r(\bar{x}))} \\
& \qquad + \sup_{x \in \Omega \cap B_r(\bar{x})} \Big( {d_x^\gamma \, |D^2 u(x)|} \Big) + \sup_{\substack{x, y \in \Omega \cap B_r(\bar{x})\\ x \ne y}} \left( d_{x, y}^{\beta + \gamma} \, \frac{|D^2 u(x) - D^2 u(y)|}{|x - y|^\beta} \right).
\end{align*}
Note that here~$d_x$ still indicates the distance of~$x$ from the entirety of the boundary of~$\Omega$.

As~$\partial \Omega$ is of class~$C^{2, \beta}$, for every~$\bar{x} \in \partial \Omega$ there exists a radius~$R = R_{\bar{x}} \in (0, 1)$ such that, up to a rotation, it holds~$\Omega \cap \CC_{R}(\bar{x}) = \big\{ {(x', x_n) \in \CC_{R}(\bar{x}) : x_n > h(x')} \big\}$, for some function~$h \in C^{2, \beta}(\R^{n - 1})$ satisfying~$h(\bar{x}') = \bar{x}_n$,~$D' h(\bar{x}') = 0'$, and~$\| h \|_{C^{2, \beta}(\R^{n - 1})} \le K$, where, from here on,~$K$ denotes a general constant larger than~$1$, depending at most on~$n$,~$\beta$,~$\gamma$,~$\partial \Omega$, and~$\bar{x}$. Without loss of generality, we may also suppose~$h$ to have compact support.

In order to prove~\eqref{uC2alphaestle0.25claim} we straighten the boundary around the point~$\bar{x}$, which, after a translation, we assume to be the origin. Consider the mapping~$\Psi: \R^n \to \R^n$ defined by~$\Psi(z) \coloneqq \big( {z', z_n + h(z')} \big)$ for every~$z = (z', z_n) \in \R^n$. It is easy to see that~$\Psi$ is a~$C^{2, \beta}$-diffeomorphism of~$\R^n$ onto itself such that~$\Psi \big( {\R^{n - 1} \times (0, +\infty)} \big) = \big\{ {(x', x_n) \in \R^n : x_n > h(x')} \big\}$ and~$\Phi(x) \coloneqq \Psi^{-1}(x) = \big( {x', x_n - h(x')} \big)$ for every~$x \in \R^n$. Also note that
\begin{equation} \label{Psiundercontrol}
\begin{dcases}
1 \le \| D\Psi \|_{C^{1, \beta}(\R^n)} \le K, \quad 1 \le \| D\Phi \|_{C^{1, \beta}(\R^n)} \le K, & \\
\frac{1}{K} \, |z - w| \le |\Psi(z) - \Psi(w)| \le K |z - w| & \quad \mbox{for all } z, w \in \R^n, \\
d_{\Psi(z)} \le z_n \le K d_{\Psi(z)} & \quad \mbox{for all } z \in \CC_{2 r}^+,
\end{dcases}
\end{equation}
for every~$r \le K^{-1} R$.

Observe now that~$\Psi \big( {\CC_{4 r_0}} \big) \subset \CC_{R}$, with~$r_0 \coloneqq \frac{R}{8} \left( 1 + \| D' h \|_{L^\infty(\R^{n - 1})} \right)^{-1}$. Thus, the function~$v \coloneqq u \circ \Psi$ belongs to~$C^2(\CC_{2 r}^+) \cap C^0(\CC_{2 r}^+ \cup \DD_{2 r})$ and solves
\begin{align*}
\left\lbrace\begin{aligned}
-\Delta v &= g && \text{in }\CC_{2 r}^+, \\
v &= 0 && \text{on }\DD_{2 r},
\end{aligned}\right.
\end{align*}
for every~$r \in \left( 0, r_0 \right]$, and where~$g \coloneqq - \mbox{Tr} \big( {A D^2 v} \big) - b \cdot Dv + f \circ \Psi $, with
$$
A(z) \coloneqq \left(\begin{array}{c|c}
\vphantom{\Big|} \mbox{0}_{n - 1} & D'h(z') \\
\hline
\vphantom{\Big|} D'h(z')^T & \big| {D'h(z')} \big|^2
\end{array}\right)
\quad \mbox{and} \quad
b(z) \coloneqq \big( {0', \Delta' h(z')} \big)^T.
$$
Hence, we may apply Proposition~\ref{C2alphacyllem} and obtain the estimate
\begin{equation} \label{vesttech1}
\begin{aligned}
\| v \|_{\widetilde{C}^{2, \beta}_{-1,\gamma}(\CC_r^+)} & \le K \left( r^{\gamma - 2} \| v \|_{L^\infty(\CC_{2 r}^+)} + \| g \|_{\widetilde{C}^{\beta}_\gamma(\CC_{2 r}^+)} \right) \\
& \le K \left( r^{\gamma - 2} \| u \|_{L^\infty(\Omega)} + \left\| \mbox{Tr} \big( {A D^2 v} \big) \right\|_{\widetilde{C}^{\beta}_\gamma(\CC_{2 r}^+)} + \| b \cdot Dv \|_{\widetilde{C}^{\beta}_\gamma(\CC_{2 r}^+)} + \| f \circ \Psi \|_{\widetilde{C}^{\beta}_\gamma(\CC_{2 r}^+)} \right).
\end{aligned}
\end{equation}

Note that~$A$ and~$b$ satisfy~$\| A \|_{L^\infty(\CC_{2 r}^+)} \le K r$ and~$\| DA \|_{L^\infty(\CC_{2 r}^+)} + \| b \|_{C^\beta(\CC_{2 r}^+)} \le K$. Consequently, after some computations we find that, if~$r$ is sufficiently small,
\begin{align*}
\left\| \mbox{Tr} \big( {A D^2 v} \big) \right\|_{\widetilde{C}^{\beta}_\gamma(\CC_{2 r}^+)} & \le K r \| u \|_{C^{2, \beta}_{-1,\gamma}(\Omega)}, \\
\| b \cdot Dv \|_{\widetilde{C}^{\beta}_\gamma(\CC_{2 r}^+)} & \le
 K r^\gamma \| u \|_{C^{2, \beta}_{-1,\gamma}(\Omega)}, \\
\| f \circ \Psi \|_{\widetilde{C}^\beta_\gamma(\CC_{2 r}^+)} & \le K \| f \|_{C^\beta_\gamma(\Omega)},
\end{align*}
where we also took advantage of~\eqref{Psiundercontrol}. Thanks to these bounds,~\eqref{vesttech1} yields
\[
\| v \|_{\widetilde{C}^{2, \beta}_{-1,\gamma}(\CC_r^+)} \le K \left( r^\gamma \| u \|_{C^{2, \beta}_{-1,\gamma}(\Omega)} + r^{\gamma - 2} \| u \|_{L^\infty(\Omega)} + \| f \|_{C^\beta_\gamma(\Omega)} \right).
\]
Letting~$\varrho \coloneqq \kappa_0 r$ with~$\kappa_0 \coloneqq \left( 2 \| D \Phi \|_{L^\infty(\R^n)} \right)^{-1}$, we clearly have that~$B_\varrho \subset \Psi(\CC_r)$. Thanks to this and~\eqref{Psiundercontrol}, it is not hard to see that~$\| u \|_{\widetilde{C}^{2, \beta}_{-1,\gamma}(\Omega; B_\varrho)} \le K \| v \|_{\widetilde{C}^{2, \beta}_{-1,\gamma}(\CC_r^+)}$. Accordingly, we conclude that
$$
\| u \|_{\widetilde{C}^{2, \beta}_{-1,\gamma}(\Omega; B_\varrho)} \le K \left( r^\gamma \| u \|_{C^{2, \beta}_{-1,\gamma}(\Omega)} + r^{\gamma - 2} \| u \|_{L^\infty(\Omega)} + \| f \|_{C^\beta_\gamma(\Omega)} \right),
$$
and claim~\eqref{uC2alphaestle0.25claim} follows by taking~$r$ (and therefore~$\varrho$) suitably small.

Now that~\eqref{uC2alphaestle0.25claim} is established, the conclusion of the lemma can be inferred from a simple covering argument. Indeed, let~$\{ \bar{x}_1, \ldots, \bar{x}_N \} \subset \partial \Omega$ be such that~$\partial \Omega \subset \bigcup_{j = 1}^N B_{\varrho_{\bar{x}_j}}(\bar{x}_j)$. Then, there exists a small~$\delta_0 > 0$ for which~$\Omega \subset \Omega_{\delta_0} \cup \bigcup_{j = 1}^N B_{\varrho_{\bar{x}_j}}(\bar{x}_j)$, with~$\Omega_{\delta_0}$ as in~\eqref{Omegaetadef}. Clearly, from this,~\eqref{C2alphaestinOmegaeta}, and~\eqref{uC2alphaestle0.25claim}, it follows that
\begin{equation} \label{C2estforu}
\begin{aligned}
& \sup_{x \in \Omega} \Big( {d_x^{-1} |u(x)|} \Big) + \| Du \|_{L^\infty(\Omega)} + \sup_{x \in \Omega} \Big( {d_x^\gamma \, |D^2 u(x)|} \Big) \\
& \hspace{100pt} \le \frac{3}{16} \, \| u \|_{C^{2, \beta}_{-1,\gamma}(\Omega)} + 3 C_\star \left( \| u \|_{L^\infty(\Omega)} + \| f \|_{C^\beta_\gamma(\Omega)} \right),
\end{aligned}
\end{equation}
with~$C_\star \coloneqq C_{\delta_0} + \max_{j \in \{ 1, \ldots, N \}} C_{\bar{x}_j}$. On the other hand, employing again~\eqref{C2alphaestinOmegaeta} and~\eqref{uC2alphaestle0.25claim}, we also compute
\begin{equation} \label{C2alphasemiestforu}
\begin{aligned}
& \sup_{\substack{x, y \in \Omega\\ x \ne y}} \left( d_{x, y}^{\beta + \gamma} \, \frac{|D^2 u(x) - D^2 u(y)|}{|x - y|^\beta} \right) = \sup_{\substack{x, y \in \Omega\\ x \ne y, \, d_x \le d_y}} \left( d_{x}^{\beta + \gamma} \, \frac{|D^2 u(x) - D^2 u(y)|}{|x - y|^\beta} \right) \\
& \hspace{26pt} \le \sup_{x, y \in \Omega_{\delta_0}} \left( d_{x}^{\beta + \gamma} \, \frac{|D^2 u(x) - D^2 u(y)|}{|x - y|^\beta} \right) + \max_{j \in \{ 1, \ldots, N \}} \, \sup_{\substack{x, y \in \Omega \cap B_{2 \varrho_{\bar{x}_j}} \!\! (\bar{x}_j)\\ x \ne y}} \left( d_{x, y}^{\beta + \gamma} \, \frac{|D^2 u(x) - D^2 u(y)|}{|x - y|^\beta} \right) \\
& \hspace{26pt} \quad + \max_{j \in \{ 1, \ldots, N \}} \, \sup_{\substack{x \in \Omega \cap B_{\varrho_{\bar{x}_j}} \!\! (\bar{x}_j), \, y \in \Omega \setminus B_{2 \varrho_{\bar{x}_j}} \!\! (\bar{x}_j)\\ x \ne y, \, d_x \le d_y}} \left( d_{x}^{\beta + \gamma} \, \frac{|D^2 u(x) - D^2 u(y)|}{|x - y|^\beta} \right) \\
& \hspace{26pt} \le \frac{7}{16} \, \| u \|_{C^{2, \beta}_{-1,\gamma}(\Omega)} + 7 C_\star \left( \| u \|_{L^\infty(\Omega)} + \| f \|_{C^\beta_\gamma(\Omega)} \right),
\end{aligned}
\end{equation}
where we used that, for every~$j \in \{ 1, \ldots, N \}$,
\begin{align*}
& \sup_{\substack{x \in \Omega \cap B_{\varrho_{\bar{x}_i}} \!\! (\bar{x}_j), \, y \in \Omega \setminus B_{2 \varrho_{\bar{x}_j}} \!\! (\bar{x}_i)\\ x \ne y, \, d_x \le d_y}} \left( d_{x}^{\beta + \gamma} \, \frac{|D^2 u(x) - D^2 u(y)|}{|x - y|^\beta} \right) \\
& \hspace{50pt} \le \sup_{\substack{x \in \Omega \cap B_{\varrho_{\bar{x}_j}} \!\! (\bar{x}_j), \, y \in \Omega \setminus B_{2 \varrho_{\bar{x}_j}} \!\! (\bar{x}_j)\\ x \ne y, \, d_x \le d_y}} \left\{ \left( \frac{d_{x}}{\varrho_{\bar{x}_j}} \right)^{\! \beta} \Big( {d_{x}^\gamma |D^2 u(x)| + d_{y}^\gamma |D^2 u(y)|} \Big) \right\} \le 2 \sup_{x \in \Omega} \Big( {d_x^\gamma |D^2 u(x)|} \Big) \\
& \hspace{50pt} \le \frac{3}{8} \, \| u \|_{C^{2, \beta}_{-1,\gamma}(\Omega)} + 6 C_\star \left( \| u \|_{L^\infty(\Omega)} + \| f \|_{C^\beta_\gamma(\Omega)} \right),
\end{align*}
thanks to~\eqref{C2estforu}. By combining~\eqref{C2estforu} and~\eqref{C2alphasemiestforu}, we obtain
$$
\| u \|_{C^{2, \beta}_{-1,\gamma}(\Omega)} \le \frac{5}{8} \, \| u \|_{C^{2, \beta}_{-1,\gamma}(\Omega)} + 10 \, C_\star \left( \| u \|_{L^\infty(\Omega)} + \| f \|_{C^\beta_\gamma(\Omega)} \right),
$$
which leads to~\eqref{C2alphamainest}, after we reabsorb on the left-hand side the~$C^{2, \beta}_{-1,\gamma}$ norm appearing on the right. The proof is complete.
\end{proof}

Thanks to the previous result, estimate~\eqref{mainC2alphaest} is reduced to an~$L^\infty(\Omega)$ estimate for solutions of~\eqref{localDirprobforu}. We take care of this issue in the following lemma, by means of a simple barrier.

\begin{lemma} \label{Linftybarrierlem}
Let~$\gamma \in (0, 1)$,~$\Omega \subset \R^n$ be a bounded domain with boundary of class~$C^2$, and~$f \in L^\infty_\loc(\Omega)$. Let~$u \in C^2(\Omega) \cap C^0(\overline{\Omega})$ be a solution of~\eqref{localDirprobforu}. Then,
\begin{equation} \label{aprioriLinftyest}
\| u \|_{L^\infty(\Omega)} \le C \sup_{x \in \Omega} \Big( {d_x^{\gamma} |f(x)|} \Big),
\end{equation}
for some constant~$C > 0$ depending only on~$n$,~$\gamma$, and~$\Omega$.
\end{lemma}
\begin{proof}
First, we recall that the distance function~$d(x) = d_x$ is of class~$C^2$ in a neighborhood~$N_{\delta_0} \coloneqq \overline{\Omega} \setminus \Omega_{\delta_0}$ of~$\partial \Omega$, for some~$\delta_0 > 0$ depending only on~$\Omega$---see, \textit{e.g.},~\cite{GT01}*{Lemma~14.16}. Here, we are using notation~\eqref{Omegaetadef}.

Set
$$
M \coloneqq \sup_{x \in \Omega} \Big( {d_x^{\gamma} |f(x)|} \Big).
$$
Up to a translation, we have that~$\Omega \subset B_R$, with~$R \coloneqq \diam(\Omega)$. Let~$\delta \in \left( 0, \frac{\delta_0}{2} \right)$ and~$\eta \in C^2([0, +\infty))$ be a non-increasing function satisfying~$\eta = 1$ in~$[0, 1]$,~$\eta = 0$ in~$[2, +\infty)$, and~$|\eta'| + |\eta''| \le 10$ in~$(1, 2)$. Define
$$
\varphi(x) \coloneqq B \left( R^2 - |x|^2 \right) - D \, \eta \big( {\delta^{-1} d(x)} \big) d(x)^{2 - \gamma} \quad \mbox{for } x \in \Omega,
$$
for some positive constants~$B$ and~$D$ to be determined. Clearly,~$\varphi \in C^2(\Omega) \cap C^0(\overline{\Omega})$ with~$\varphi \ge 0$ on~$\partial \Omega$. 

Using that~$|D d| = 1$ in~$N_{2 \delta}$, we compute
\begin{align*}
- \Delta \varphi(x) & = 2 n B + D \Big\{ {\eta'' \big( {\delta^{-1} d(x)} \big) \delta^{-2} d(x)^2 + 2 (2- \gamma) \eta' \big( {\delta^{-1} d(x)} \big) \delta^{-1} d(x) + (2 - \gamma) (1 - \gamma) \eta \big( {\delta^{-1} d(x)} \big)} \\
& \quad {+ \Big( {\eta' \big( {\delta^{-1} d(x)} \big) \delta^{-1} d(x) + (2 - \gamma) \eta \big( {\delta^{-1} d(x)} \big)} \Big) d(x) \Delta d(x)} \Big\} d(x)^{- \gamma} \\
& \ge B + D \, \Big\{ {- 120 \chi_{(\delta, 2 \delta)}\big( {d(x)} \big)} + (1 - \gamma) \chi_{(0, \delta]} \big( {d(x)} \big) - 50 \delta \| \Delta d \|_{L^\infty(N_{\delta_0})} \chi_{(0, 2 \delta)}\big( {d(x)} \big) \Big\} d(x)^{- \gamma} \\
& \ge B + D \left\{ \frac{1 - \gamma}{2} \, \chi_{(0, \delta]} \big( {d(x)} \big) - 121  \, \chi_{(\delta, 2\delta)} \big( {d(x)} \big) \right\} d(x)^{- \gamma},
\end{align*}
provided we take~$\delta \coloneqq \min \left\{ \frac{\delta_0}{4}, \frac{1 - \gamma}{100} \| \Delta d \|_{L^\infty(N_{\delta_0})}^{-1} \right\}$. Choosing now
\[
D \coloneqq \frac{2 M}{1 - \gamma} \quad \mbox{and} \quad B \coloneqq \frac{M}{\delta^\gamma} \left( \frac{242}{1 - \gamma} + 1 \right),
\]
we get that~$- \Delta \varphi \ge |f|$ in~$\Omega$. From the weak comparison principle we thus infer that~$|u| \le \varphi$ in~$\Omega$, which gives~\eqref{aprioriLinftyest}.
\end{proof}

We have now all the ingredients needed to establish Theorem~\ref{mainC2alphaexistprop}.
In particular, by the last two results, we only need to show that problem~\eqref{localDirprobforu} actually admits a~$C^{2, \beta}_{-1,\gamma}(\Omega)$-solution.

\begin{proof}[Proof of Theorem~\ref{mainC2alphaexistprop}]
To establish the existence of a~$C^{2, \beta}_{-1,\gamma}(\Omega)$-solution of problem~\eqref{localDirprobforu}, we use the following simple approximation argument.

Let~$\eta \in C^1([0, +\infty))$ be a non-decreasing function satisfying~$\eta = 0$ in~$[0, 1]$,~$\eta = 1$ in~$[2, +\infty)$, and~$|\eta'| \le 2$ in~$(1, 2)$. Given~$k \in \N$, let~$f_k(x) \coloneqq \eta \big( {k d(x)} \big) f(x)$ for~$x \in \Omega$. Clearly,~$f_k$ is of class~$C^\beta$ up to the boundary of~$\Omega$ and thus, in particular,~$f_k \in C^\beta_\gamma(\Omega)$. Moreover, its~$C^\beta_\gamma(\Omega)$ norm is bounded uniformly with respect to~$k$, as it holds
\begin{equation} \label{fknormlef}
\| f_k \|_{C^\beta_\gamma(\Omega)} \le 5 \| f \|_{C^\beta_\gamma(\Omega)}.
\end{equation}
Indeed, we have
$$
\sup_{x \in \Omega} \Big( {d_x^{\gamma} |f_k(x)|} \Big) \le \sup_{x \in \Omega} \Big( {d_x^{\gamma} |f(x)|} \Big)
$$
and
\begin{align*}
& \sup_{\substack{x, y \in \Omega\\x \ne y}} \left( d_{x, y}^{\beta + \gamma} \frac{|f_k(x) - f_k(y)|}{|x - y|^\beta} \right) = \sup_{\substack{x, y \in \Omega\\x \ne y, \, d_x \le d_y}} \left( d_x^{\beta + \gamma} \frac{|f_k(x) - f_k(y)|}{|x - y|^\beta} \right) \\
& \hspace{30pt} \le \sup_{\substack{x, y \in \Omega\\x \ne y, \, d_x \le \min \left\{ d_y, \frac{2}{k} \right\}}} \left\{ d_x^{\gamma} |f(x)| \left( d_x \, \frac{\big| {\eta \big( {k d(x)} \big) - \eta \big( {k d(y)} \big)} \big|}{|x - y|} \right)^{\! \beta} \big| {\eta \big( {k d(x)} \big) - \eta \big( {k d(y)} \big)} \big|^{1 - \beta} \right\} \\
& \hspace{30pt} \quad + \sup_{\substack{x, y \in \Omega\\x \ne y, \, d_x \le \min \left\{ d_y, \frac{2}{k} \right\}}} \left( d_x^{\beta + \gamma} \frac{|f(x) - f(y)|}{|x - y|^\beta} \, \big| {\eta \big( {k d(y)} \big)} \big| \right) + \sup_{\substack{x, y \in \Omega\\x \ne y, \, \frac{2}{k} < d_x \le d_y}} \left( d_x^{\beta + \gamma} \frac{|f(x) - f(y)|}{|x - y|^\beta} \right) \\
& \hspace{30pt} \le 2^{2 \beta} \sup_{x \in \Omega} \Big( d_x^\gamma |f(x)| \Big) + 2 \sup_{\substack{x, y \in \Omega\\x \ne y}} \left( d_{x, y}^{\beta + \gamma} \frac{|f(x) - f(y)|}{|x - y|^\beta} \right),
\end{align*}
which give~\eqref{fknormlef}.

As~$f_k \in C^\beta(\overline\Omega)$, by standard elliptic theory there exists a unique solution~$u_k \in C^{2, \beta}(\overline\Omega)$ of
\begin{align*}
\left\lbrace\begin{aligned}
-\Delta u_k &= f_k && \text{in } \Omega, \\
u_k &= 0 && \text{on } \partial \Omega.
\end{aligned}\right.
\end{align*}
By virtue of Theorem~\ref{aprioriC2alphaestlem}, Lemma~\ref{Linftybarrierlem}, and estimate~\eqref{fknormlef}, we have that
$$
\| u_k \|_{C^{2, \beta}_{-1,\gamma}(\Omega)} \le C \| f \|_{C^\beta_\gamma(\Omega)} \quad \mbox{for every } k \in \N,
$$
for some constant~$C > 0$ depending only on~$n$,~$\beta$,~$\gamma$, and~$\Omega$. Thanks to this uniform bound, standard compact embeddings of H\"older spaces, and a diagonal argument, we conclude that, up to a subsequence,~$\{ u_k \}$ converges in~$C^0(\overline{\Omega}) \cap C^2(\Omega)$ to a function~$u \in C^{2, \beta}_{-1,\gamma}(\Omega)$. Clearly,~$u$ solves~\eqref{localDirprobforu} and satisfies~\eqref{mainC2alphaest}.
\end{proof}

\section{Maximum principles}\label{sec:maxprinc}

Before turning to the proof of Theorem~\ref{mainlinearexthmm}, 
we need maximum principles for the mixed operator~$p \, (-\Delta) + q \, L_k + g \cdot D$ and an \textit{a priori} estimate for the norm~$\| \cdot \|_{C^0_{-1}(\Omega)}$ which stems from them.
In order to do that, we need a preparatory lemma.

\begin{lemma}\label{lem:aux}
Let~$k$ be a kernel satisfying~\eqref{k}, for some~$s \in (0, 1)$ and~$\kappa_2 \ge \kappa_1 > 0$. Given~$\lambda, R > 0$, let
\begin{equation}\label{vlemaux}
v(x) = v_{\lambda, R}(x) \coloneqq e^{\lambda x_1}\chi_{B_{2R}}(x)
\qquad\text{for }x\in\R^n.
\end{equation}
Then, there exists a constant~$C \ge 1$, depending only on~$n$,~$s$,~$\kappa_1$, and~$\kappa_2$, such that
\[
L_k v(x)\leq 
- \frac{1}{C} \frac{e^{\lambda R/2}}{\lambda R^{1+2s}} \, v(x)
\qquad\text{for every }x\in B_R,
\]
provided~$\lambda R \ge C$.
\end{lemma}
\begin{proof}
By recalling~\eqref{k} and taking~$\lambda \ge 4/R$, we directly estimate, for~$x\in B_R$,
\begin{equation} \label{f89u03489hf4hfi}
\begin{aligned}
L_k v(x) &= 
\pv \int_{\R^n} \left( e^{\lambda x_1}-e^{\lambda(x_1 + z_1)}\chi_{B_{2R}}(x + z) \right) k(z) \, dz \\
& \le
v(x) \left\{
\pv\int_{B_R} \left( 1-e^{\lambda z_1} \right) k(z) \, dz
+ \int_{\R^n\setminus B_R} k(z) \, dz \right\} \\
& =
v(x) \left\{ \frac{1}{\lambda^{n}} \,
\pv\int_{B_{\lambda R}} \left( 1-e^{w_1} \right) k \left( \frac{w}{\lambda} \right) dw
+ \int_{\R^n\setminus B_R} k(z) \, dz \right\} \\
&\leq
v(x) \left\{ \frac{1}{\lambda^{n}} \int_{B_{\lambda R}\setminus B_1} \left( 1-e^{w_1} \right) k \left(  \frac{w}{\lambda} \right) dw - \frac{1}{\lambda^{n}}\,
\pv\int_{B_1} w_1 \, k \left( \frac{w}{\lambda} \right) dw + \kappa_2 \int_{\R^n\setminus B_R} \frac{dz}{{|z|}^{n+2s}} \right\} \\
& = 
v(x) \left\{\frac{1}{\lambda^{n}} \int_{B_{\lambda R}\setminus B_1} \left( 1-e^{w_1} \right) k \left(  \frac{w}{\lambda} \right) dw
+ \frac{\cH^{n-1}(\partial B_1) \, \kappa_2}{2s} \, R^{-2s}
\right\},
\end{aligned}
\end{equation}
where we have also used that
\[
1-e^{w_1}\leq - w_1 \quad \mbox{for every } w_1 \in \R \quad \text{and} \quad 
\pv\int_{B_1} w_1 \, k \left( \frac{w}{\lambda} \right) dw = 0,
\quad\text{by symmetry}.
\]
We now look more closely at the term
\[
\int_{B_{\lambda R}\setminus B_1} \left( 1-e^{w_1} \right) k \left(  \frac{w}{\lambda} \right) dw,
\]
that we split as
\[
\int_{B_{\lambda R}\setminus B_1} \left( 1-e^{w_1} \right) k \left(  \frac{w}{\lambda} \right) dw
=
\int_{(B_{\lambda R}\setminus B_1)\cap\{w_1\le1\}} \left( 1-e^{w_1} \right) k \left(  \frac{w}{\lambda} \right) dw
+ \int_{B_{\lambda R}\cap\{w_1>1\}} \left( 1-e^{w_1} \right) k \left(  \frac{w}{\lambda} \right) dw.
\]
Using~\eqref{k}, the first of the above two integrals can simply be estimated by
\begin{equation} \label{09u24hge}
\begin{aligned}
\int_{(B_{\lambda R}\setminus B_1)\cap\{w_1 \le 1\}} \left( 1-e^{w_1} \right) k \left( \frac{w}{\lambda} \right) dw
& \leq
\int_{(B_{\lambda R}\setminus B_1)\cap\{w_1 \le 1\}} k \left( \frac{w}{\lambda} \right) dw \\
& \leq 
\kappa_2 \, \lambda^{n + 2 s} \int_{\R^n\setminus B_1}\frac{dw}{{|w|}^{n+2s}}
=
\frac{\cH^{n-1}(\partial B_1)}{2s} \, \kappa_2 \, \lambda^{n + 2 s}.
\end{aligned}
\end{equation}
The second integral can be instead treated as follows. Keeping in mind once again assumption~\eqref{k} and that~$1-e^{w_1} < - \frac{1}{2} \, e^{w_1}$ for every~$w_1 > 1$, we have that
\begin{align*}
\int_{B_{\lambda R}\cap\{w_1>1\}} \left( 1-e^{w_1} \right) k \left(  \frac{w}{\lambda} \right) dw
& \leq - \frac{1}{2} \int_{B_{\lambda R}\cap\{w_1>1\}} e^{w_1} k \left(  \frac{w}{\lambda} \right) dw \\
& \leq - \frac{\kappa_1}{2} \, \lambda^{n + 2 s} \int_{B_{\lambda R}\cap\{w_1>1\}} \frac{e^{w_1}}{|w|^{n+2s}}\, dw \\
& \leq - \frac{\kappa_1}{2} \, R^{- n - 2 s} \int_{B_{\lambda R}\cap\{w_1>1\}} e^{w_1} \, dw.
\end{align*}
By defining
\[
E\coloneqq\bigg\{(w_1, w')\in\R\times\R^{n-1}:1<w_1<\frac{\lambda R}2,\,|w'|<\frac{\sqrt3}2\lambda R\bigg\}
\subset B_{\lambda R}\cap\{w_1>1\},
\]
we have
\[
\int_{B_{\lambda R}\cap\{w_1>1\}} e^{w_1} \, dw \geq
\int_E e^{w_1} \, dw =
c\,(\lambda R)^{n-1}\int_1^{\lambda R/2} e^{w_1} \, dw_1
\ge \frac{c}{2} \, (\lambda R)^{n-1} e^{\lambda R/2},
\]
for some constant~$c>0$ only depending on~$n$ and where we have used that~$\lambda R \ge 4$. Therefore,
\[
\int_{B_{\lambda R}\cap\{w_1>1\}} \left( 1-e^{w_1} \right) k \left(  \frac{w}{\lambda} \right) dw \le - \frac{c \, \kappa_1}{4} \, \lambda^{n - 1} R^{- 1 - 2 s} \, e^{\lambda R/2}.
\]
Plugging this and~\eqref{09u24hge} into~\eqref{f89u03489hf4hfi} we obtain
\[
L_k v(x) \le v(x) \left\{ C \, \Big( {\lambda^{2 s} + R^{- 2 s}} \Big) - \frac{2}{C} \, \frac{e^{\lambda R/2}}{\lambda R^{1 + 2 s}}
\right\} = - \frac{1}{C} \, \frac{e^{\lambda R/2}}{\lambda R^{1 + 2 s}} \, v(x) \left\{ 2 - C^2 \left( \frac{(\lambda R)^{1 + 2 s}}{e^{\lambda R/2}} + \frac{\lambda R}{e^{\lambda R/2}} \right) \right\},
\]
for some constant~$C \ge 4$ only depending on~$n$,~$s$,~$\kappa_1$, and~$\kappa_2$. This leads to the inequality claimed in the statement, provided we take~$\lambda R$ suitably large.
\end{proof}

With this in hand, we may now state and prove our first maximum principle.

\begin{lemma}\label{lem:comparison0}
Let~$\Omega\subseteq\R^n$ be a bounded open set and~$k$ be a kernel satisfying~\eqref{k}, for some~$s \in (0, 1)$ and~$\kappa_2 \ge \kappa_1 > 0$. Let~$g\in L^\infty(\Omega)$
and~$p, q: \Omega \to \R$ be two measurable non-negative functions.
Let~$\underline{u}: \R^n \to \R$ be a measurable function, continuous in an open neighborhood of~$\overline{\Omega}$, and~$C^2$ inside~$\Omega$, which satisfies
$$
\int_{\R^n} \frac{|\underline{u}(x)|}{1 + |x|^{n + 2 s}} \, dx < +\infty
$$
and
\begin{equation}\label{0f9jio4}
p \, (-\Delta) \underline{u} + q \, L_k \underline{u} + g\cdot D\underline{u} \le 0  
\qquad \mbox{in } \Omega.
\end{equation}
\begin{enumerate}[label=$(\roman*)$,leftmargin=*]
\item \label{comparison01} If the inequality in~\eqref{0f9jio4} is strict, then 
\[
\sup_\Omega\underline{u}\leq\sup_{\R^n\setminus\Omega}\underline{u}.
\]
\item \label{comparison02} If~$q>0$ in~$\Omega$ and~$\underline{u}$
achieves a global maximum in $\Omega$, 
then $\underline{u}$ is constant.
\item \label{comparison03} If~$\inf_\Omega(p+q)>0$, then 
\[
\sup_\Omega\underline{u}\leq\sup_{\R^n\setminus\Omega}\underline{u}.
\]
\end{enumerate}
\end{lemma}
\begin{proof}
Without loss of generality, we can suppose that
\[
\sup_{\R^n\setminus\Omega}\underline{u}<+\infty,
\]
otherwise our claims are trivial.

Suppose that $\underline{u}$ has a global maximum at $x_0\in\Omega$. Then
\[
-\Delta\underline{u}(x_0)\geq 0,\quad
L_k\underline{u}(x_0)\geq 0, \quad
D\underline{u}(x_0)=0,
\]
and therefore
\begin{equation}\label{f90j4nf}
p(x_0) \, (-\Delta) \underline{u}(x_0) + q(x_0) \, L_k\underline{u}(x_0) + g(x_0)\cdot D\underline{u}(x_0) =
p(x_0) \, (-\Delta) \underline{u}(x_0) + q(x_0) \, L_k\underline{u}(x_0) \geq 0.
\end{equation}
This is not compatible with~\eqref{0f9jio4} holding with a strict inequality. Hence, point~\ref{comparison01} is established.

On the other hand,~\eqref{f90j4nf} can hold alongside~\eqref{0f9jio4} with~$q>0$ only if~$L_k\underline{u}(x_0)=0$,
which is in turn only possible when~$\underline{u}$ is constant 
(given the fact that at $x_0$ a global maximum is reached). This proves~\ref{comparison02}.

We turn to the proof of~\ref{comparison03}. Let~$R>0$ be fixed in such a way that~$\Omega\subset B_R$ and~$v = v_{\lambda, R}$ be as in~\eqref{vlemaux}. An application of Lemma~\ref{lem:aux} and straightforward computations give that
\[
-\Delta v(x)=-\lambda^2v(x),\quad
L_k v(x) \leq - \frac{1}{C} \frac{e^{\lambda R/2}}{\lambda R^{1+2s}}\,v(x),\quad
Dv(x)=\lambda v(x)\,e_1,
\qquad\text{for every }x\in\Omega,
\]
for some~$C \ge 1$, depending only on~$n$,~$s$,~$\kappa_1$, and~$\kappa_2$, and provided~$\lambda \ge C R^{-1}$.
Define now~$u_\eps\coloneqq\underline{u}+\eps v$, for~$\eps>0$. Taking~\eqref{0f9jio4} into account, we see that
\[
p \, (-\Delta) u_\eps + q \, L_k u_\eps + g\cdot D u_\eps \le 
\eps \Big( {p \, (-\Delta) v + q \, L_k v + g\cdot D v} \Big) \leq
\eps \bigg( {-\lambda^2 p - \frac{q}{C} \frac{e^{\lambda R/2}}{\lambda R^{1+2s}} + \lambda g\cdot e_1} \bigg) v
\quad\text{in }\Omega.
\]
Using that~$\inf_\Omega(p+q)>0$ and~$v > 0$ in~$\Omega$, by choosing~$\lambda$ sufficiently large we then get that
\[
p \, (-\Delta) u_\eps + q \, L_k u_\eps + g\cdot D u_\eps < 0
\qquad\text{in }\Omega,
\]
By point~\ref{comparison01}, we deduce from the above inequality that
\[
\underline{u}<u_\eps\leq\sup_\Omega u_\eps\leq\sup_{\R^n\setminus\Omega}u_\eps
\leq\sup_{\R^n\setminus\Omega}\underline{u}+\eps e^{2 \lambda R}
\qquad\text{in } \Omega,
\]
for every~$\varepsilon > 0$. By sending~$\eps\to 0^+$ we obtain the estimate claimed in~\ref{comparison03}.
\end{proof}

As a consequence of the above maximum principle we deduce the following.

\begin{proposition}\label{prop:comparison}
Let~$\Omega\subseteq\R^n$ be a bounded open set and~$k$ be a kernel satisfying~\eqref{k}, for some~$s \in (0, 1)$ and~$\kappa_2 \ge \kappa_1 > 0$. Let~$g\in L^\infty(\Omega)$ and~$p, q: \Omega \to \R$ be two measurable non-negative functions.
Let~$\underline{u} \in C^2(\Omega) \cap C^0(\R^n)$ be satisfying
\begin{equation}\label{0f9jio42}
\left\lbrace\begin{aligned}
p \, (-\Delta) \underline{u} + q \, L_k \underline{u} + g\cdot D\underline{u} &\le 0 && \mbox{in } \Omega, \\
\underline{u} &\le 0 && \mbox{on } \partial \Omega, \\
\underline{u} &\le 0 && \mbox{in } \R^n \setminus \overline{\Omega},
\end{aligned}\right.
\end{equation}
and assume either one of the following conditions to hold: 
\begin{enumerate}[label=$(\roman*)$,leftmargin=*]
\item the inequality on the first line of~\eqref{0f9jio42} is strict;
\item $q>0$ in~$\Omega$;
\item $\inf_\Omega(p+q)>0$.
\end{enumerate}
Then,~$\underline{u} \le 0$ in~$\R^n$.
\end{proposition}

Thanks to the maximum principle, we can also establish an \textit{a priori} estimate for subsolutions of the Dirichlet problem~\eqref{dir} ensuring linear growth from the boundary.

\begin{theorem} \label{lineargrowththm}
Let~$\Omega\subseteq\R^n$ be a bounded open set having the exterior ball property and~$k$ be a kernel satisfying~\eqref{k}, for some~$s \in (0, 1)$ and~$\kappa_2 \ge \kappa_1 > 0$. Let~$g\in L^\infty(\Omega)$ and~$p, q: \Omega \to \R$ be two measurable non-negative functions satisfying
\[
\inf_\Omega p > 0 
\qquad \text{and} \qquad 
\sup_\Omega q <+\infty.
\]
Let~$f: \Omega \to \R$ be a measurable function bounded from above. If~$u \in C^2(\Omega) \cap C^0(\R^n)$ satisfies
\[
\left\lbrace\begin{aligned}
p \, (-\Delta) u + q \, L_k u + g\cdot Du &\le f && \mbox{in } \Omega, \\
u &\le 0 && \mbox{on } \partial \Omega, \\
u &\le 0 && \mbox{in } \R^n \setminus \overline{\Omega},
\end{aligned}\right.
\]
then
\[
\| u_+ \|_{C^0_{-1}(\Omega)} \le C \| f_+ \|_{L^\infty(\Omega)},
\]
for some constant~$C > 0$ depending only on~$n$,~$s$,~$\kappa_1$,~$\kappa_2$,~$\Omega$,~$\inf_\Omega p$,~$\sup_\Omega q$, and~$\|g\|_{L^\infty(\Omega)}$.
\end{theorem}
\begin{proof}
We begin by showing that the~$C^0_{-1}(\Omega)$ norm of~$u_+$ can be bounded in terms of the suprema of~$u_+$ and~$f_+$, that is
\begin{equation} \label{lineargrowthtechine}
\| u_+ \|_{C^0_{-1}(\Omega)} \le C \Big( {\| u_+ \|_{L^\infty(\Omega)} + \| f_+ \|_{L^\infty(\Omega)}} \Big),
\end{equation}
for some constant~$C > 0$ depending only on~$n$,~$s$,~$\kappa_2$,~$\Omega$,~$\inf_\Omega p$,~$\sup_\Omega q$, and~$\|g\|_{L^\infty(\Omega)}$.

As~$\Omega$ has the exterior ball property, there exists a radius~$r_0 \in (0, 1)$ such that, corresponding to each point~$z \in \partial \Omega$ there is a point~$y_z \in \R^n \setminus \Omega$ such that~$\overline{B_{r_0}(y_z)} \cap \overline{\Omega} = \{ z \}$. It is immediate to verify that~\eqref{lineargrowthtechine} will be proved if we show that there exist two constants~$C \ge 1$ and~$\delta \in \left( 0, \frac{1}{2} \right]$, both depending only on~$n$,~$s$,~$\kappa_2$,~$r_0$,~$\inf_\Omega p$,~$\sup_\Omega q$, and~$\|g\|_{L^\infty(\Omega)}$, such that
\begin{equation} \label{ulessthanlinearnearbdry}
u(x) \le C \Big( {\| u_+ \|_{L^\infty(\Omega)} + \| f_+ \|_{L^\infty(\Omega)}} \Big) |x - z| \quad \mbox{for every } z \in \partial \Omega \mbox{ and } x \in \Omega \cap B_{(1 + \delta) r_0}(y_z).
\end{equation}

Let~$z \in \partial \Omega$ be fixed. Up to a translation, we may assume that~$y_z = 0$, so that in particular~$z \in \partial B_{r_0}$. Let~$\sigma \in \left( 0, \min \big\{ {2 - 2 s, 1} \big\} \right)$ and consider the radially symmetric function
\[
\psi(x) \coloneqq \begin{dcases}
0 & \quad \mbox{if } x \in \overline{B}_{r_0}, \\
|x| - r_0 - \frac{\big( {|x| - r_0} \big)^{1 + \sigma}}{1 + \sigma} & \quad \mbox{if } x \in B_{2 r_0} \setminus \overline{B}_{r_0}, \\
r_0 - \frac{r_0^{1 + \sigma}}{1 + \sigma} & \quad \mbox{if } x \in \R^n \setminus B_{2 r_0}.
\end{dcases}
\]
Clearly,~$\psi$ is globally bounded and Lipschitz continuous, smooth inside~$B_{2 r_0} \setminus \overline{B}_{r_0}$. Moreover, it is radially non-decreasing and thus non-negative, since~$r_0 \le 1$. We claim that
\begin{equation} \label{psisuperclaim}
p \, (- \Delta) \psi + q \, L_k \psi + g \cdot D\psi \ge 1 \quad \mbox{in } B_{(1 + \delta) r_0} \setminus \overline{B}_{r_0},
\end{equation}
if~$\delta \in \left( 0, \frac{1}{2} \right]$ is sufficiently small, in dependence of~$n$,~$s$,~$\kappa_2$,~$r_0$,~$\inf_\Omega p$,~$\sup_\Omega q$, and~$\|g\|_{L^\infty(\Omega)}$, and~$\sigma$ only.

In order to establish~\eqref{psisuperclaim}, we compute, for~$x \in B_{2 r_0} \setminus \overline{B}_{r_0}$,
\begin{equation} \label{DpsiD2psi}
\begin{aligned}
D \psi(x) & = \left( 1 - \big( {|x| - r_0} \big)^{\sigma} \right) \frac{x}{|x|}, \\
D^2 \psi(x) & = - \sigma \big( {|x| - r_0} \big)^{\sigma - 1} \, \frac{x \otimes x}{|x|^2} + \left( 1 - \big( {|x| - r_0} \big)^{\sigma} \right) \left( \frac{I_n}{|x|} - \frac{x \otimes x}{|x|^3} \right),
\end{aligned}
\end{equation}
where~$I_n$ is the~$n \times n$ identity matrix. From this, it follows in particular that
\begin{equation} \label{psisuperharm}
\begin{aligned}
- \Delta \psi(x) & \ge \sigma \big( {|x| - r_0} \big)^{\sigma - 1} - \frac{n - 1}{r_0} \ge \frac{\sigma}{2} \big( {|x| - r_0} \big)^{\sigma - 1} \\
g(x)\cdot D\psi(x) & \ge -\|g\|_{L^\infty(\Omega)}
\end{aligned}
\qquad \mbox{for all } x \in B_{(1 + \delta) r_0} \setminus \overline{B}_{r_0},
\end{equation}
provided~$\delta \le \left( \frac{\sigma r_0^\sigma}{2 (n - 1)} \right)^{\! \frac{1}{1 - \sigma}}$. 
We now estimate~$L_k \psi$. To this aim, we write
\[
L_k \psi(x) = I_1(x) + I_2(x) + I_3(x) + I_4(x),
\]
where
\begin{align*}
I_1(x) & \coloneqq - \int_{B_{|x| - r_0}} \big( {\psi(x + z) - \psi(x) - D \psi(x) \cdot z} \big) k(z) \, dz, \\
I_2(x) & \coloneqq \psi(x) \int_{B_{r_0}} k(y - x) \, dy, \\
I_3(x) & \coloneqq \big( {\psi(x) - \psi(2 r_0 e_1)} \big) \int_{\R^n \setminus B_{2 r_0}} k(y - x) \, dy, \\
I_4(x) & \coloneqq \int_{B_{2 r_0} \setminus \left( B_{r_0} \cup B_{|x| - r_0}(x) \right)} \big( {\psi(x) - \psi(y)} \big) k(y - x) \, dy.
\end{align*}
We stress that, here and elsewhere in the paper,~$B_r(p)$ denotes the open ball of radius~$r$ centered at the point~$p$, while~$B_r$ stands for the ball of radius~$r$ centered at the origin, \textit{i.e.},~$B_r = B_r(0)$. By Taylor expansion, for every~$z \in B_{|x| - r_0}$ there exists~$p = p(x, z) \in B_{|z|}(x) \subset B_{2 r_0} \setminus \overline{B}_{r_0}$ such that
\[
\psi(x + z) - \psi(x) - D \psi(x) \cdot z = \frac{1}{2} \langle D^2 \psi(p) z, z \rangle.
\]
Recalling~\eqref{DpsiD2psi}, we deduce that
\[
\psi(x + z) - \psi(x) - D \psi(x) \cdot z = - \frac{\sigma}{2} \, \big( {|p| - r_0} \big)^{\sigma - 1} \, \frac{(p \cdot z)^2}{|p|^2} + \frac{1 - \big( {|p| - r_0} \big)^{\sigma}}{2} \left( \frac{|z|^2}{|p|} - \frac{(p \cdot z)^2}{|p|^3} \right) \le \frac{|z|^2}{r_0},
\]
and thus, thanks to assumption~\eqref{k},
$$
I_1(x) \ge - \frac{\kappa_2}{r_0} \int_{B_{|x| - r_0}} \frac{dz}{|z|^{n - 2 + 2 s}} = - \frac{\mathcal{H}^{n - 1}(\partial B_1) \, \kappa_2}{2 (1 - s)} \frac{\big( {|x| - r_0} \big)^{2 - 2 s}}{r_0}.
$$
On the other hand, recalling the definition of~$\psi$,~\eqref{DpsiD2psi}, and again~\eqref{k}, we simply estimate
$$
I_2(x) \ge 0, \qquad I_3(x) \ge - \kappa_2 \, r_0 \int_{\R^n \setminus B_{\frac{r_0}{2}}} \frac{dz}{|z|^{n + 2 s}} \ge - \frac{2 \, \mathcal{H}^{n - 1}(\partial B_1) \, \kappa_2}{s} \, r_0^{1 - 2 s},
$$
and
$$
I_4(x) \ge - \kappa_2 \| D\psi \|_{L^\infty(B_{2 r_0})} \int_{B_{4 r_0} \setminus B_{|x| - r_0}} \frac{dz}{|z|^{n - 1 + 2 s}} \ge - C \, r_0^{1 - 2 s} \, A_s \left( \frac{|x| - r_0}{r_0} \right),
$$
for some constant~$C > 0$, depending only on~$n$,~$s$, and~$\kappa_2$, and where
$$
A_s(r) \coloneqq \begin{cases}
1 & \quad \mbox{if } s \in \left( 0, \frac{1}{2} \right), \\
- \log r & \quad \mbox{if } s = \frac{1}{2}, \\
r^{1 - 2 s} & \quad \mbox{if } s \in \left( \frac{1}{2}, 1 \right),
\end{cases}
$$
for~$r \in (0, 1)$. All in all, we found that
\begin{equation} \label{eq:Lkge-}
L_k \psi(x) \ge - C r_0^{1 - 2 s} A_s \! \left( \frac{|x| - r_0}{r_0} \right) \quad \mbox{for every } x \in B_{(1 + \delta) r_0} \setminus \overline{B}_{r_0}.
\end{equation}
By comparing this with~\eqref{psisuperharm}, we are easily led to~\eqref{psisuperclaim}, provided we take~$\delta$ sufficiently small.

We are now in a position to establish~\eqref{ulessthanlinearnearbdry}. Indeed, as~$\psi \ge 0$ in~$B_{(1 + \delta) r_0}$,~$\psi \ge \kappa \coloneqq \psi \big( {(1 + \delta) r_0 e_1} \big)$ in~$\R^n \setminus B_{(1 + \delta) r_0}$, and~\eqref{psisuperclaim} holds true, we get that~$v \coloneqq u - M \psi$, with~$M \coloneqq \kappa^{-1} \| u_+ \|_{L^\infty(\Omega)} + \| f_+ \|_{L^\infty(\Omega)}$, satisfies
$$
\left\lbrace\begin{aligned}
p \, (-\Delta) v + q \, L_k v + g\cdot Dv &\le 0 && \mbox{in } \Omega \cap B_{(1 + \delta) r_0}, \\
v &\le 0 && \mbox{in } \R^n \setminus \big( {\Omega \cap B_{(1 + \delta) r_0}} \big).
\end{aligned}\right.
$$
By the maximum principle of Proposition~\ref{prop:comparison}, we conclude that~$v \le 0$ also in~$\Omega \cap B_{(1 + \delta) r_0}$, and therefore that~$u(x) \le M \big( {|x| - r_0} \big) \le M |x - z|$ for every~$x \in \Omega \cap B_{(1 + \delta) r_0}$, which is~\eqref{ulessthanlinearnearbdry}.

We have established that~\eqref{lineargrowthtechine} holds true. In order to conclude the proof of the theorem, it thus suffices to show that
\begin{equation} \label{lineargrowthtechine2}
\| u_+ \|_{L^\infty(\Omega)} \le C \| f_+ \|_{L^\infty(\Omega)}.
\end{equation}
We do this by means of another barrier, arguing as in the proof of~\cite{GT01}*{Theorem~3.7}. Let~$R > 0$ be large enough to have that~$\Omega \subset B_R$ and~$\phi \coloneqq e^{2 \lambda R} - v_{\lambda, R}$, with~$v_{\lambda, R}$ as in~\eqref{vlemaux}. Clearly,~$\phi \ge 0$ in~$\R^n$. In addition, via the same computations performed in the proof of Lemma~\ref{lem:comparison0} it is immediate to see that
$$
p \, (-\Delta) \phi + q \, L_k \phi + g \cdot D\phi \ge 1 \quad \mbox{in } \Omega,
$$
provided~$\lambda > 0$ is chosen sufficiently large, in dependence of~$n$,~$s$,~$\kappa_1$,~$\kappa_2$,~$R$,~$\inf_\Omega p$, and~$\| g \|_{L^\infty(\Omega)}$ only. Then, the function~$w \coloneqq u - \| f_+ \|_{L^\infty(\Omega)} \, \phi$ satisfies
$$
\left\lbrace\begin{aligned}
p \, (-\Delta) w + q \, L_k w + g\cdot Dw &\le 0 && \mbox{in } \Omega, \\
w &\le 0 && \mbox{in } \R^n \setminus \Omega.
\end{aligned}\right.
$$
Invoking again Proposition~\ref{prop:comparison}, we infer that~$w \le 0$ in~$\Omega$, which gives~\eqref{lineargrowthtechine2}. The proof is thus complete.
\end{proof}

We point out that, when~$L_k$ is~$2s$-stable---\textit{i.e.}, when its kernel~$k$ is homogeneous---, it is possible to obtain~\eqref{eq:Lkge-} through a simpler computation based on the representation formula provided in~\cite{MR4038144}*{Lemma~2.4}.

\section{Existence, uniqueness and boundary regularity for~\texorpdfstring{\eqref{dir}}{dir}. Proof of Theorem~\texorpdfstring{\ref{mainlinearexthmm}}{mainlinearexthmm}}\label{sec:main}

We prove here Theorem~\ref{mainlinearexthmm}, concerning the existence, uniqueness, and regularity of solutions to the Dirichlet problem~\eqref{dir}. 
To establish it, we move the nonlocal term~$q \, L_k u$ to the right-hand side and run a fixed-point argument based on the solvability properties of the (standard) Laplacian.

To this aim, it is paramount to understand the regularity of the operator~$L_k$ applied to smooth functions which are only Lipschitz continuous across the boundary of~$\Omega$---recall that solutions of problem~\eqref{dir} (with, say,~$p, q \equiv 1$,~$g \equiv 0$, and $\kappa_1=\kappa_2$ in~\eqref{k}) are typically no better than Lipschitz in~$\R^n$, thanks, \textit{e.g.}, to the Hopf lemma of~\cite{BM21}*{Theorem~2.2} or~\cite{AC23}*{Theorem~1.2}.

In light of these observations, we proceed to deal separately with the two cases~$s \in \left( 0, \frac{1}{2} \right)$ and~$s \in \left[ \frac{1}{2}, 1 \right)$, since the regularity properties of the solutions, as well as the functional spaces used to measure them, change significantly.

\subsection{The case~\texorpdfstring{$s \in \left( 0, \frac{1}{2} \right)$}{0<s<1/2}.}

We prove here the following statement, which implies in particular Theorem~\ref{mainlinearexthmm} in the case~$s < \frac{1}{2}$.

\begin{proposition} \label{mainlinears<12prop}
Let~$\Omega \subset \R^n$ be a bounded open set with boundary of class~$C^{2, \alpha}$, for some~$\alpha \in (0, 1)$. Let~$k$ be a kernel satisfying~\eqref{k}, for some~$s \in \left( 0, \frac{1}{2} \right)$ and~$\kappa_2 \ge \kappa_1 > 0$. Let~$p, q \in C^\alpha(\overline\Omega)$ be two non-negative functions, with~$p$ satisfying~$\inf_\Omega p > 0$. Let~$f,g \in C^\alpha(\overline\Omega)$.

Then, problem~\eqref{dir} has a unique solution~$u \in C^2(\Omega) \cap C^0(\R^n)$. Moreover,~$u \in C^{2, \beta}(\overline\Omega)$, with~$\beta = \min \{ 1 - 2 s, \alpha \}$, and it satisfies
\begin{equation} \label{C2betaest}
\| u \|_{C^{2, \beta}(\overline\Omega)} \le C \| f \|_{C^\alpha(\overline\Omega)},
\end{equation}
for some constant~$C > 0$ depending only on~$n$,~$s$,~$\alpha$,~$\kappa_1$,~$\kappa_2$,~$\Omega$,~$\| p \|_{C^\alpha(\overline\Omega)}$,~$\inf_\Omega p$, and~$\| q \|_{C^\alpha(\overline\Omega)}$.
\end{proposition}

We can address the proof of Proposition~\ref{mainlinears<12prop} through unweighted H\"older spaces, since, when~$s < \frac{1}{2}$, the operator~$L_k$ maps Lipschitz functions to H\"older continuous ones. The following lemma provides a statement of this fact in full details.

\begin{lemma} \label{fraclapofus<12}
Let~$\Omega \subset \R^n$ be a bounded open set with Lipschitz boundary and~$k$ be a kernel satisfying~\eqref{k}, for some~$s \in \left( 0, \frac{1}{2} \right)$ and~$\kappa_2 \ge \kappa_1 > 0$. Let~$u \in C^{0, 1}(\overline\Omega) \cap C^0(\R^n)$ with~$u = 0$ in~$\R^n \setminus \Omega$. Then,~$L_k u \in C^{1 - 2 s}(\overline\Omega)$, with
\[
\| L_k u \|_{C^{1 - 2 s}(\overline\Omega)} \le C \| Du \|_{L^\infty(\overline\Omega)},
\]
for some constant~$C > 0$ depending only on~$n$,~$s$,~$\kappa_2$, and~$\textnormal{diam}(\Omega)$.
\end{lemma}
\begin{proof}
Set~$R \coloneqq \diam(\Omega)$. Using that~$u$ is globally Lipschitz in~$\R^n$ with~$\| Du \|_{L^\infty(\R^n)} = \| Du \|_{L^\infty(\Omega)}$, that~$|u(x)| \le \| Du \|_{L^\infty(\Omega)} \, d_x$ for~$x \in \Omega$, and that~$k$ satisfies~\eqref{k}, we compute
\begin{align*}
\big| L_k u(x) \big| & \le \int_{\Omega} |u(x) - u(w)| \, k(w - x) \, dw + |u(x)| \int_{\R^n \setminus \Omega} k(w - x) \, dw \\
& \le \kappa_2 \| D u \|_{L^\infty(\R^n)}
\int_{B_R} \frac{dz}{|z|^{n - 1 + 2 s}} 
+ \kappa_2 \| Du \|_{L^\infty(\Omega)} \, d_x \int_{\R^n \setminus B_{d_x}} \frac{dz}{|z|^{n + 2 s}} \\
& \le \mathcal{H}^{n - 1}(\partial B_1) \, \kappa_2 \| Du \|_{L^\infty(\Omega)}
\left( \frac{R^{1 - 2 s}}{1 - 2 s} + \frac{d_x^{1 - 2 s}}{2 s} \right) \\
& \le \frac{\mathcal{H}^{n - 1}(\partial B_1)}{2 s (1 - 2 s)}  \, \kappa_2 R^{1 - 2 s} \| Du \|_{L^\infty(\Omega)},
\end{align*}
for every~$x \in \Omega$. This gives the boundedness in~$\Omega$ of~$L_k u$. We now address its H\"older continuity. To do it, we first observe that, given~$x, y \in \Omega$, it holds
$$
\big| u(x) - u(x + z) - u(y) + u(y + z) \big| \le 2 \, \| Du \|_{L^\infty(\R^n)} \min \big\{ {|x - y|, |z|} \big\} \quad \mbox{for every } z \in \R^n.
$$
Hence, we have
\begin{align*}
\big| L_k u(x) - L_k u(y) \big| & \le \int_{\R^n} \big| {u(x) - u(x + z) - u(y) + u(y + z)} \big| \, k(z) \, dz \\
& \le 2 \kappa_2 \| Du \|_{L^\infty(\R^n)}
\left( \int_{B_{|x - y|}} \frac{dz}{|z|^{n - 1 + 2 s}} + |x - y| \int_{\R^n \setminus B_{|x - y|}} \frac{dz}{|z|^{n + 2 s}} \right) \\
& \le \frac{\mathcal{H}^{n - 1}(\partial B_1)}{s(1 - 2 s)} \, \kappa_2 \| Du \|_{L^\infty(\Omega)}
\,|x - y|^{1 - 2 s},
\end{align*}
and the proof is complete.
\end{proof}

Thanks to this result, we may directly address Proposition~\ref{mainlinears<12prop}.

\begin{proof}[Proof of Proposition~\ref{mainlinears<12prop}]
Of course, the uniqueness claim directly follows from the maximum principle of Proposition~\ref{prop:comparison}. As to the existence, we run a fixed-point argument in the Banach space
\[
X \coloneqq \Big\{ {u \in C^2(\overline{\Omega})} \cap C^0(\R^n) : u = 0 \mbox{ in } \R^n \setminus \Omega \Big\},
\]
endowed with the norm~$\| u \|_X \coloneqq \| u \|_{C^2(\overline\Omega)}$.

Let~${(-\Delta|_\Omega)}^{-1}$ be the inverse operator of the Laplacian coupled with homogeneous boundary conditions on~$\Omega$. That is, given~$h \in C^\beta(\overline\Omega)$ with~$\beta \coloneqq \min \{ 1 - 2 s, \alpha \} \in (0, 1)$, we indicate by~${(-\Delta|_\Omega)}^{-1} [h]$ the unique solution~$u \in C^{2, \beta}(\overline\Omega)$ of the Dirichlet problem
\[
\left\lbrace\begin{aligned}
- \Delta u &= h && \quad \text{in } \Omega, \\
u &= 0 && \quad \text{on } \partial \Omega.
\end{aligned}\right.
\]
Consider the affine map~$\widetilde{T}$ defined as
\begin{equation} \label{tildeTdef}
\widetilde{T}[u]\coloneqq {(-\lapl|_\Omega)}^{-1}\bigg[\frac{f-g\cdot Du - q\,L_k u}{p}\bigg] 
\qquad \mbox{for } u \in X.
\end{equation}
By Lemma~\ref{fraclapofus<12} and the standard Schauder theory for the Laplacian, we have
\begin{equation} \label{Ttildeest}
\begin{aligned}
\big\| {\widetilde{T}[u]} \big\|_{C^{2, \beta}(\overline\Omega)} & \le C \left\| \frac{f-g\cdot Du-q\,L_k u}{p} \right\|_{C^{\beta}(\overline\Omega)} \\
& \le C \left\| \frac{1}{p} \right\|_{C^\alpha(\overline\Omega)} \Big( {\| f \|_{C^\alpha(\overline\Omega)} 
+ \| g \|_{C^\alpha(\overline\Omega)} \| Du \|_{C^\alpha(\overline\Omega)}
+ \| q \|_{C^\alpha(\overline\Omega)} \big\| L_k u \big\|_{C^{1 - 2 s}(\overline\Omega)}} \Big) \\
& \le C \Big( {\| f \|_{C^\alpha(\overline\Omega)} + \| Du \|_{C^\alpha(\overline\Omega)}} \Big) \le C \Big( {\| f \|_{C^\alpha(\overline\Omega)} + \| u \|_X} \Big),
\end{aligned}
\end{equation}
where, from now on,~$C > 0$ denotes a general constant depending only on~$n$, $s$, $\alpha$, $\kappa_1$, $\kappa_2$, $\Omega$, $\| p \|_{C^\alpha(\overline\Omega)}$, $\inf_\Omega p$, $\| q \|_{C^\alpha(\overline\Omega)}$, and~$\| g \|_{C^\alpha(\overline\Omega)}$. Hence,~$\widetilde{T}$ is a continuous operator from~$X$ to the subspace~$C^{2, \beta}_0(\overline\Omega)$ of~$C^{2, \beta}(\overline\Omega)$ made up of those functions that have vanishing limits on~$\partial \Omega$.

Let now~$\iota: C^{2, \beta}_0(\overline\Omega) \to C^2_0(\overline{\Omega}) \coloneqq \big\{ {u \in C^2(\overline{\Omega}) : u = 0 \mbox{ on } \partial \Omega} \big\}$ be the inclusion map and~$e_\Omega$ be the trivial extension operator outside of~$\Omega$, \textit{i.e.},
\begin{equation} \label{extOmegadef}
e_\Omega w = \begin{cases}
w & \quad \mbox{in } \Omega, \\
0 & \quad \mbox{in } \R^n \setminus \Omega,
\end{cases}
\quad \mbox{for every measurable function~} w: \Omega \to \R.
\end{equation}
Let~$T: X \to X$ be defined by~$T \coloneqq e_\Omega \big( {\iota \circ \widetilde{T}} \big)$. Clearly, the solvability of problem~\eqref{dir} is equivalent to the existence of a fixed point for~$T$.

Thanks to~\eqref{Ttildeest}, the map~$T$ is continuous and satisfies
$$
\big\| {T[u]} \big\|_X \le C \Big( {\| f \|_{C^\alpha(\overline\Omega)} + \| u \|_X} \Big) \quad \mbox{for every } u \in X.
$$
Moreover, since~$\iota$ is compact,~$T$ is compact as well. Hence, we can show the existence of a fixed-point for~$T$ via the Leray-Schauder Theorem (see~\cite{GT01}*{Theorem 11.3}), provided we check that
\begin{align}\label{8492734}
\|v\|_X \le C \| f \|_{C^\alpha(\overline\Omega)} \text{ for every~$v\in X$ such that~$v = \lambda \, T[v]$ for some~$\lambda\in[0,1]$}.
\end{align}

To see this, we first remark that every such~$v$ is a~$C^2(\overline\Omega) \cap C^0(\R^n)$-solution of the Dirichlet problem
\begin{equation} \label{lambdavprob}
\left\lbrace\begin{aligned}
p \, (-\Delta) v + \lambda q \, L_k v + \lambda g \cdot Dv &= \lambda f && \mbox{in } \Omega, \\
v &= 0 && \mbox{on } \partial \Omega, \\
v &= 0 && \mbox{in } \R^n \setminus \overline{\Omega}.
\end{aligned}\right.
\end{equation}
Hence, by applying Theorem~\ref{lineargrowththm} to both~$v$ and~$-v$, we deduce that
\[
\| v \|_{L^\infty(\Omega)} \le \diam(\Omega) \, \| v \|_{C^0_{-1}(\Omega)} \le C \| f \|_{L^\infty(\Omega)}.
\]
Moreover, by~\eqref{Ttildeest},
\[
\| v \|_{C^{2, \beta}(\overline\Omega)} = \lambda \, \big\| {\widetilde{T}[v]} \big\|_{C^{2, \beta}(\overline\Omega)} \le C \Big( { \| f \|_{C^{\alpha}(\overline\Omega)} + \| v \|_{C^2(\overline\Omega)}} \Big).
\]
By combining the last two estimates with the interpolation inequality~$\| w \|_{C^2(\overline\Omega)} \le C \| w \|_{L^\infty(\Omega)}^{\frac{\beta}{2 + \beta}} \| w \|_{C^{2, \beta}(\overline\Omega)}^{\frac{2}{2 + \beta}}$---see, \textit{e.g.},~\cite{L95}*{Definition~1.1.1 and Proposition~1.1.3}---, we obtain that
\[
\| v \|_{C^{2, \beta}(\overline\Omega)} \le C_0 \left( {\| f \|_{C^\alpha(\overline\Omega)} + \| f \|_{L^\infty(\Omega)}^{\frac{\beta}{2 + \beta}} \| v \|_{C^{2, \beta}(\overline\Omega)}^{\frac{2}{2 + \beta}}} \right),
\]
for some constant~$C_0 > 0$ depending on the same quantities as~$C$. Claim~\eqref{8492734} follows at once, thanks to the weighted Young's inequality 
\[
\| f \|_{L^\infty(\Omega)}^\frac{\beta}{2 + \beta} \| v \|_{C^{2, \beta}(\overline\Omega)}^{\frac{2}{2 + \beta}} \leq C \| f \|_{L^\infty(\Omega)} + \frac1{2C_0} \| v \|_{C^{2, \beta}(\overline\Omega)}
\]
and the fact that~$\| v \|_X \le \| v \|_{C^{2, \beta}(\overline\Omega)}$.

We thus proved that there exists a unique~$C^2(\overline\Omega) \cap C^0(\R^n)$-solution to~\eqref{dir}. Since the~$C^{2, \beta}$ estimate~\eqref{C2betaest} also immediately follows from the previous calculations, the proof is complete.
\end{proof}

\subsection{The case~\texorpdfstring{$s \in \left[ \frac{1}{2}, 1 \right)$}{1/2<s<1}.}

We establish in this subsection the following proposition, which gives Theorem~\ref{mainlinearexthmm} when~$s \ge \frac{1}{2}$.

\begin{proposition} \label{mainlinears>12prop}
Let~$\Omega \subset \R^n$ be a bounded open set with boundary of class~$C^{2, \alpha}$, for some~$\alpha \in (0, 1)$. Let~$k$ be a kernel satisfying~\eqref{k}, for some~$s \in \left[ \frac{1}{2}, 1 \right)$ and~$\kappa_2 \ge \kappa_1 > 0$. Let~$p, q \in C^\alpha(\overline\Omega)$ be two non-negative functions, with~$p$ satisfying~$\inf_\Omega p > 0$. Let~$f,g \in C^\alpha(\overline\Omega)$.

Then, problem~\eqref{dir} has a unique solution~$u \in C^2(\Omega) \cap C^0(\R^n)$. Moreover,~$u$ has the following regularity properties:
\begin{itemize}[leftmargin=*]
\item If~$s \in \left( \frac{1}{2}, 1 \right)$, then~$u \in C^{2, \beta}_{-1,2s - 1}(\Omega) \subset C^{1, 2 - 2 s}(\overline\Omega)$, with~$\beta \coloneqq \min \big\{ {2 - 2 s, \alpha} \big\}$, and it satisfies
\[
\| u \|_{C^{2, \beta}_{-1,2s - 1}(\Omega)} + \| u \|_{C^{1, 2 - 2 s}(\overline\Omega)} \le C \| f \|_{C^\alpha(\overline\Omega)},
\]
for some~$C > 0$ depending only on~$n$,~$s$,~$\alpha$,~$\kappa_1$,~$\kappa_2$,~$\Omega$,~$\| p \|_{C^\alpha(\overline\Omega)}$,~$\inf_\Omega p$,~$\| q \|_{C^\alpha(\overline\Omega)}$, and~$\| g \|_{C^\alpha(\overline\Omega)}$.
\item If~$s = \frac{1}{2}$, then~$u \in C^{2, \alpha}_{-1,\varepsilon}(\Omega) \subset C^{1, 1 - \varepsilon}(\overline\Omega)$ for every~$\varepsilon \in (0, 1)$ and it satisfies
\[
\| u \|_{C^{2, \alpha}_{-1,\varepsilon}(\Omega)} + \| u \|_{C^{1, 1- \varepsilon}(\overline\Omega)} \le C_\varepsilon \| f \|_{C^\alpha(\overline\Omega)},
\]
for some~$C_\varepsilon > 0$ depending only on~$n$,~$s$,~$\alpha$,~$\kappa_1$,~$\kappa_2$,~$\Omega$,~$\| p \|_{C^\alpha(\overline\Omega)}$,~$\inf_\Omega p$,~$\| q \|_{C^\alpha(\overline\Omega)}$,~$\| g \|_{C^\alpha(\overline\Omega)}$, and~$\varepsilon$.
\end{itemize}
\end{proposition}

In order to prove this result, we need to first investigate the behavior of~$L_k$ when applied to functions which qualitatively resemble the sought solution~$u$ of~\eqref{dir}. In light of the Hopf lemma established in~\cite{BM21} or~\cite{AC23}---at least when~$p, q \equiv 1$, $g\equiv 0$, and~$f \ge 0$---the graph of~$u$ presents corner points at the boundary of~$\Omega$. When~$s \ge \frac{1}{2}$,~$L_k u$ will then typically blow-up at those points, no matter how smooth the function is inside~$\Omega$. The next two lemmas quantify the blow-up rate in the two cases~$s \in \left( \frac{1}{2}, 1 \right)$ and~$s = \frac{1}{2}$.

\begin{lemma} \label{fraclapests>12lem}
Let~$\Omega \subset \R^n$ be a bounded open set with Lipschitz boundary and~$k$ be a kernel satisfying~\eqref{k}, for some~$s \in \left( \frac{1}{2}, 1 \right)$ and~$\kappa_2 \ge \kappa_1 > 0$. Let~$u \in C^2_{-1}(\Omega) \cap C^0(\R^n)$ with~$u = 0$ in~$\R^n \setminus \Omega$. Then,~$L_k u \in C^\beta_{2 s - 1}(\Omega)$ for every~$\beta \in (0, 2 - 2 s]$, with
\[
\big\| L_k u \big\|_{C^\beta_{2 s - 1}(\Omega)} \le C \| u \|_{C^2_{-1}(\Omega)},
\]
for some constant~$C > 0$ depending only on~$n$,~$s$,~$\kappa_2$, and~$\textnormal{diam}(\Omega)$.
\end{lemma}
\begin{proof}
We begin with the weighted~$L^\infty$ estimate for~$L_k u$, \textit{i.e.}, we claim that
\begin{equation} \label{weightedLinftyfraclap}
\big| L_k u(x) \big| \le C \| u \|_{C^2_{-1}(\Omega)} \, d_x^{1 - 2 s} \qquad \mbox{for all } x \in \Omega,
\end{equation}
for some constant~$C > 0$ depending only on~$n$,~$s$, and~$\kappa_2$. For~$x \in \Omega$, we write~$L_k u(x) = I_\rho(x) + O_\rho(x)$, with
\begin{align*}
I_\rho(x) & \coloneqq \pv \int_{B_\rho} \big( {u(x) - u(x + z)} \big) k(z) \, dz, \\
O_\rho(x) & \coloneqq \int_{\R^n \setminus B_\rho} \big( {u(x) - u(x + z)} \big) k(z) \, dz,
\end{align*}
and~$\rho > 0$. Recalling assumption~\eqref{k} on~$k$ and choosing~$\rho \coloneqq d_x / 2$, we compute
\begin{align*}
\left| I_\rho(x) \right| & = \left| \int_{B_\rho} \big( {u(x + z) - u(x) - D u(x) \cdot z} \big) k(z) \, dz \right| \le \kappa_2 \int_{B_\rho} \frac{\big| {u(x + z) - u(x) - D u(x) \cdot z} \big|}{|z|^{n + 2 s}} \, dz \\
& \le \kappa_2 \, \rho^{- 1} \sup_{y \in \Omega} \Big( {d_y |D^2 u(y)|} \Big)  \int_{B_\rho} \frac{dz}{|z|^{n - 2 + 2 s}} \le \frac{\mathcal{H}^{n - 1}(\partial B_1)}{2(1 - s)} \, \kappa_2 \, \| u \|_{C^2_{-1}(\Omega)} \, \rho^{1 - 2 s}.
\end{align*}
On the other hand, exploiting the fact that~$u$ is globally Lipschitz continuous with~$\| Du \|_{L^\infty(\R^n)} = \| Du \|_{L^\infty(\Omega)}$, we simply have
\begin{align*}
\left| O_\rho(x) \right| & \le \kappa_2 \int_{\R^n \setminus B_\rho} \frac{|u(x) - u(x+z)|}{|z|^{n + 2 s}} \; dz \\
& \le \kappa_2 \| Du \|_{L^\infty(\R^n)} \int_{\R^n \setminus B_\rho} \frac{dz}{|z|^{n - 1 + 2 s}} 
\le \frac{\mathcal{H}^{n - 1}(\partial B_1)}{2 s - 1} \, \kappa_2
\, \| u \|_{C^2_{-1}(\Omega)} \, \rho^{1 - 2 s},
\end{align*}
and claim~\eqref{weightedLinftyfraclap} follows.

We now move to the weighted H\"older continuity estimate. Thanks to~\eqref{weightedLinftyfraclap} and symmetry considerations, it is clear that we may restrict to proving that
\begin{equation} \label{weightedHolfraclap}
\big| {L_k u(x) - L_k u(y)} \big| \le C \| u \|_{C^2_{-1}(\Omega)} \frac{|x - y|^\beta}{d_x^{2 s - 1 + \beta}} \quad \mbox{for all } x, y \in \Omega \mbox{ s.t.~} d_x \le d_y \mbox{ and } |x - y| \le \frac{d_x}{4 },
\end{equation}
for some constant~$C > 0$ depending only on~$n$,~$s$, and~$\kappa_2$. To see this, we let as before~$\rho \coloneqq d_x / 2$ and observe that
\[
\big| {u(x) - u(x + z) - u(y) + u(x + z)} \big| \le 2 \| Du \|_{L^\infty(\Omega)} \min \big\{ {|z|, |x - y|} \big\} \quad \mbox{for all } z \in \R^n.
\]
Using this, we compute
\begin{equation} \label{OHolderest}
\begin{aligned}
\left| O_\rho(x) - O_\rho(y) \right| & \le \kappa_2
\int_{\R^n \setminus B_\rho} \frac{\big| {u(x) - u(x + z) - u(y) + u(x + z)} \big|}{|z|^{n + 2 s}} \, dz \\
& \le 2 \kappa_2 \, \| Du \|_{L^\infty(\Omega)} |x - y| \int_{\R^n \setminus B_\rho} \frac{dz}{|z|^{n + 2 s}} \\
& \le \frac{\mathcal{H}^{n - 1}(\partial B_1)}{s} \, \kappa_2 \, \| u \|_{C^2_{-1}(\Omega)} \, \frac{|x - y|}{\rho^{2 s}}.
\end{aligned}
\end{equation}
To estimate the difference between the~$I_\rho$ terms, we use the second order estimate
\[
\big| u(x + z) - u(x) - Du(x) \cdot z - u(y + z) + u(y) + Du(y) \cdot z \big| \le 2 \| u \|_{C^2_{-1}(\Omega)} \, \rho^{-1} |z| \min \big\{ {|z|, |x - y|} \big\},
\]
valid for all~$z \in B_\rho$. Thanks to this, we get
\begin{align*}
\left| I_\rho(x) - I_\rho(y) \right| & \le \kappa_2 \int_{B_\rho} \frac{\big| u(x + z) - u(x) - Du(x) \cdot z - u(y + z) + u(y) + Du(y) \cdot z \big|}{|z|^{n + 2 s}} \, dz \\
& \le 2 \kappa_2 \, \| u \|_{C^2_{-1}(\Omega)} \rho^{-1} \left( \int_{B_{|x - y|}} \frac{dz}{|z|^{n - 2 + 2 s}} + |x - y| \int_{B_\rho \setminus B_{|x - y|}} \frac{dz}{|z|^{n - 1 + 2 s}} \right) \\
& \le \frac{\mathcal{H}^{n - 1}(\partial B_1)}{(1 - s) (2 s - 1)} \, \kappa_2 \,
\| u \|_{C^2_{-1}(\Omega)} \, \rho^{-1}  |x - y|^{2 - 2 s}.
\end{align*}
This and~\eqref{OHolderest} immediately lead to~\eqref{weightedHolfraclap}. The proof of the lemma is thus complete.
\end{proof}

When~$s = \frac{1}{2}$,~$L_k u$ could develop logarithmic singularities at the boundary. With no aim to describe this behavior with such precision, we state the following result, which can be proved with minor modifications to the computations presented for Lemma~\ref{fraclapests>12lem}.

\begin{lemma} \label{fraclapests=12lem}
Let~$\Omega \subset \R^n$ be a bounded open set with Lipschitz boundary and~$k$ be a kernel satisfying~\eqref{k}, with~$s = \frac{1}{2}$ and for some~$\kappa_2 \ge \kappa_1 > 0$. Let~$u \in C^2_{-1}(\Omega) \cap C^0(\R^n)$ with~$u = 0$ in~$\R^n \setminus \Omega$. Then,~$L_k u \in C^\beta_{\varepsilon}(\Omega)$ for every~$\beta, \varepsilon \in (0, 1)$, with
\[
\big\| {L_k u} \big\|_{C^\beta_{\varepsilon}(\Omega)} \le C_\varepsilon \| u \|_{C^2_{-1}(\Omega)},
\]
for some constant~$C_\varepsilon > 0$ depending only on~$n$,~$\kappa_2$,~$\beta$,~$\textnormal{diam}(\Omega)$, and~$\varepsilon$.
\end{lemma}

Next, we include the following result, which addresses the regularity of the gradient of functions belonging to the space~$C^2_{-1}(\Omega)$.

\begin{lemma} \label{lem:gradmapprop}
Let~$\Omega \subset \R^n$ be a bounded open set with Lipschitz boundary and~$u \in C^2_{-1}(\Omega)$. Then,~$D_j u \in C^\beta_\gamma(\Omega)$ for every~$j \in \{ 1, \ldots, n \}$ and~$\beta, \gamma \in (0, 1)$, with
$$
\big\| {D_j u} \big\|_{C^\beta_\gamma(\Omega)} \le C \| u \|_{C^2_{-1}(\Omega)},
$$
for some constant~$C > 0$ depending only on~$n$,~$\beta$,~$\gamma$, and~$\textnormal{diam}(\Omega)$.
\end{lemma}
\begin{proof}
On the one hand, we have that
\begin{equation} \label{partjuLinftyest}
\sup_{x \in \Omega} \Big( {d_x^\gamma \, |D_j u(x)|} \Big) \le \textnormal{\diam}(\Omega)^\gamma \| Du \|_{L^\infty(\Omega)} \le \textnormal{\diam}(\Omega)^\gamma \| u \|_{C^2_{-1}(\Omega)}.
\end{equation}
On the other hand, given any~$x \in \Omega$ and~$y \in B_{\frac{d_x}{2}}(x)$, by Lagrange's mean value theorem there exists a point~$z = z(x, y) \in B_{\frac{d_x}{2}} \! (x)$ such that
$$
\left| D_j u(x) - D_j u(y) \right| = \big| {D D_j u(z) \cdot (x - y)} \big| \le |x - y| \sup_{w \in B_{\frac{d_x}{2}} \! (x)} |D^2 u(w)| \le 2 \| u \|_{C^2_{-1}(\Omega)} d_x^{-1} |x - y|.
$$
Hence,
\begin{align*}
\sup_{\substack{x, y \in \Omega\\ 0 < |x - y| < \frac{d_x}{2}}} \!\! \left( d_{x, y}^{\beta + \gamma} \, \frac{|D_j u(x) - D_j u(y)|}{|x - y|^\beta} \right) & = \! \sup_{\substack{x, y \in \Omega\\ 0 < |x - y| < \frac{d_x}{2}}} \!\! \left( d_{x, y}^{\beta + \gamma} \left( \frac{|D_j u(x) - D_j u(y)|}{|x - y|}\right)^{\!\beta} |D_j u(x) - D_j u(y)|^{1 - \beta} \right) \\
& \le 2 \, \textnormal{\diam}(\Omega)^\gamma \| u \|_{C^2_{-1}(\Omega)}^\beta \| Du \|_{L^\infty(\Omega)}^{1 - \beta} \le 2 \, \textnormal{\diam}(\Omega)^\gamma \| u \|_{C^2_{-1}(\Omega)}.
\end{align*}
Since the supremum of the above quantity with respect to points~$x, y \in \Omega$ with~$|x - y| \ge \frac{d_x}{2}$ can be easily estimated using~\eqref{partjuLinftyest}, the proof of the lemma is complete.
\end{proof}

With these preparatory results in hand, we are ready to show the claims contained in Proposition~\ref{mainlinears>12prop}.

\begin{proof}[Proof of Proposition~\ref{mainlinears>12prop}]
The general strategy of the proof is the same adopted to establish Proposition~\ref{mainlinears<12prop}. The differences lie in the functional spaces that we employ and the estimates that we take advantage of.

First of all, we write
\[
\beta \coloneqq \min \Big\{ \alpha, 2 - 2 s \Big\} \quad \mbox{and} \quad \gamma \coloneqq \begin{dcases}
2 s - 1 & \quad \mbox{if } s \in \left( \frac{1}{2}, 1 \right) \!, \\
\varepsilon & \quad \mbox{if } s = \frac{1}{2},
\end{dcases}
\]
where, in the case~$s = \frac{1}{2}$,~$\varepsilon$ is any fixed number in~$(0, 1)$. In what follows,~$C$ indicates a general positive constant depending at most on~$n$,~$s$,~$\alpha$,~$\kappa_1$,~$\kappa_2$,~$\Omega$,~$\| p \|_{C^\alpha(\overline\Omega)}$,~$\inf_\Omega p$,~$\| q \|_{C^\alpha(\overline\Omega)}$,~$\| g \|_{C^\alpha(\overline\Omega)}$, and also on~$\varepsilon$ when~$s = \frac{1}{2}$.

As in the proof of Proposition~\ref{mainlinears<12prop}, we plan to obtain the existence of a solution of~\eqref{dir} through the Leray-Schauder fixed-point theorem applied to an affine endomorphism~$T$ in a Banach space~$X$. The domain of~$T$ is now
\[
X \coloneqq \Big\{ { u \in C^2_{-1}(\Omega) \cap C^0(\R^n) : u=0 \in \R^n\setminus\Omega} \Big\},
\]
with~$\|u\|_{X} \coloneqq \|u\|_{C^2_{-1}(\Omega)}$, while the map~$T$ is formally defined as before. That is,~$T \coloneqq e_\Omega \left( \iota \circ \widetilde{T} \right)$, where~$\widetilde{T}$ is exactly as in~\eqref{tildeTdef},~$e_\Omega$ is the extension operator~\eqref{extOmegadef}, and~$\iota$ is now the inclusion of~$C^{2, \beta}_{-1,\gamma}(\Omega)$ into~$C_{-1}^{2}(\Omega)$. Of course, this map is well-defined provided we check that~$\widetilde{T}$ maps~$X$ in~$C^{2, \beta}_{-1,\gamma}(\Omega)$. This is a consequence of Lemma~\ref{fraclapests>12lem} (when~$s > \frac{1}{2}$) or Lemma~\ref{fraclapests=12lem} (when~$s = \frac{1}{2}$), Lemma~\ref{lem:gradmapprop},  and Theorem~\ref{mainC2alphaexistprop}, whose combined use yields in particular the continuity of~$\widetilde{T}$ and~$T$ via the estimates
\begin{equation} \label{Tuests>12}
\big\| {T[u]} \big\|_X \le C \big\| {\widetilde{T}[u]} \big\|_{C^{2, \beta}_{-1,\gamma}(\Omega)} \le C \Big( {\| f \|_{C^\alpha(\overline\Omega)} + \| u \|_X} \Big) \quad \mbox{for every } u \in X.
\end{equation}

Next, we observe that~$T$ is compact, as a consequence of the compactness of~$\iota$ warranted by Lemma~\ref{compactembedlem}. Hence, in order to obtain a fixed-point for~$T$ (and thus a solution of~\eqref{dir}), we only need to check the validity of the \textit{a priori} estimate~\eqref{8492734}. As in the proof of Proposition~\ref{mainlinears<12prop}, any~$v \in X$ satisfying~$v = \lambda \, T[v]$ for some~$\lambda \in [0, 1]$ solves in particular the Dirichlet problem~\eqref{lambdavprob}. By applying to it~\eqref{Tuests>12}, the interpolation inequality of Lemma~\ref{interpollem}, and the weighted~$L^\infty$ estimate of Theorem~\ref{lineargrowththm}, we deduce that
\begin{align*}
\| v \|_{C^{2, \beta}_{-1,\gamma}(\Omega)} & = \lambda \, \big\| {\widetilde{T}[v]} \big\|_{C^{2, \beta}_{-1,\gamma}(\Omega)} \le C \Big( {\| f \|_{C^\alpha(\overline\Omega)} + \| v \|_{C^2_{-1}(\Omega)}} \Big) \\
& \le C \left( {\| f \|_{C^\alpha(\overline\Omega)} + \| v \|_{C^0_{-1}(\Omega)}^{\frac{\beta}{2 (1 + \beta)}} \| v \|_{C^{2, \beta}_{-1}(\Omega)}^{\frac{2 + \beta}{2 (1 + \beta)}}} \right) \le C \left( {\| f \|_{C^\alpha(\overline\Omega)} + \| f \|_{L^\infty(\Omega)}^{\frac{\beta}{2 (1 + \beta)}} \| v \|_{C^{2, \beta}_{-1,\gamma}(\Omega)}^{\frac{2 + \beta}{2 (1 + \beta)}}} \right),
\end{align*}
from which claim~\eqref{8492734} follows after an application of the weighted Young's inequality.

We have thus proved that problem~\eqref{dir} admits a solution~$u \in C^{2, \beta}_{-1,\gamma}(\Omega)$. Of course, its uniqueness is granted by Proposition~\ref{prop:comparison}, while its unweighted global fractional regularity follows from the embedding of Lemma~\ref{C2alphaisunweightedlem}.
\end{proof}

\section{Sharpness of the boundary regularity. Proof of Theorem~\ref{counterthm}}\label{sec:counter}

Given~$s \in (0, 1)$, we consider the one-dimensional fractional Laplace operator
$$
(-\Delta)^s u(x) \coloneqq C_s \, \pv \int_\R \frac{u(x) - u(y)}{|x - y|^{1 + 2 s}} \, dy,
$$
with~$C_s = \frac{2^{2 s} \Gamma \left( \frac{1 + 2 s}{2} \right)}{\sqrt{\pi} \Gamma (2 - s)} s (1 - s)$, as well as the mixed local-nonlocal operator
$$
- u'' + (-\Delta)^s u.
$$

In order to construct~$u_k$ and~$f_k$ as sought in the statement of Theorem~\ref{counterthm}, we consider the following auxiliary functions. Given~$\alpha \in (0, +\infty)$ and~$j \in \N$, let~$u_{\alpha,0},u_{\alpha, j} \in L^\infty(\R) \cap C^0(\R) \cap C^\infty((0, 1))$ be the functions defined by
$$
u_{\alpha, 0}(x) \coloneqq \begin{cases}
0 & \quad \mbox{if } x \in (-\infty, 0], \\
x^\alpha  & \quad \mbox{if } x \in (0, 1), \\
1 & \quad \mbox{if } x \in [1, +\infty),
\end{cases}
\qquad
u_{\alpha, j}(x) \coloneqq \begin{cases}
0 & \quad \mbox{if } x \in (-\infty, 0]\cup[1,+\infty), \\
x^\alpha (\log x)^j  & \quad \mbox{if } x \in (0, 1).
\end{cases}
$$
In the following, in order to reduce and shorten the notation, we are going to write~$\log^j\!x$ instead of~$(\log x)^j$.
In the next lemma, we compute the fractional Laplacian of these functions inside the interval~$\left( 0, \frac{1}{2} \right)$.

\begin{lemma} \label{buildingblocklem}
Let~$s \in (0, 1)$,~$\alpha \in (0, +\infty)$, and~$j \in \N \cup \{ 0 \}$. Then, there exist constants~$a_{\alpha, j}^{(0)}, \ldots, a_{\alpha, j}^{(j + 1)} \in \R$ and a function~$f_{\alpha, j} \in C^\infty \! \left(  \left[ 0, \frac{1}{2} \right] \right)$ such that
\begin{equation} \label{fraclapofuraz}
(- \Delta)^s u_{\alpha, j}(x) = x^{\alpha - 2 s} \sum_{k = 0}^{j + 1} a_{\alpha, j}^{(k)} \,  \log^k \! x + f_{\alpha, j}(x) \qquad \mbox{for every } x \in \left( 0, \frac{1}{2} \right).
\end{equation}
Furthermore,~$a_{\alpha, j}^{(j + 1)} = 0$ when~$\alpha - 2 s \notin \N \cup \{ 0 \}$ and~$a_{\alpha, j}^{(0)} = 0$ when~$\alpha - 2 s \in \N \cup \{ 0 \}$. 
In addition,~$a_{1, 0}^{(0)} \ne 0$ if~$s \ne \frac{1}{2}$ and~$a_{1, 0}^{(1)} \ne 0$ if~$s = \frac{1}{2}$.
\end{lemma}
\begin{proof}
Changing variables appropriately, for~$x \in \left( 0, \frac{1}{2} \right)$ we compute
\begin{equation} \label{techcompujk}
\begin{aligned}
C_s^{-1} (- \Delta)^s u_{\alpha, j}(x)
& = \frac{x^{\alpha - 2 s} \log^j \! x}{2 s} + x^{\alpha - 2 s} \, \pv \int_0^2 \frac{\log^j \! x - t^{\alpha} \log^j (x t)}{|1 - t|^{1 + 2 s}} \, dt + \frac{x^{\alpha - 2 s} \log^j \! x}{2 s} \\
& \quad - \frac{1}{2 s} \frac{x^{\alpha} \log^j \! x}{(1 - x)^{2 s}} - x^{\alpha - 2 s} \int_2^{\frac{1}{x}} \frac{t^{\alpha} \log^j (x t)}{(t - 1)^{1 + 2 s}} \, dt + \frac{1}{2 s} \frac{x^{\alpha} \log^j \! x}{(1 - x)^{2 s}} - \frac{1}{2 s} \frac{\delta_{j 0}}{(1 - x)^{2 s}} \\
& = \left\{ \frac{1}{s} + \pv \int_0^2 \frac{1 - t^\alpha}{|1 - t|^{1 + 2 s}} \, dt \right\} x^{\alpha - 2 s} \log^j \! x \\
& \quad - \sum_{k = 0}^{j - 1} \left\{ \binom {j}{k} \, \pv \int_0^2 \frac{t^\alpha \log^{j - k} \! t}{|1 - t|^{1 + 2 s}} \, dt \right\} x^{\alpha - 2 s} \log^k \! x \\
& \quad - \sum_{k = 0}^j \left\{ (-1)^{j - k} \binom{j}{k} \int_x^{\frac{1}{2}} \frac{v^{2 s - 1 - \alpha} \log^{j - k} \! v}{(1 - v)^{1 + 2 s}} \, dv \right\}  x^{\alpha - 2 s} \log^k \! x - \frac{1}{2 s} \frac{\delta_{j 0}}{(1 - x)^{2 s}}.
\end{aligned}
\end{equation}
In order to expand in~$x$ the integral function appearing on the last line, we observe that
$$
\frac{1}{(1 - v)^{1 + 2 s}} = \sum_{i = 0}^\infty \binom{i + 2 s}{i} v^i, \quad \mbox{where, for~$\beta \in \R$, } \binom{\beta}{i} \coloneqq \begin{dcases}
1 & \mbox{if } i = 0, \\
\frac{1}{i!} \prod_{\ell = 0}^{i - 1} \left( \beta - \ell \right) & \mbox{if } i \in \N.
\end{dcases}
$$
Hence, by Lebesgue's dominated convergence theorem,
$$
\int_x^{\frac{1}{2}} \frac{v^{2 s - 1 - \alpha} \log^{j - k} \! v}{(1 - v)^{1 + 2 s}} \, dv = \sum_{i = 0}^\infty \binom{i + 2 s}{i} \int_x^{\frac{1}{2}} v^{2 s - 1 - \alpha + i} \log^{j - k} \! v \, dv.
$$

To evaluate the last integral, we distinguish between the two cases~$\alpha - 2 s \in \N \cup \{ 0\}$ and~$\alpha - 2 s \notin \N \cup \{ 0 \}$. In the first case, there exists a unique~$i_\star \in \N \cup \{ 0 \}$ such that~$2 s - \alpha + i_\star = 0$. As a result, after a few integration by parts, we find that
\begin{align*}
\int_x^{\frac{1}{2}} \frac{v^{2 s - 1 - \alpha} \log^{j - k} \! v}{(1 - v)^{1 + 2 s}} \, dv & = (-1)^{j - k} \sum_{i \in \left( \N \cup \{ 0 \} \right) \setminus \{ i_\star \}} \binom{i + 2 s}{i} 2^{\alpha - 2 s - i} \sum_{\ell = 0}^{j - k} \frac{(j - k)_\ell \log^{j - k - \ell} \! 2}{(2 s - \alpha + i)^{\ell + 1}} \\
& \quad - \sum_{i \in \left( \N \cup \{ 0 \} \right) \setminus \{ i_\star \}} \binom{i + 2 s}{i} x^{2 s - \alpha + i} \sum_{\ell = 0}^{j - k} \frac{(-1)^\ell (j - k)_\ell}{(2 s - \alpha + i)^{\ell + 1}} \log^{j - k - \ell} \! x \\
& \quad + \binom{i_\star + 2 s}{i_\star} \left( \frac{(-1)^{j - k + 1} \log^{j - k + 1}  \! 2}{j - k + 1} - \frac{\log^{j - k + 1} \! x}{j - k + 1} \right),
\end{align*}
where~$(m)_\ell = \frac{m!}{(m - \ell)!}$ indicates the descending Pochhammer symbol. Combining this with~\eqref{techcompujk}, we get
\begin{align*}
& C_s^{-1} (- \Delta)^s u_{\alpha, j}(x) \\
& \hspace{5pt} = \left\{ \frac{1}{s} + \pv \int_0^2 \frac{1 - t^\alpha}{|1 - t|^{1 + 2 s}} \, dt \right\} x^{\alpha - 2 s} \log^j \! x - \sum_{k = 0}^{j - 1} \left\{ \binom{j}{k} \, \pv \int_0^2 \frac{t^\alpha \log^{j - k} \! t}{|1 - t|^{1 + 2 s}} \, dt \right\} x^{\alpha - 2 s} \log^k \! x \\
& \hspace{5pt} \quad - \sum_{k = 0}^j \binom{j}{k} \left\{ \sum_{i \in \left( \N \cup \{ 0 \} \right) \setminus \{ i_\star \}} \binom{i + 2 s}{i} 2^{\alpha - 2 s - i} \sum_{\ell = 0}^{j - k} \frac{(j - k)_\ell \log^{j - k - \ell} \! 2}{(2 s - \alpha + i)^{\ell + 1}} - \binom{i_\star + 2 s}{i_\star} \frac{\log^{j - k + 1}  \! 2}{j - k + 1} \right\}  x^{\alpha - 2 s} \log^k \! x \\
& \hspace{5pt} \quad + \sum_{k = 0}^j \left\{ (-1)^{j - k} \binom{j}{k} \sum_{i \in \left( \N \cup \{ 0 \} \right) \setminus \{ i_\star \}} \binom{i + 2 s}{i} \sum_{\ell = 0}^{j - k} \frac{(-1)^\ell (j - k)_\ell}{(2 s - \alpha + i)^{\ell + 1}} x^i \log^{j - \ell} \! x \right\} \\
& \hspace{5pt} \quad + \left\{ \binom{i_\star + 2 s}{i_\star} \sum_{\ell = 0}^j \binom{j}{\ell} \frac{(-1)^{j - \ell}}{j - \ell + 1} \right\} x^{\alpha - 2 s} \log^{j + 1} \! x - \frac{1}{2 s} \frac{\delta_{j 0}}{(1 - x)^{2 s}}.
\end{align*}
Notice now that, by changing indices as~$m = j - k$ and exchanging the order of summation,
\begin{align*}
& \sum_{k = 0}^j \left\{ (-1)^{j - k} \binom{j}{k} \sum_{i \in \left( \N \cup \{ 0 \} \right) \setminus \{ i_\star \}} \binom{i + 2 s}{i} \sum_{\ell = 0}^{j - k} \frac{(-1)^\ell (j - k)_\ell}{(2 s - \alpha + i)^{\ell + 1}} x^i \log^{j - \ell} \! x \right\} \\
& = \sum_{i \in \left( \N \cup \{ 0 \} \right) \setminus \{ i_\star \}} \binom{i + 2 s}{i} \sum_{\ell = 0}^j \left\{ \sum_{m = \ell}^j \binom{j}{j - m} (-1)^{m - \ell} (m)_\ell \right\} \frac{x^i \log^{j - \ell} \! x}{(2 s - \alpha + i)^{\ell + 1}} \\
& = j! \sum_{i \in \left( \N \cup \{ 0 \} \right) \setminus \{ i_\star \}} \binom{i + 2 s}{i} \frac{x^i}{(2 s - \alpha + i)^{j + 1}},
\end{align*}
where for the last identity we used that~$\binom{j}{j - m} (m)_\ell = \frac{j!}{m! (j - m)!} \frac{m!}{(m - \ell)!} = \frac{j!}{(j - \ell)!} \binom{j - \ell}{m - \ell}$ to deduce that, shifting the index of summation,
\begin{align*}
\sum_{m = \ell}^j \binom{j}{j - m} (-1)^{m - \ell} (m)_\ell & = \frac{j!}{(j - \ell)!} \sum_{k = 0}^{j - \ell} \binom{j - \ell}{k} (-1)^k 1^{j - \ell - k} = \frac{j!}{(j - \ell)!} (1 - 1)^{j - \ell} = \delta_{j \ell} \, j!.
\end{align*}
Accordingly, the claim of the lemma follows with
\begin{align*}
a_{\alpha, j}^{(0)} =\ & 0, \\
a_{\alpha, j}^{(k)} =\ & - C_s \binom{j}{k} \Bigg\{ \pv \int_0^2 \frac{t^\alpha \log^{j - k} \! t}{|1 - t|^{1 + 2 s}} \, dt + \sum_{i \in \left( \N \cup \{ 0 \} \right) \setminus \{ i_\star \}} \binom{i + 2 s}{i} 2^{\alpha - 2 s - i} \sum_{\ell = 0}^{j - k} \frac{(j - k)_\ell \log^{j - k - \ell} \! 2}{(2 s - \alpha + i)^{\ell + 1}}  \\
& -\binom{i_\star + 2 s}{i_\star} \frac{\log^{j - k + 1}  \! 2}{j - k + 1} \Bigg\}, \\
a_{\alpha, j}^{(j)} =\ & C_s \Bigg\{ \frac{1}{s} + \pv \int_0^2 \frac{1 - t^\alpha}{|1 - t|^{1 + 2 s}} \, dt - \! \sum_{i \in \left( \N \cup \{ 0 \} \right) \setminus \{ i_\star \}} \! \binom{i + 2 s}{i} \frac{2^{\alpha - 2 s - i}}{2 s - \alpha + i} + \binom{i_\star + 2 s}{i_\star} \log 2 \Bigg\}, \\
a_{\alpha, j}^{(j + 1)} =\ & C_s \binom{i_\star + 2 s}{i_\star} \sum_{\ell = 0}^j \binom{j}{\ell} \frac{(-1)^{j - \ell}}{j - \ell + 1}, \\
f_{\alpha, j}(x) =\ & - C_s \Bigg\{ \Bigg( \pv \int_0^2 \frac{t^\alpha \log^j \! t}{|1 - t|^{1 + 2 s}} \, dt + \sum_{i \in \left( \N \cup \{ 0 \} \right) \setminus \{ i_\star \}} \binom{i + 2 s}{i} 2^{\alpha - 2 s - i} \sum_{\ell = 0}^{j} \frac{(j)_\ell \log^{j - \ell} \! 2}{(2 s - \alpha + i)^{\ell + 1}}  \\
& -\binom{i_\star + 2 s}{i_\star} \frac{\log^{j + 1}  \! 2}{j + 1} \Bigg) x^{\alpha - 2 s} - j! \sum_{i \in \left( \N \cup \{ 0 \} \right) \setminus \{ i_\star \}} \binom{i + 2 s}{i} \frac{x^i}{(2 s - \alpha + i)^{j + 1}} + \frac{1}{2 s} \frac{\delta_{j 0}}{(1 - x)^{2 s}} \Bigg\},
\end{align*}
for every~$k = 1, \ldots, j - 1$. Note that~$a_{\alpha, j}^{(0)}$ is equal to zero since the corresponding term in the sum~\eqref{fraclapofuraz} has been incorporated into the smooth remainder~$f_{\alpha, j}$, as~$\alpha - 2s \in \N \cup \{ 0 \}$.

When~$\alpha - 2 s \notin \N \cup \{ 0 \}$, then no such~$i_\star$ exists. Nevertheless, the above computations are still valid, provided that the terms involving~$i_\star$ are neglected and the sums over~$\left( \N \cup \{ 0 \} \right) \setminus \{ i_\star \}$ are understood to be over the non-negative integers. The resulting values for the coefficients~$a_{\alpha, j}^{(k)}$'s and the function~$f_{\alpha, j}$ are
\begin{align*}
a_{\alpha, j}^{(k)} =\ & - C_s \binom{j}{k} \Bigg\{ {\pv \int_0^2 \frac{t^\alpha \log^{j - k} \! t}{|1 - t|^{1 + 2 s}} \, dt + \sum_{i = 0}^{\infty} \binom{i + 2 s}{i} 2^{\alpha - 2 s - i} \sum_{\ell = 0}^{j - k} \frac{(j - k)_\ell \log^{j - k - \ell} \! 2}{(2 s - \alpha + i)^{\ell + 1}}} \Bigg\}, \\
a_{\alpha, j}^{(j)} =\ & C_s \left\{ \frac{1}{s} + \pv \int_0^2 \frac{1 - t^\alpha}{|1 - t|^{1 + 2 s}} \, dt - \sum_{i = 0}^\infty \! \binom{i + 2 s}{i} \frac{2^{\alpha - 2 s - i}}{2 s - \alpha + i} \right\}, \\
a_{\alpha, j}^{(j + 1)} =\ & 0, \\
f_{\alpha, j}(x) =\ & C_s \left\{ j! \sum_{i = 0}^\infty \binom{i + 2 s}{i} \frac{x^i}{(2 s - \alpha + i)^{j + 1}} - \frac{1}{2 s} \frac{\delta_{j 0}}{(1 - x)^{2 s}} \right\},
\end{align*}
for every~$k = 1, \ldots, j - 1$.

Finally, the statement concerning the non-vanishing nature of the coefficients~$a_{1, 0}^{(0)}$ (when~$s \ne \frac{1}{2}$) and~$a_{1, 0}^{(1)}$ (when~$s = \frac{1}{2}$) follows from a direct inspection of identity~\eqref{techcompujk} when~$\alpha = 1$ and~$j = 0$. Indeed, in this case claim~\eqref{fraclapofuraz} holds true with~$a_{1, 0}^{(0)} = \frac{C_s}{2s (1 - 2 s)}$,~$a_{1, 0}^{(1)} = 0$, and~$f_{1, 0}(x) = - \frac{C_s}{2 s (1 - 2 s)} (1 - x)^{1 - 2 s}$ when~$s \ne \frac{1}{2}$ and with~$a_{1, 0}^{(0)} = 0$,~$a_{1, 0}^{(1)} = C_{\frac{1}{2}}$, and~$f_{1, 0}(x) = - C_{\frac{1}{2}} \log(1 - x)$ when~$s = \frac{1}{2}$.
\end{proof}

Thanks to this result, we may now proceed with the

\begin{proof}[Proof of Theorem~\ref{counterthm}]
We distinguish between the two cases~$s \notin \Q$ and~$s \in \Q$. In the first case, we let~$M \in \N \setminus \{ 1 \}$ and set
$$
v_M(x) \coloneqq \sum_{i = 0}^M b_i \, u_{2 (1 - s) i + 1, 0}(x),
$$
for some coefficients~$b_i \in \R$ to be determined. For~$x \in \left( 0, \frac{1}{2} \right)$, taking advantage of Lemma~\ref{buildingblocklem} we have that
\begin{align*}
& - v_M''(x) + (-\Delta)^s v_M(x) = \sum_{i = 0}^M b_i \Big\{ {- u_{2 (1 - s) i + 1, 0}''(x) + (-\Delta)^s u_{2 (1 - s) i + 1, 0}(x)} \Big\} \\
&\qquad = 
\sum_{i = 0}^M b_i \Big\{ {- 2 (1 - s) i (2 (1 - s) i + 1) x^{2 (1 - s) i - 1} + a_{2 (1 - s) i + 1, 0}^{(0)} \, x^{2 (1 - s) i + 1 - 2 s}  + f_{2 (1 - s) i + 1, 0}(x)} \Big\} \\
&\qquad = 
\sum_{i = 1}^{M} \Big\{ {a_{2 (1 - s) (i - 1) + 1, 0}^{(0)} \, b_{i - 1} - 2 (1 - s) i (2 (1 - s) i + 1) \, b_i} \Big\} x^{2 (1 - s) i - 1} \\
&\qquad \quad + a_{2 (1 - s) M + 1, 0}^{(0)} \, b_M \, x^{2 (1 - s) (M + 1) - 1} + \sum_{i = 0}^M b_i f_{2 (1 - s) i + 1, 0}(x).
\end{align*}
Here we used the fact that~$2 (1 - s) i + 1 - 2 s \notin \N \cup \{ 0 \}$ for every~$i \in \N \cup \{ 0 \}$, as~$s \notin \Q$. By choosing
$$
\begin{dcases}
b_0 = 1, & \\
b_i = \frac{a_{2 (1 - s) (i - 1) + 1, 0}^{(0)} \, b_{i - 1}}{2 (1 - s) i (2 (1 - s) i + 1)} & \quad \mbox{for } i \in \{ 1, \ldots, M \},
\end{dcases}
$$
we obtain that
$$
- v_M''(x) + (-\Delta)^s v_M(x) = a_{2 (1 - s) M + 1, 0} \, b_M \, x^{2 (1 - s) (M + 1) - 1} + \sum_{i = 0}^M b_i f_{2 (1 - s) i + 1, 0}(x) \qquad \mbox{for all } x \in \left( 0, \frac{1}{2} \right).
$$
Notice that the right-hand side of this equation belongs to~$C^k \! \left( \left[ 0, \frac{1}{2} \right] \right)$ if we take, say,~$M = M_1(k) \coloneqq \left\lceil \! \frac{k + 1}{2 (1 - s)} \! \right\rceil$. Observe that the function~$u_k \coloneqq v_{M_1(k)}$ thus constructed lies in~$C^{3 - 2 s} \! \left( \left[ 0, \frac{1}{2} \right] \right)$, but is not of class~$C^{3 - 2 s + \varepsilon}$ at~$0$ for any~$\varepsilon > 0$. Indeed,
\begin{equation} \label{irrexpan}
u_k(x) = x + \frac{a_{1, 0}^{(0)}}{2 (1 - s) (3 - 2 s)} \, x^{3 - 2 s} + o \! \left( x^{3 - 2 s} \right) \quad \mbox{as } x \rightarrow 0^+,
\end{equation}
and~$a_{1, 0}^{(0)} \ne 0$, thanks to Lemma~\ref{buildingblocklem}.

We now address the case of~$s \in \Q$. To handle it, we need a more refined construction. Let~$p, q \ge 1$ be the two unique coprime integers such that~$2 (1 - s) = \frac{p}{q}$. Then,~$2 (1 - s) i + 1 - 2 s \in \N \cup \{ 0 \}$ for some non-negative integer~$i$ if and only if~$i + 1 \in q \N$. Given any~$M \in \N \setminus \{ 1 \}$, we define
$$
w_M(x) \coloneqq \sum_{m = 0}^{M} \sum_{\ell = 1}^{q} \sum_{j = 0}^m b_{m, \ell, j} \, u_{2 (1 - s) (m q - 1 + \ell) + 1, j}(x),
$$
for some coefficients~$b_{m, \ell, j} \in \R$ to be determined. Lemma~\ref{buildingblocklem} yields that
\begin{align} 
- w_M''(x) + (-\Delta)^s w_M(x) & = \sum_{m = 0}^{M} \sum_{\ell = 1}^{q} \sum_{j = 0}^m b_{m, \ell, j} \left\{ - u_{2 (1 - s) (m q - 1 + \ell) + 1, j}''(x) + (-\Delta)^s u_{2 (1 - s) (m q - 1 + \ell) + 1, j}(x) \right\} \nonumber \\
& = A_1(x) + A_2(x) + A_3(x) + \sum_{m = 0}^{M} \sum_{\ell = 1}^q \sum_{j = 0}^m b_{m, \ell, j} \, f_{2 (1 - s) (m q - 1 + \ell) + 1, j}(x), \label{LwM}
\end{align}
for every~$x \in \left( 0, \frac{1}{2} \right)$, with
\begin{align*}
A_1(x) & \coloneqq \sum_{\ell = 1}^{q - 1} b_{0, \ell, 0} \left\{ a_{2 (1 - s) (\ell - 1) + 1, 0}^{(0)} \, x^{2 (1 - s) \ell - 1} - 2 (1 - s) (\ell - 1) \big( {2 (1 - s) (\ell - 1) + 1} \big) \, x^{2(1 - s)(\ell - 1) - 1} \right\} \\
& \quad\, + b_{0, q, 0} \left\{ a_{2 (1 - s) (q - 1) + 1, 0}^{(1)} \, x^{2 (1 - s) q - 1} \log x - 2 (1 - s) (q - 1) \big( {2 (1 - s) (q - 1) + 1} \big) x^{2(1 - s)(q - 1) - 1} \right\}, \\
A_2(x) & \coloneqq \sum_{m = 1}^{M} \sum_{\ell = 1}^{q - 1} \sum_{j = 0}^m b_{m, \ell, j} \left\{ \vphantom{\sum_{k = 0}^{j + 1} a_{2 (1 - s) (m q - 1 + \ell) + 1, j}^{(k)} \, x^{2 (1 - s) (m q - 1 + \ell) + 1 - 2 s} \log^k \! x} {- x^{2(1 - s)(m q - 1 + \ell) - 1} \Big( {2 (1 - s) (m q - 1 + \ell) \big( {2 (1 - s) (m q - 1 + \ell) + 1} \big) \log^j \! x}} \right. \\
& \quad\, + {\big( {4 (1 - s) (m q - 1 + \ell) + 1} \big) j \log^{j - 1} \! x + j (j - 1) \log^{j - 2} \! x} \Big)  \\
& \quad\, \left. + \, x^{2 (1 - s) (m q + \ell) - 1} \sum_{k = 0}^{j} a_{2 (1 - s) (m q - 1 + \ell) + 1, j}^{(k)} \log^k \! x \right\}, \\
A_3(x) & \coloneqq \sum_{m = 1}^{M} \sum_{j = 0}^m b_{m, q, j} \left\{ \vphantom{\sum_{k = 0}^{j + 1} a_{2 (1 - s) ( (m + 1) q - 1) + 1, j}^{(k)} \, x^{2 (1 - s) ( (m + 1) q - 1) + 1 - 2 s} \log^k \! x} {- x^{2(1 - s)((m + 1) q - 1) - 1} \Big( {2 (1 - s) \big( {(m + 1) q - 1} \big) \big( {2 (1 - s) \big( {(m + 1) q - 1} \big) + 1} \big) \log^j \! x}} \right. \\
& \quad\, + {\big( {4 (1 - s) \big( {(m + 1) q - 1} \big) + 1} \big) j \log^{j - 1} \! x + j (j - 1) \log^{j - 2} \! x} \Big)  \\
& \quad\, \left. + \, x^{2 (1 - s) (m + 1) q - 1} \sum_{k = 1}^{j + 1} a_{2 (1 - s) ( (m + 1) q - 1) + 1, j}^{(k)} \, \log^k \! x \right\},
\end{align*}
for~$m \in \{ 0, \ldots, M \}$.

We begin by analyzing~$A_1$, which can be dealt with similarly to what we did in the case of irrational~$s$. By splitting it into two sums and shifting indices, we have
\begin{equation} \label{A1simplified}
\begin{aligned}
A_1(x) =\ & 
\sum_{\ell = 1}^{q - 1} \Big\{ {a_{2 (1 - s) (\ell - 1) + 1, 0}^{(0)} \, b_{0, \ell, 0} - 2 (1 - s) \ell \big(2 (1 - s) \ell + 1\big) \, b_{0, \ell + 1, 0}} \Big\} x^{2 (1 - s) \ell - 1} \\
& 
+ a_{2 (1 - s) (q - 1) + 1, 0}^{(1)} \, b_{0, q, 0} \, x^{2 (1 - s) q - 1} \log x \\
=\ &
a_{2 (1 - s) (q - 1) + 1, 0}^{(1)} \, b_{0, q, 0} \, x^{2 (1 - s) q - 1} \log x,
\end{aligned}
\end{equation}
provided we choose the~$b_{0, \ell, 0}$'s recursively as follows:
$$
\begin{dcases}
b_{0, 1, 0} = 1, & \\
b_{0, \ell+1, 0} = \frac{a_{2 (1 - s) (\ell - 1) + 1, 0}^{(0)} \, b_{0, \ell, 0}}{2 (1 - s) \ell \big( {2 (1 - s) \ell + 1} \big)} & \quad \mbox{for } \ell \in \{ 1, \ldots, q-1 \}.
\end{dcases}
$$

We then move to the~$A_2$ term. Our goal is to rearrange it and factor out the different terms of the form~$x^{2 (1 - s) (m q - 1 + \ell) - 1} \log^j \! x$---much like what we just did for~$A_1$, but now with logarithms involved as well. After some tedious computations involving shifts in the sum indices and exchanges of the orders of summation, we find that
$$
A_2(x) = \sum_{m = 1}^M \sum_{\ell = 1}^{q} \sum_{j = 0}^m C_{m, \ell, j} \, x^{2 (1 - s) (m q - 1 + \ell) - 1} \log^j \! x,
$$
with
\begin{align*}
C_{m, 1, j} \coloneqq\ & 
- 2 (1 - s) \big( {2 (1 - s) m q + 1} \big) m q \, b_{m, 1, j} - \big( {4 (1 - s) m q + 1} \big) (j + 1) b_{m, 1, j + 1} \\
& 
- (j + 1) (j + 2) b_{m, 1, j + 2}  \qquad \mbox{for } j \in \{ 0, \ldots, m - 2 \}, \\
C_{m, 1, m - 1} \coloneqq\ & 
- m \Big( {2 (1 - s) \big( {2 (1 - s) m q + 1} \big) q \, b_{m, 1, m - 1} + \big( {4 (1 - s) m q + 1} \big) \, b_{m, 1, m}} \Big), \\
C_{m, 1, m} \coloneqq\ & 
- 2 (1 - s) \big( {2 (1 - s) m q + 1} \big) m q \, b_{m, 1, m}, \\
C_{m, \ell, j} \coloneqq\ & 
\sum_{k = j}^m a_{2 (1 - s) (m q - 2 + \ell) + 1, k}^{(j)} \, b_{m, \ell - 1, k} - 2 (1 - s) (m q - 1 + \ell) \big( {2 (1 - s) (m q - 1 + \ell) + 1} \big) b_{m, \ell, j} \\
& 
- \big( {4 (1 - s) (m q - 1 + \ell) + 1} \big) (j + 1) b_{m, \ell, j + 1} - (j + 1) (j + 2) b_{m, \ell, j + 2} \quad \mbox{for } j \in \{0, \ldots, m - 2\}, \\
C_{m, \ell, m - 1} \coloneqq\ & 
\sum_{k = m - 1}^m a_{2 (1 - s) (m q - 2 + \ell) + 1, k}^{(m - 1)} \, b_{m, \ell - 1, k} - 2 (1 - s) (m q - 1 + \ell) \big( {2 (1 - s) (m q - 1 + \ell) + 1} \big) b_{m, \ell, m - 1} \\
& - \big( {4 (1 - s) (m q - 1 + \ell) + 1} \big) m \, b_{m, \ell, m}, \\
C_{m, \ell, m} \coloneqq\ & 
a_{2 (1 - s) (m q - 2 + \ell) + 1, m}^{(m)} \, b_{m, \ell - 1, m} - 2 (1 - s) (m q - 1 + \ell) \big( {2 (1 - s) (m q - 1 + \ell) + 1} \big) b_{m, \ell, m}, \\
C_{m, q, j} \coloneqq\ & 
\sum_{k = j}^m a_{2 (1 - s) ( (m + 1) q - 2) + 1, k}^{(j)} \, b_{m, q - 1, k} \qquad \mbox{for } j \in \{0, \ldots, m\},
\end{align*}
for~$m \in \{ 1, \ldots, M \}$ and~$\ell \in \{ 2, \ldots, q - 1 \}$. We make the majority of these coefficients vanish by choosing
\begin{align*}
b_{m, \ell, m} & = \frac{a_{2 (1 - s) (m q - 2 + \ell) + 1, m}^{(m)} \, b_{m, \ell - 1, m}}{2 (1 - s) (m q - 1 + \ell) \big( {2 (1 - s) (m q - 1 + \ell) + 1} \big)}, \\
b_{m, \ell, m - 1} & = \frac{\sum\limits_{k = m - 1}^m a_{2 (1 - s) (m q - 2 + \ell) + 1, k}^{(m - 1)} \, b_{m, \ell - 1, k} - \big( {4 (1 - s) (m q - 1 + \ell) + 1} \big) m \, b_{m, \ell, m}}{2 (1 - s) (m q - 1 + \ell) \big( {2 (1 - s) (m q - 1 + \ell) + 1} \big)}, \\
b_{m, \ell, j} & = \frac{\sum\limits_{k = j}^m a_{2 (1 - s) (m q - 2 + \ell) + 1, k}^{(j)} \, b_{m, \ell - 1, k}}{2 (1 - s) (m q - 1 + \ell) \big( {2 (1 - s) (m q - 1 + \ell) + 1} \big)} \\
& \quad - \frac{\big( {4 (1 - s) (m q - 1 + \ell) + 1} \big) (j + 1) b_{m, \ell, j + 1} + (j + 1) (j + 2) b_{m, \ell, j + 2}}{2 (1 - s) (m q - 1 + \ell) \big( {2 (1 - s) (m q - 1 + \ell) + 1} \big)},
\end{align*}
for~$m \in \{ 1, \ldots, M\}$,~$\ell \in \{2, \ldots, q - 1\}$, and~$j \in \{ 0, \ldots, m \}$. This leaves us with
\begin{equation} \label{A2simplified}
A_2(x) = \sum_{m = 1}^M \sum_{j = 0}^m C_{m, 1, j} \, x^{2 (1 - s) m q - 1} \log^j \! x + \sum_{m = 1}^M \sum_{j = 0}^m C_{m, q, j} \, x^{2 (1 - s) ((m + 1) q - 1) - 1} \log^j \! x,
\end{equation}
and~$b_{m, 1, j}$,~$b_{m, q, j}$ still free to choose, for~$m \in \{ 1, \ldots, M \}$ and~$j \in \{ 0, \ldots, m \}$.

In order to get rid of these terms, we now inspect~$A_3$. By rearranging the logarithmic terms, we write it as
$$
A_3(x) = \sum_{m = 1}^M \sum_{j = 0}^m D_{m, q, j} \, x^{2 (1 - s) ((m + 1) q - 1) - 1} \log^j \! x + \sum_{m = 1}^M \sum_{j = 1}^{m + 1} D_{m, q + 1, j} \, x^{2 (1 - s) (m + 1) q - 1} \log^{j} \! x,
$$
with
\begin{align*}
D_{m, q, j} & \coloneqq - 2(1 - s) \big( {(m + 1) q - 1} \big) \big( {2(1 - s) \big( {(m + 1) q - 1} \big) + 1} \big) \, b_{m, q, j} \\
& \quad\,\,\,\, {- \big( {4 (1 - s) \big( {(m + 1) q - 1} \big) + 1} \big)(j + 1) \, b_{m, q, j + 1}} \\
& \quad\,\,\,\, {- (j + 1)(j + 2) \, b_{m, q, j + 2}} \quad \mbox{for } j \in \{ 0, \ldots, m - 2 \}, \\
D_{m, q, m - 1} & \coloneqq - 2(1 - s) \big( {(m + 1) q - 1} \big) \big( {2(1 - s) \big( {(m + 1) q - 1} \big) + 1} \big) \, b_{m, q, m - 1} \\
& \quad\,\,\,\, {- \big( {4 (1 - s) \big( {(m + 1) q - 1} \big) + 1} \big) m \, b_{m, q, m}}, \\
D_{m, q, m} & \coloneqq - 2(1 - s) \big( {(m + 1) q - 1} \big) \big( {2(1 - s) \big( {(m + 1) q - 1} \big) + 1} \big) \, b_{m, q, m}, \\
D_{m, q + 1, j} & \coloneqq \sum_{k = j - 1}^m a_{2 (1 - s) ((m + 1) q - 1) + 1, k}^{(j)} \, b_{m, q, k} \quad \mbox{for } j \in \{ 1, \ldots, m + 1\},
\end{align*}
for~$m \in \{1, \ldots, M\}$. As a result, recalling~\eqref{LwM},~\eqref{A1simplified}, and~\eqref{A2simplified}, we have
\begin{align*}
& - w_M''(x) + (-\Delta)^s w_M(x) = \\
& = \Big( {a_{2 (1 - s) (q - 1) + 1, 0}^{(1)} \, b_{0, q, 0} + C_{1, 1, 1}} \Big) \, x^{2 (1 - s) q - 1} \log x \\
& \quad + \sum_{m = 1}^M C_{m, 1, 0} \, x^{2 (1 - s) m q - 1} + \sum_{m = 2}^M \sum_{j = 1}^m \Big( {C_{m, 1, j} + D_{m - 1, q + 1, j}} \Big) \, x^{2 (1 - s) m q - 1} \log^j \! x \\
& \quad + \sum_{m = 1}^M \sum_{j = 0}^m \Big( {C_{m, q, j} + D_{m, q, j}} \Big) \, x^{2 (1 - s) ((m + 1) q - 1) - 1} \log^j \! x \\
& \quad + \sum_{j = 1}^{M + 1} D_{M, q + 1, j} \, x^{2 (1 - s) (M + 1) q - 1} \log^{j} \! x + \sum_{m = 0}^{M} \sum_{\ell = 1}^q \sum_{j = 0}^m b_{m, \ell, j} \, f_{2 (1 - s) (m q - 1 + \ell) + 1, j}(x).
\end{align*}

By setting recursively
\begin{align*}
b_{1, 1, 1} & = \frac{a_{2 (1 - s) (q - 1) + 1, 0}^{(1)} \, b_{0, q, 0}}{2 (1 - s) \big( {2 (1 - s) q + 1} \big) q}, \qquad b_{1, 1, 0} = - \frac{\big( {4 (1 - s) q + 1} \big) \, b_{1, 1, 1}}{2 (1 - s) \big( {2 (1 - s) q + 1} \big) q}, \\
b_{1, q, 1} & = \frac{a_{4 (1 - s) (q - 1) + 1, 1}^{(1)} \, b_{1, q - 1, 1}}{2(1 - s) \big( {2 q - 1} \big) \big( {2(1 - s) \big( {2 q - 1} \big) + 1} \big)}, \\
b_{1, q, 0} & = \frac{- \big( {4 (1 - s) \big( {2 q - 1} \big) + 1} \big) \, b_{1, q, 1} + a_{4 (1 - s) (q - 1) + 1, 0}^{(0)} \, b_{1, q - 1, 0} + a_{4 (1 - s) (q - 1) + 1, 1}^{(0)} \, b_{1, q - 1, 1}}{2(1 - s) \big( {2 q - 1} \big) \big( {2(1 - s) \big( {2 q - 1} \big) + 1} \big)}, \\
b_{m, 1, m} & = \frac{a_{2 (1 - s) (m q - 1) + 1, m - 1}^{(m)} \, b_{m - 1, q, m - 1}}{2 (1 - s) \big( {2 (1 - s) m q + 1} \big) m q}, \\
b_{m, 1, m - 1} & = \frac{- \big( {4 (1 - s) m q + 1} \big) m \, b_{m, 1, m} + \sum\limits_{k = m - 2}^{m - 1} a_{2 (1 - s) (m q - 1) + 1, k}^{(m - 1)} \, b_{m - 1, q, k}}{2 (1 - s) \big( {2 (1 - s) m q + 1} \big) m q}, \\
b_{m, 1, j} & = \frac{- \big( {4 (1 - s) m q + 1} \big) (j + 1) \, b_{m, 1, j + 1} - (j + 1) (j + 2) \, b_{m, 1, j + 2} + \sum\limits_{k = j - 1}^{m - 1} a_{2 (1 - s) (m q - 1) + 1, k}^{(j)} \, b_{m - 1, q, k}}{2 (1 - s) \big( {2 (1 - s) m q + 1} \big) m q}, \\
b_{m, 1, 0} & = - \frac{\big( {4 (1 - s) m q + 1} \big) b_{m, 1, 1} + 2 b_{m, 1, 2}}{2 (1 - s) \big( {2 (1 - s) m q + 1} \big) m q}, \\
b_{m, q, m} & = \frac{a_{2 (1 - s) ( (m + 1) q - 2) + 1, m}^{(m)} \, b_{m, q - 1, m}}{2(1 - s) \big( {(m + 1) q - 1} \big) \big( {2(1 - s) \big( {(m + 1) q - 1} \big) + 1} \big)}, \\
b_{m, q, m - 1} & = \frac{- \big( {4 (1 - s) \big( {(m + 1) q - 1} \big) + 1} \big) m \, b_{m, q, m} + \sum\limits_{k = m - 1}^m a_{2 (1 - s) ( (m + 1) q - 2) + 1, k}^{(m - 1)} \, b_{m, q - 1, k}}{2(1 - s) \big( {(m + 1) q - 1} \big) \big( {2(1 - s) \big( {(m + 1) q - 1} \big) + 1} \big)}, \\
b_{m, q, j} & = - \frac{\big( {4 (1 - s) \big( {(m + 1) q - 1} \big) + 1} \big)(j + 1) \, b_{m, q, j + 1} + (j + 1)(j + 2) \, b_{m, q, j + 2}}{2(1 - s) \big( {(m + 1) q - 1} \big) \big( {2(1 - s) \big( {(m + 1) q - 1} \big) + 1} \big)} \\
& \quad + \frac{\sum\limits_{k = j}^m a_{2 (1 - s) ( (m + 1) q - 2) + 1, k}^{(j)} \, b_{m, q - 1, k}}{2(1 - s) \big( {(m + 1) q - 1} \big) \big( {2(1 - s) \big( {(m + 1) q - 1} \big) + 1} \big)}, \\
b_{m, q, 0} & = \frac{- \big( {4 (1 - s) \big( {(m + 1) q - 1} \big) + 1} \big) \, b_{m, q, 1} - 2 \, b_{m, q, 2} + \sum\limits_{k = 0}^m a_{2 (1 - s) ( (m + 1) q - 2) + 1, k}^{(0)} \, b_{m, q - 1, k}}{2(1 - s) \big( {(m + 1) q - 1} \big) \big( {2(1 - s) \big( {(m + 1) q - 1} \big) + 1} \big)},
\end{align*}
for~$m \in \{2, \ldots, M\}$ and~$j \in \{ 1, \ldots, m - 2 \}$, the previous expression further simplifies to
\begin{align*}
- w_M''(x) + (-\Delta)^s w_M(x) & = \sum_{j = 1}^{M + 1} D_{M, q + 1, j} \, x^{2 (1 - s) (M + 1) q - 1} \log^{j} \! x + \sum_{m = 0}^{M} \sum_{\ell = 1}^q \sum_{j = 0}^m b_{m, \ell, j} \, f_{2 (1 - s) (m q - 1 + \ell) + 1, j}(x).
\end{align*}
If we now take~$M = M_2(k) \coloneqq \left\lceil \frac{k + 1}{2 (1 - s) q} \right\rceil$, the right-hand side of the above identity belongs to~$C^k \! \left( \left[ 0, \frac{1}{2 }\right] \right)$. Moreover, the function~$u_k \coloneqq w_{M_2(k)}$ just constructed has the regularity claimed in the statement of the theorem. Indeed, if~$s \ne \frac{1}{2}$, then~$q \ge 2$ and it is therefore easy to see that the expansion~\eqref{irrexpan} holds true. If, on the other hand,~$s = \frac{1}{2}$, then~$q = 1$ and we have
$$
u_k(x) = x + \frac{a_{1, 0}^{(1)}}{2} \, x^2 \log x + O \! \left( x^2 \right) \quad \mbox{as } x \rightarrow 0^+,
$$
with~$a_{1, 0}^{(1)} \ne 0$ in view of Lemma~\ref{buildingblocklem}. The proof is thus complete.
\end{proof}

\section*{Acknowledgments}
We thank the anonymous referees for their attentive reading and keen remarks, which substantially contributed to the improvement of the paper.

The research of the authors has been supported by the GNAMPA-INdAM project ``Equazioni nonlocali di tipo misto e geometrico'' (Italy), CUP E53C22001930001. 
NA is also partially supported by the PRIN project 2022R537CS ``$NO^3$ - NOdal Optimization, NOnlinear elliptic equations, NOnlocal geometric problems, with a focus on regularity'' (Italy), while~MC by the PRIN project 20229M52AS\_004 ``Partial differential equations and related geometric-functional inequalities'' (Italy) and by the Spanish grants PID2021-123903NB-I00 and RED2022-134784-T funded by MCIN/AEI/10.13039/501100011033 (Spain) and by ERDF ``A way of making Europe'' (European Union).
The research of EA and NA is also partially supported by the AAP-BQRI Grant ``BORDER'' funded by the University of Rouen Normandie (France).


\begin{bibdiv}
\begin{biblist}

\bib{AC21}{article}{
   author={Abatangelo, Nicola},
   author={Cozzi, Matteo},
   title={An elliptic boundary value problem with fractional nonlinearity},
   journal={SIAM J. Math. Anal.},
   volume={53},
   date={2021},
   number={3},
   pages={3577--3601},
}

\bib{libbro}{book}{
   author={Abatangelo, Nicola},
   author={Dipierro, Serena},
   author={Valdinoci, Enrico},
   title={A gentle introduction to the fractional world},
   series={Unitext},
   volume={176},
   publisher={Springer, Cham},
   year={2025},
   pages={x+257},
}

\bib{MR4038144}{article}{
   author={Abatangelo, Nicola},
   author={Ros-Oton, Xavier},
   title={Obstacle problems for integro-differential operators: higher
   regularity of free boundaries},
   journal={Adv. Math.},
   volume={360},
   date={2020},
   pages={106931, 61 pp},
}

\bib{MR3967804}{article}{
   author={Abatangelo, Nicola},
   author={Valdinoci, Enrico},
   title={Getting acquainted with the fractional Laplacian},
   conference={
      title={Contemporary research in elliptic PDEs and related topics},
   },
   book={
      series={Springer INdAM Ser.},
      volume={33},
      publisher={Springer, Cham},
   },
   isbn={978-3-030-18920-4},
   isbn={978-3-030-18921-1},
   date={2019},
   pages={1--105},
}

\bib{AC23}{article}{
   author={Antonini, Carlo Alberto},
   author={Cozzi, Matteo},
   title={Global gradient regularity and a Hopf lemma for quasilinear
   operators of mixed local-nonlocal type},
   journal={J. Differential Equations},
   volume={425},
   date={2025},
   pages={342--382},
}

\bib{as}{book}{
   author={Abramowitz, Milton},
   author={Stegun, Irene A.},
   title={Handbook of mathematical functions with formulas, graphs, and
   mathematical tables},
   series={National Bureau of Standards Applied Mathematics Series},
   volume={No. 55},
   publisher={U. S. Government Printing Office, Washington, DC},
   date={1964},
   pages={xiv+1046},
}

\bib{MR2735074}{article}{
   author={Barles, Guy},
   author={Chasseigne, Emmanuel},
   author={Imbert, Cyril},
   title={H\"older continuity of solutions of second-order non-linear
   elliptic integro-differential equations},
   journal={J. Eur. Math. Soc. (JEMS)},
   volume={13},
   date={2011},
   number={1},
   pages={1--26},
}

\bib{BDVV22}{article}{
   author={Biagi, Stefano},
   author={Dipierro, Serena},
   author={Valdinoci, Enrico},
   author={Vecchi, Eugenio},
   title={Mixed local and nonlocal elliptic operators: regularity and
   maximum principles},
   journal={Comm. Partial Differential Equations},
   volume={47},
   date={2022},
   number={3},
   pages={585--629},
}

\bib{BDVV23}{article}{
   author={Biagi, Stefano},
   author={Dipierro, Serena},
   author={Valdinoci, Enrico},
   author={Vecchi, Eugenio},
   title={A Faber-Krahn inequality for mixed local and nonlocal operators},
   journal={J. Anal. Math.},
   volume={150},
   date={2023},
   number={2},
   pages={405--448},
}

\bib{BM21}{article}{
   author={Biswas, Anup},
   author={Modasiya, Mitesh},
   title={Mixed local-nonlocal operators: maximum principles, eigenvalue problems and their applications},
   journal={J. Anal. Math.},
   volume={156},
   date={2025},
   number={1},
   pages={47--81},
}

\bib{BMS23}{article}{
   author={Biswas, Anup},
   author={Modasiya, Mitesh},
   author={Sen, Abhrojyoti},
   title={Boundary regularity of mixed local-nonlocal operators and its
   application},
   journal={Ann. Mat. Pura Appl. (4)},
   volume={202},
   date={2023},
   number={2},
   pages={679--710},
}

\bib{MR3469920}{book}{
   author={Bucur, Claudia},
   author={Valdinoci, Enrico},
   title={Nonlocal diffusion and applications},
   series={Lecture Notes of the Unione Matematica Italiana},
   volume={20},
   publisher={Springer, [Cham]; Unione Matematica Italiana, Bologna},
   date={2016},
   pages={xii+155},
}

\bib{MR2944369}{article}{
   author={Di Nezza, Eleonora},
   author={Palatucci, Giampiero},
   author={Valdinoci, Enrico},
   title={Hitchhiker's guide to the fractional Sobolev spaces},
   journal={Bull. Sci. Math.},
   volume={136},
   date={2012},
   number={5},
   pages={521--573},
}

\bib{dyda}{article}{
   author={Dyda, Bart\l omiej},
   title={Fractional calculus for power functions and eigenvalues of the
   fractional Laplacian},
   journal={Fract. Calc. Appl. Anal.},
   volume={15},
   date={2012},
   number={4},
   pages={536--555},
}

\bib{MR3916700}{article}{
   author={Garofalo, Nicola},
   title={Fractional thoughts},
   conference={
      title={New developments in the analysis of nonlocal operators},
   },
   book={
      series={Contemp. Math.},
      volume={723},
      publisher={Amer. Math. Soc., [Providence], RI},
   },
   isbn={978-1-4704-4110-4},
   date={2019},
   pages={135},
}

\bib{GT01}{book}{
   author={Gilbarg, David},
   author={Trudinger, Neil S.},
   title={Elliptic partial differential equations of second order},
   series={Classics in Mathematics},
   note={Reprint of the 1998 edition},
   publisher={Springer-Verlag, Berlin},
   date={2001},
   pages={xiv+517},
}

\bib{L95}{book}{
   author={Lunardi, Alessandra},
   title={Analytic semigroups and optimal regularity in parabolic problems},
   series={Progress in Nonlinear Differential Equations and their
   Applications},
   volume={16},
   publisher={Birkh\"auser Verlag, Basel},
   date={1995},
   pages={xviii+424},
}

\bib{RS14}{article}{
   author={Ros-Oton, Xavier},
   author={Serra, Joaquim},
   title={The Dirichlet problem for the fractional Laplacian: regularity up
   to the boundary},
   journal={J. Math. Pures Appl. (9)},
   volume={101},
   date={2014},
   number={3},
   pages={275--302},
}

\bib{SVWZ22}{article}{
   author={Su, Xifeng},
   author={Valdinoci, Enrico},
   author={Wei, Yuanhong},
   author={Zhang, Jiwen},
   title={Regularity results for solutions of mixed local and nonlocal
   elliptic equations},
   journal={Math. Z.},
   volume={302},
   date={2022},
   number={3},
   pages={1855--1878},
}

\bib{SVWZ23}{article}{
   author={Su, Xifeng},
   author={Valdinoci, Enrico},
   author={Wei, Yuanhong},
   author={Zhang, Jiwen},
   title={Multiple solutions for mixed local and nonlocal elliptic
   equations},
   journal={Math. Z.},
   volume={308},
   date={2024},
   number={3},
   pages={Paper No. 40, 37},
}

\bib{SVWZ25}{article}{
   author={Su, Xifeng},
   author={Valdinoci, Enrico},
   author={Wei, Yuanhong},
   author={Zhang, Jiwen},
   title={On some regularity properties of mixed local and nonlocal elliptic
   equations},
   journal={J. Differential Equations},
   volume={416},
   date={2025},
   pages={576--613},
}

\end{biblist}
\end{bibdiv}
\end{document}